\documentclass{amsart}
\usepackage{amssymb,amsmath,amsthm,amscd}
\usepackage{leftidx}
\usepackage{graphicx}
\usepackage{paralist}
\usepackage{color}
\usepackage{hyperref}
\hypersetup{urlcolor=blue, citecolor=red}
\usepackage[usenames,dvipsnames,svgnames,table]{xcolor}
\usepackage{tikz}
\usepackage{comment}
\usepackage{url}
\usepackage{epstopdf}
\usepackage{float}
\usepackage{color}
\usepackage{blindtext}

\usepackage[pagewise]{lineno}

\newtheorem{theorem}{Theorem}[section]
\newtheorem{corollary}{Corollary}

\newtheorem{lemma}[theorem]{Lemma}
\newtheorem{proposition}[theorem]{Proposition}

\theoremstyle{definition}
\newtheorem{definition}[theorem]{Definition}
\newtheorem{remark}{Remark}

\newcommand{\C}{\mathbb{C}}

\newcommand{\N}{\mathbb{N}}

\newcommand{\R}{\mathbb{R}}
\newcommand{\sphere}{\mathbb{S}^2}
\newcommand{\Z}{\mathbb{Z}}

\newcommand{\cC}{\mathcal{C}}

\newcommand{\cU}{\mathcal{U}}

\DeclareMathOperator{\re}{Re}
\DeclareMathOperator{\im}{Im}

\renewcommand{\epsilon}{\varepsilon}
\renewcommand{\phi}{\varphi}
\renewcommand{\theta}{\vartheta}

\title[Invisible tricorns]{Invisible tricorns in real slices of rational maps}

\begin{document}

\author[R. Lodge]{Russell Lodge}
\address{Department of Mathematics and Computer Science, Indiana State University, Terre Haute, IN 47809, USA}
\email{russell.lodge@indstate.edu}

\author[S. Mukherjee]{Sabyasachi Mukherjee}
\address{School of Mathematics, Tata Institute of Fundamental Research, 1 Homi Bhabha Road, Mumbai 400005, India}
\email{sabya@math.tifr.res.in}

\maketitle

\begin{abstract}
One of the conspicuous features of real slices of bicritical rational maps is the existence of Tricorn-type hyperbolic components. Such a hyperbolic component is called \emph{invisible} if the non-bifurcating sub-arcs on its boundary do not intersect the closure of any other hyperbolic component. Numerical evidence suggests an abundance of invisible Tricorn-type components in real slices of bicritical rational maps. In this paper, we study two different families of real bicritical maps and characterize invisible Tricorn-type components in terms of suitable topological properties in the dynamical planes of the representative maps. We use this criterion to prove the existence of infinitely many invisible Tricorn-type components in the corresponding parameter spaces. Although we write the proofs for two specific families, our methods apply to generic families of real bicritical maps.
\end{abstract}

\tableofcontents

\section{Introduction}
In \cite{M4}, Milnor studied the dynamics and parameter spaces of rational maps with two critical points (we will call such maps `strictly bicritical'). A rational map is called \emph{real} if it commutes with an antiholomorphic involution of the Riemann sphere $\hat{\C}$. In suitable regions of parameter spaces, the two critical orbits of a strictly bicritical real map are related by an antiholomorphic involution. The dynamics of such maps are reminiscent of unicritical antiholomorphic polynomials and their parameter spaces display Tricorn-like geometry \cite[\S 5]{M4}.

Note that up to a M{\"o}bius change of coordinates, an antiholomorphic involution of $\hat{\C}$ can be written either as the complex conjugation map $\iota(z)=\overline{z}$ or as the antipodal map $\eta(z)=-\frac{1}{\overline{z}}$. The former has a circle of fixed points and the latter has no fixed point. Hence, a real rational map is (in suitable coordinates) either a map with real coefficients or an antipode-preserving map (i.e. it sends pairs of antipodal points to pairs of antipodal points). By a theorem of Borsuk and Hopf, the latter possibility can only be realized by rational maps of odd degree.

More generally, a family of rational maps with only two \emph{free} critical points exhibits many features similar to those of strictly bicritical rational maps. As an abuse of terminology, we will refer to such families as real bicritical families. In this article, we will study the Tricorn-like geometry and its topological consequences for real bicritical families of rational maps such that the two free critical orbits are related by an antiholomorphic involution.

A natural example of a real bicritical family commuting with $\iota$ is given by degree $4$ real Newton maps
$$N_a: \hat{\mathbb{C}}\to \hat{\mathbb{C}}$$
$$N_a(z)=z-\frac{f_a(z)}{f_a'(z)}$$
corresponding to the polynomials $f_a(z)=(z-1)(z+1)(z-a)(z-\bar{a})$, $a\in\cU$, where $\cU$ is the set of all parameters in $\mathbb{H}$ such that the two non-fixed critical points of $N_a$ are complex conjugate. We will call this family
\[\mathcal{N}_4^*=\{N_a\ |\ a\in \mathcal{U}\}.\]

The choice and parametrization of the family deserve some explanation (this family was also considered in \cite{Su}). Since a Newton map of any degree $d\geq 2$ has a repelling fixed point at $\infty$ (i.e. $\infty$ is a marked point), one can always send two fixed points of a Newton map (i.e. two roots of the corresponding polynomial) to $1$ and $-1$ by an affine conjugacy. The other $d-2$ fixed points parametrize the family of all Newton maps of degree $d$, so it is a complex $(d-2)$-dimensional family. In particular, the parameter space of all Newton maps of degree four is complex $2$-dimensional. We have chosen the real slice $\mathbb{H}$ of degree four Newton maps so that the other two fixed points are complex conjugates of each other.\footnote{For $a\in\mathbb{R}$, the fixed points $a$ and $\bar{a}$ are real, so they are not complex conjugates of each other. Moreover, for any $a\in\mathbb{C}$, we have that $N_a=N_{\overline{a}}$. Hence it suffices to restrict attention to the upper-half plane.} Note that two maps $N_a$ and $N_b$ in the family $\mathcal{N}_4^*$ with $a\neq b$ are holomorphically (affinely) conjugate if and only if $b=-\bar{a}$. The conjugating map between $N_a$ and $N_{-\bar{a}}$ is $z\mapsto -z$.

\begin{figure}[ht!]
\begin{center}
\includegraphics[scale=0.4]{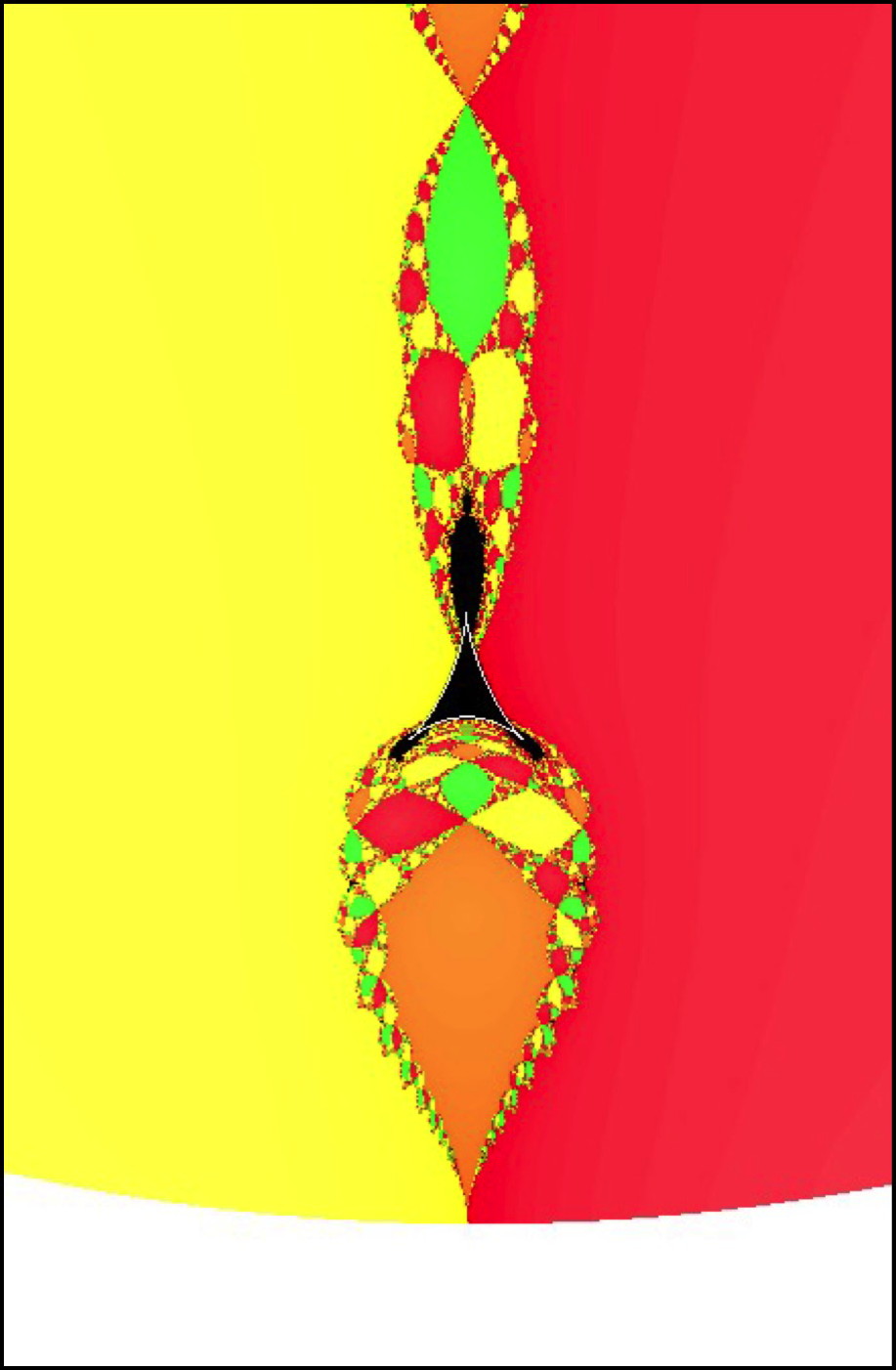}
\caption{A part of the parameter plane of the family $\mathcal{N}_4^*$ that contains a Tricorn component (enclosed by white curves) with two visible and one invisible boundary arcs. The white region at the bottom represents the complement of $\cU$.}
\label{parameter_plane}
\end{center}
\end{figure}

Since $f_a$ is a real polynomial, $N_a$ is a rational map with real coefficients and $N_a\circ\iota=\iota\circ N_a$. For $a\in\cU$, all the zeroes of the polynomial $f_a$ are distinct and hence $1, -1, a, \bar{a}$ are super-attracting fixed points of $N_a$. In particular, they are critical points of $N_a$. Since $N_a$ is a degree $4$ rational map, it has $2\cdot4-2=6$ critical points in $\hat{\mathbb{C}}$, four of which are $1, -1, a, \bar{a}$ as mentioned above. The other two critical points of $N_a$ are:

\begin{equation*}
\frac{3(a+\bar{a})\pm\sqrt{9(a^2+\bar{a}^2)-6\vert a\vert^2+24}}{12}.
\end{equation*}
Note that the two ``free'' critical points\footnote{Here, the word ``free'' is used in an informal sense to indicate the fact that these two critical points can exhibit various different dynamical behavior as opposed to the other four critical points $1, -1, a, \bar{a}$, which are necessarily fixed by the dynamics.} are complex conjugate precisely when $9(a^2+\bar{a}^2)-6\vert a\vert^2+24<0$; i.e. $2\im(a)^2-\re(a)^2-2>0$. Hence,
\[\cU=\{a\in\mathbb{C}:2\im(a)^2-\re(a)^2-2>0\}\] (see Figure \ref{conjugate}). The fact that the two free critical orbits of all maps in $\mathcal{N}_4^*$ are related by the antiholomorphic involution $\iota$ is going to play a pivotal role in our investigation.

\begin{figure}[ht]
\begin{center}
\includegraphics[scale=0.42]{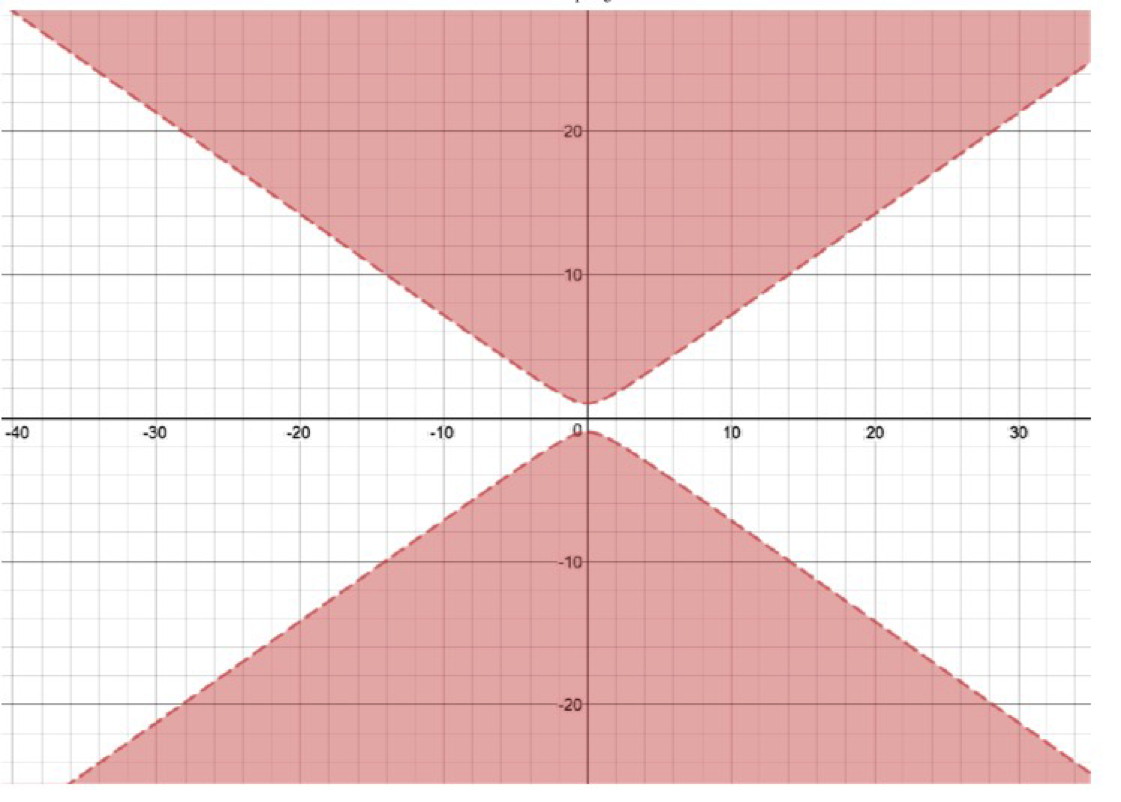}
\caption{For every parameter in the shaded region, the two free critical points of $N_a$ are complex conjugate. The upper red component is denoted by $\cU$.}
\label{conjugate}
\end{center}
\end{figure}

The second family of real bicritical rational maps that we will investigate in this paper is the family of antipode-preserving cubic rational maps. More precisely, we will consider the maps

\begin{align*}
f_q: \hat{\mathbb{C}}\to \hat{\mathbb{C}}\\
f_q(z)=z^2\frac{q-z}{1+\overline{q}z}
\end{align*}
for $q\in\C$.

Since $f_q\circ \eta=\eta\circ f_q$, each $f_q$ is a real rational map. Moreover, the critical points $0$ and $\infty$ of $f_q$ are super-attracting fixed points. It follows that the family $$\mathcal{A}_3:=\{f_q:q\in\C\}$$ is bicritical. The family $\mathcal{A}_3$ was introduced by Bonifant, Buff, and Milnor in \cite{BBM1}.

The Tricorn is the connectedness locus of quadratic antiholomorphic polynomials $\overline{z}^2+c$ (see \cite[\S 2]{IM2} for a general background on the combinatorics and topology of the Tricorn). A hyperbolic component in the parameter space of $\mathcal{N}_4^*$ (respectively, $\mathcal{A}_3$) is called a \emph{Tricorn component} if the corresponding maps have a unique self-conjugate (respectively, self-antipodal) attracting cycle. Such an attracting cycle necessarily attracts both critical orbits of the map. Maps in a Tricorn component behave, in a certain sense, like quadratic antiholomorphic polynomials, and hence can be profitably studied using tools from antiholomorphic dynamics. For a typical Tricorn component in the parameter space $\mathcal{N}_4^*$, see Figure \ref{parameter_plane}.

On the other hand, a hyperbolic component of $\mathcal{N}_4^*$ (respectively, $\mathcal{A}_3$) is called a \emph{Mandelbrot component} if the corresponding maps have two distinct attracting cycles. Due to the real symmetry of the maps, these two cycles are complex conjugate (respectively, antipodal), and hence have the same period.

The boundary of every Tricorn component in both parameter spaces consists of three parabolic arcs (real-analytic arcs of quasiconformally conjugate simple parabolic parameters) and parabolic cusps (double parabolic parameters). In the family $\mathcal{N}_4^*$, every Tricorn component is bounded. More precisely, the boundary of a Tricorn component is a Jordan curve consisting of three parabolic arcs and three cusp points such that two parabolic arcs meet at each cusp. On the other hand, the Tricorn components in the family $\mathcal{A}_3$ come in two different flavors. The first kind of Tricorn components (which are more conspicuous in the parameter space) are unbounded, their boundaries comprise two unbounded parabolic arcs and a bounded parabolic arc. Each unbounded arc meets the unique bounded arc at a finite cusp point, and the two unbounded arcs stretch out to infinity to ``meet" at an ``ideal" cusp point. These components are referred to as \emph{tongues}. The second type of Tricorn components in $\mathcal{A}_3$ are bounded, and their boundaries are again topological triangles with vertices being parabolic cusps and sides being parabolic arcs.

At the ends of every bounded parabolic arc (in either family), there are bifurcations from the Tricorn component to Mandelbrot components across sub-arcs of the arc (see Figure \ref{parameter_plane}). We will refer to the complement of these sub-arcs (across which bifurcation from the Tricorn component to Mandelbrot components occurs) as the ``non-bifurcating" sub-arc of a parabolic arc. We say that a bounded Tricorn component is \emph{invisible} if the non-bifurcating sub-arcs on its boundary do not intersect the closure of any other hyperbolic component in the parameter space (see Definition \ref{parameter_visibility} for a precise formulation, and Figure \ref{inaccessible}(right) and Figure \ref{inaccessible_2}(right) for pictures of invisible Tricorn components).

The principal goal of this article is to study the geometry of the parameter space of the above families near the Tricorn components. Using parabolic implosion methods, we characterize invisible Tricorn components in the parameter space in terms of `invisibility' (see Definition \ref{dynamical_visibility}) of Fatou components in the dynamical plane.

Our main theorem for the family $\mathcal{N}_4^*$ is the following (compare Figure \ref{inaccessible}).

\begin{theorem}[Invisible Tricorn Components in $\mathcal{N}_4^*$]\label{inaccessible_arcs}
For each $n\in\N$, $n>1$, there exists a Tricorn component $H^{(n)}$ of period $2n$ in the parameter space of the family $\mathcal{N}_4^*$ (having its center in the symmetry locus) such that
\begin{enumerate}
\item $\partial H^{(n)}$ has two visible and one invisible parabolic arcs, and

\item every neighborhood of the invisible parabolic arc on $\partial H^{(n)}$ intersects infinitely many capture components and invisible Tricorn components.
\end{enumerate}
\end{theorem}

For the family $\mathcal{A}_3$, we prove the following result which confirms a conjecture of Buff, Bonifant, and Milnor \cite[Remark 6.12]{BBM3}.

\begin{theorem}[Invisible Tricorn Components in $\mathcal{A}_3$]\label{inaccessible_arcs_1}
Let $H$ be a tongue component. Then, the bounded parabolic arc on $\partial H$ is invisible. Moreover, every neighborhood of the non-bifurcating sub-arc of this bounded arc intersects infinitely many capture components and invisible bounded Tricorn components.
\end{theorem}

As the above theorems suggest, the existence of invisible Tricorn components is a fairly general phenomenon in parameter spaces of real bicritical rational maps (also see \cite[Figure 9(c)]{CFG}). Our methods, with minor modifications, apply to any reasonable real bicritical family of rational maps.

Let us now detail the organization of the paper. Part \ref{part_one} of the paper concerns the family $\mathcal{N}_4^*$. In Section \ref{sec_symmetry}, we discuss some elementary consequences of real symmetry in the dynamical plane of the maps $N_a$. In Section \ref{hyperbolic}, we give a classification of hyperbolic components in the family $\mathcal{N}_4^*$. Section \ref{Tricorn_comp} is devoted to studying Tricorn components and their boundaries. The main result of this section is Theorem \ref{Tricorn_boundary}, which states that the boundary of every Tricorn component consists of three parabolic arcs each of which accumulates at parabolic cusps at both ends. Most of the arguments used in this section are inspired by the proofs of the corresponding results in the antiholomorphic polynomial setting. However, the poles of $N_a$ contribute additional complexity to some of the proofs. In Section \ref{sec:ConstructingPCFmaps}, we construct postcritically finite Newton maps (more precisely, centers of Tricorn components of $\mathcal{N}_4^*$) with prescribed combinatorics and topology. This is done by producing a sequence of branched coverings of topological spheres with desired combinatorics, and then invoking W. Thurston's characterization of rational maps to prove that the covers are realized by rational maps which we show to be Newton maps. In particular, this yields infinitely many postcritically finite Newton maps with desired interaction between the basins of attracting fixed points and the immediate basins containing the free critical points. Section \ref{sec_implosion} contains the main technical tool and key lemmas that lead to the proof of existence of invisible Tricorn components in $\mathcal{N}_4^*$. In Subsection \ref{para_imp_back}, we discuss the technique of parabolic implosion for antiholomorphic maps. The notions of visibility of Fatou components in the dynamical plane and Tricorn components in the parameter plane are defined in Subsection \ref{vis_dyn_plane}. In Subsection \ref{char_inv_tri}, we use parabolic implosion techniques to characterize invisible Tricorn components in terms of visibility properties of Fatou components in the dynamical plane. Finally in Section \ref{sec_mainthm_1}, we combine the results of Section \ref{sec:ConstructingPCFmaps} and Section \ref{sec_implosion} to prove Theorem \ref{inaccessible_arcs}.

Part \ref{part_two} of the paper deals with the family $\mathcal{A}_3$. This family has been extensively studied in \cite{BBM1,BBM3}, and the proofs of most of the basic results about this family that we will have need for can be found in their work. We briefly recall the various types of hyperbolic components of $\mathcal{A}_3$ in Section \ref{hyperbolic_1}. The Tricorn components in $\mathcal{A}_3$ come in two different flavors (namely, tongues and bounded Tricorns), and we describe the dynamical differences between their representative maps in Section \ref{tri_comp_antipode}. The final Section \ref{inv_tri_bare_arc} is devoted to the proof of Theorem \ref{inaccessible_arcs_1}. Since the proof of Theorem \ref{inaccessible_arcs_1} essentially follows the same strategy as that of Theorem \ref{inaccessible_arcs}, we only indicate the necessary modifications. We conclude the paper with a complementary result on the existence of bare regions on the boundaries of low period tongues.
\bigskip

\begin{huge}\part{The Family $\mathcal{N}_4^*$}\end{huge}\label{part_one}
\bigskip

In the first part of the paper, we will study the parameter space of degree $4$ real Newton maps, and prove the existence of infinitely many invisible Tricorn components (Theorem \ref{inaccessible_arcs}).
\bigskip

\section{Symmetry in the dynamical plane}\label{sec_symmetry}

In this section, we will delve into some of the consequences of real symmetry of the maps in the family $\mathcal{N}_4^*$. Recall that the two free critical points of every map $N_a$ in $\mathcal{N}_4^*$ are complex conjugate. This limits the number of distinct dynamical configurations as the dynamics of one of the free critical points dictates the dynamics of the other.

We denote the free critical point of $N_a$ that lies in the upper half-plane by $c_a$ and the one that lies in the lower half-plane by $\overline{c_a}$.

Following \cite{M4}, we will define the symmetry locus of the family $\mathcal{N}_4^*$. First note that the automorphism group of the rational map $N_a$ is defined as $$\mathrm{Aut}(N_a)=\{M\in PSL_2(\C): M\circ N_a=N_a\circ M\},$$ where $PSL_2(\C)$ is the group of M{\"o}bius automorphisms of $\hat{\C}$.

\begin{definition}[Symmetry Locus]
The \emph{symmetry locus} of the family $\mathcal{N}_4^*$ is defined as $\mathcal{S}=\{a\in\cU:\mathrm{Aut}(N_a)\neq\{\mathrm{id}\}\}$.
\end{definition}

We will have need for an explicit description of the symmetry locus of the family $\mathcal{N}_4^*$.

\begin{proposition}[Symmetry Locus of $\mathcal{N}_4^*$]\label{prop_symmetry_locus}
$\mathcal{S}=\cU\cap i\R$.
\end{proposition}
\begin{proof}
For each $a\in\cU\cap i\R$, we have $N_a(-z)=-N_a(z)$. So, $a\in\mathcal{S}$.

Now let $a\in\mathcal{S}$, and $M\in\mathrm{Aut}(N_a)\setminus\{\mathrm{id}\}$. Since $N_a$ has a unique repelling fixed point at $\infty$, the M{\"o}bius map $M$ must fix $\infty$. Hence, $M(z)=\alpha z+\beta,$ for some $\alpha, \beta\in\C$.

Note that $N_a$ has precisely two non-fixed critical points $c_a$ and $\overline{c_a}$. Evidently, $M$ must either fix these two critical points, or act as a transposition on them. But if $M$ fixes $c_a$ and $\overline{c_a}$, then $M$ is the identity map (a M{\"o}bius map fixing three points is the identity).

Therefore, we have $M(c_a)=\overline{c_a}$ and $M(\overline{c_a})=c_a$. A simple computation shows that $\alpha=-1$ and $\beta=c_a+\overline{c_a}$. In particular, $M$ fixes the real line. So $M$ must permute the fixed critical points $1$ and $-1$. Since $M$ is not the identity map, we must have $M(1)=-1$ and $M(-1)=1$. It now follows that $\beta=c_a+\overline{c_a}=0$; i.e. $\mathrm{Re}(a)=0$. Thus, $a\in i\R$. It is immediately seen that $\mathrm{Aut}(N_a)=\{\mathrm{id},-z\}$.

We conclude that $\mathcal{S}=\cU\cap i\R$.
\end{proof}

The next proposition describes the location of the poles of $N_a$.

\begin{proposition}[Poles of $N_a$]
\label{poles}
The Newton map $N_a$ has three distinct finite poles. Exactly one of them lies on the real line (more precisely, in the interval $\left(-1,1\right)$), and the other two are complex conjugates of each other.
\end{proposition}
\begin{proof}
Note that $N_a$ fixes $\infty$, so it can have at most three finite poles. Since all zeroes of $f_a$ are simple, it follows that the finite poles of $N_a$ are zeroes of $f_a'$. A brief computation shows that

\begin{equation*}
f_a'(z)=4z^3-6\re(a)z^2+2(\re(a)^2+\im(a)^2-1)z+2\re(a).
\end{equation*}

Since $f_a'(z)$ is a real cubic polynomial, it must have at least one real root which is immediately seen to be a pole of $N_a$. Now, for all $z\in\mathbb{R}$, we have that
\begin{align*}
f_a''(z)&=12z^2-12\re(a)z+2(\re(a)^2+\im(a)^2-1)\\
&=12\left[\left(z-\frac{\re(a)}{2}\right)^2+\frac{2\im(a)^2-\re(a)^2-2}{12}\right]\\
&> 0.
\end{align*}

This proves that $f_a'$ has a unique simple real root. As $f_a'$ is a real polynomial, the other two roots of $f_a'$ must be complex conjugate. Moreover, $f_a'(1)=2(\re(a)-1)^2+2\im(a)^2>0$ and $f_a'(-1)=-2(\re(a)+1)^2-2\im(a)^2<0$. By the intermediate value theorem, the unique real pole of $N_a$ lies in $\left(-1,1\right)$.
\end{proof}

For any $a$, let us denote the unique real pole of $N_a$ by $p_a$. The immediate basins of attraction of the fixed points $1, -1, a$, and $\bar{a}$ of $N_a$ will be denoted by $\mathcal{B}^{\mathrm{imm}}_{1}, \mathcal{B}^{\mathrm{imm}}_{-1}, \mathcal{B}^{\mathrm{imm}}_{a}$, and $\mathcal{B}^{\mathrm{imm}}_{\bar{a}}$ (respectively). The full basins of these fixed points will be denoted simply by dropping the superscript $``\mathrm{imm}"$. In the following proposition, we will collect a number of easy results about the topology and mapping properties of these immediate basins.

\begin{proposition}\label{basins_prop}
1) All immediate basins are simply connected.

2) The immediate basins $\mathcal{B}^{\mathrm{imm}}_{1}$ and $\mathcal{B}^{\mathrm{imm}}_{-1}$ are invariant under $\iota$. The other two immediate basins $\mathcal{B}^{\mathrm{imm}}_{a}$ and $\mathcal{B}^{\mathrm{imm}}_{\bar{a}}$ are disjoint from the real line, and they are interchanged by $\iota$.

3) If $1$ (respectively $-1$) is the only critical point in $\mathcal{B}^{\mathrm{imm}}_{1}$ (respectively in $\mathcal{B}^{\mathrm{imm}}_{-1}$), then $N_a:\mathcal{B}^{\mathrm{imm}}_{1}\to\mathcal{B}^{\mathrm{imm}}_{1}$ (respectively $N_a:\mathcal{B}^{\mathrm{imm}}_{-1}\to\mathcal{B}^{\mathrm{imm}}_{-1}$) is a $2:1$ branched cover. Otherwise, $\mathcal{B}^{\mathrm{imm}}_{1}$ (respectively $\mathcal{B}^{\mathrm{imm}}_{-1}$) contains three critical points (the fixed point and the two free critical points), and $N_a:\mathcal{B}^{\mathrm{imm}}_{1}\to\mathcal{B}^{\mathrm{imm}}_{1}$ (respectively $N_a:\mathcal{B}^{\mathrm{imm}}_{-1}\to\mathcal{B}^{\mathrm{imm}}_{-1}$)  is a $4:1$ branched cover.

4) If $\mathcal{B}^{\mathrm{imm}}_{1}$ (respectively $\mathcal{B}^{\mathrm{imm}}_{-1}$) contains only one critical point, then $N_a:\mathcal{B}^{\mathrm{imm}}_{1}\to\mathcal{B}^{\mathrm{imm}}_{1}$ (respectively $N_a:\mathcal{B}^{\mathrm{imm}}_{-1}\to\mathcal{B}^{\mathrm{imm}}_{-1}$) is conformally conjugate to $w^2:\mathbb{D}\to\mathbb{D}$ via the Riemann map of $\mathcal{B}^{\mathrm{imm}}_{1}$ (respectively $\mathcal{B}^{\mathrm{imm}}_{-1}$).

5) If $\mathcal{B}^{\mathrm{imm}}_{1}$ and $\mathcal{B}^{\mathrm{imm}}_{-1}$ contain only one critical point each, then $p_a\in \partial \mathcal{B}^{\mathrm{imm}}_{1}\cap\partial \mathcal{B}^{\mathrm{imm}}_{-1}$, and  $\mathbb{R}\setminus\{p_a\} \subset \mathcal{B}^{\mathrm{imm}}_{1}\cup \mathcal{B}^{\mathrm{imm}}_{-1}$.
\end{proposition}

\begin{proof}
1) This follows from the fact that Julia sets of Newton maps arising from polynomials are connected \cite{Prz,Shi1}.

2) This is obvious as $\iota$ conjugates $N_a$ to itself, fixes $1$ and $-1$, and acts as a transposition on the set $\{a, \bar{a}\}$. Moreover, as $N_a$ leaves the real line invariant and $a$, $\bar{a}$ are strictly complex, the real line cannot intersect the immediate basins $\mathcal{B}^{\mathrm{imm}}_{a}$ and $\mathcal{B}^{\mathrm{imm}}_{\bar{a}}$.

3) Note that due to the symmetry of $\mathcal{B}^{\mathrm{imm}}_{1}$ (respectively $\mathcal{B}^{\mathrm{imm}}_{-1}$) with respect to $\iota$, if $\mathcal{B}^{\mathrm{imm}}_{1}$ (respectively $\mathcal{B}^{\mathrm{imm}}_{-1}$) contains one of the free critical points then it must contain the other as well. The statements about degrees of the map $N_a:\mathcal{B}^{\mathrm{imm}}_{1}\to\mathcal{B}^{\mathrm{imm}}_{1}$ (respectively $N_a:\mathcal{B}^{\mathrm{imm}}_{-1}\to\mathcal{B}^{\mathrm{imm}}_{-1}$) now directly follow from the Riemann-Hurwitz formula.

4) If the super-attracting Fatou component $\mathcal{B}^{\mathrm{imm}}_{1}$ (respectively $\mathcal{B}^{\mathrm{imm}}_{-1}$) contains only one critical point (which is necessarily a simple critical point), then the corresponding B{\"o}ttcher coordinate extends as a biholomorphism between $\mathcal{B}^{\mathrm{imm}}_{1}$ (respectively $\mathcal{B}^{\mathrm{imm}}_{-1}$) and $\mathbb{D}$ such that it conjugates $N_a$ to $w^2$.

5) Let $\phi_1:\mathcal{B}^{\mathrm{imm}}_{1}\to\mathbb{D}$ be the\footnote{since $N_a$ is $2:1$ on the immediate basin, the B{\"o}ttcher coordinate is unique.} B{\"o}ttcher coordinate that conjugates $N_a$ to $w^2$. The antiholomorphic involution $\iota$ of $\mathcal{B}^{\mathrm{imm}}_{1}$ can be transported by $\phi_1$ to define an antiholomorphic involution $\widetilde{\iota}:=\phi_1 \circ \iota \circ \phi_1^{\circ (-1)}$ of $\mathbb{D}$. As $\phi_1(1)=0$ and $\iota$ fixes $1$, it follows that $\widetilde{\iota}$ fixes $0$. A simple computation using the description of conformal automorphisms of $\mathbb{D}$ implies that there exists an $\alpha\in\mathbb{S}^1$ such that $\widetilde{\iota}(w)=\alpha \bar{w}$, for all $w\in\mathbb{D}$. Again, since $\iota$ conjugates $N_a$ to itself, $\widetilde{\iota}$ must conjugate $w^2$ to itself on $\mathbb{D}$. Therefore, $\alpha=1$; i.e. $\widetilde{\iota}(w)=\overline{w}$, for all $w\in\mathbb{D}$.

Since the radial lines at angles $0$ and $1/2$ in $\mathbb{D}$ are fixed by $\widetilde{\iota}$, it follows that the dynamical rays at angles $0$ and $1/2$ in $\mathcal{B}^{\mathrm{imm}}_{1}$ are fixed by $\iota$. Hence, these two dynamical rays are contained in the real line. As the dynamical $0$-ray is fixed by $N_a$, it must land at $\infty$ (which is the only fixed point of $N_a$ on $\partial \mathcal{B}^{\mathrm{imm}}_{1}$). Therefore, the dynamical $0$-ray in $\mathcal{B}^{\mathrm{imm}}_{1}$ is the interval $\left[1,+\infty\right)$. Similarly, the dynamical $1/2$-ray must land at a pole of $N_a$ on $\mathbb{R}$. Hence it must land on $p_a$; i.e. the dynamical $1/2$-ray in $\mathcal{B}^{\mathrm{imm}}_{1}$ is the interval $\left(p_a,1\right]$. In particular, $p_a\in\partial \mathcal{B}^{\mathrm{imm}}_{1}$.

One can similarly prove that the dynamical $0$-ray in $\mathcal{B}^{\mathrm{imm}}_{-1}$ is the interval $\left(-\infty,-1\right]$ and the dynamical $1/2$-ray in $\mathcal{B}^{\mathrm{imm}}_{-1}$ is the interval $\left[1,p_a\right)$. In particular, $p_a\in\partial \mathcal{B}^{\mathrm{imm}}_{-1}$.

This proves that $p_a\in \partial \mathcal{B}^{\mathrm{imm}}_{1}\cap\partial \mathcal{B}^{\mathrm{imm}}_{-1}$, and  $\mathbb{R}\setminus\{p_a\} \subset \mathcal{B}^{\mathrm{imm}}_{1}\cup \mathcal{B}^{\mathrm{imm}}_{-1}$.
\end{proof}

\begin{proposition}\label{no_free}
The immediate basins $\mathcal{B}^{\mathrm{imm}}_{a}$ and $\mathcal{B}^{\mathrm{imm}}_{\bar{a}}$ do not contain any free critical point.
\end{proposition}
\begin{proof}
Let us suppose that the immediate basin $\mathcal{B}^{\mathrm{imm}}_{a}$ contains a free critical point. Then $\mathcal{B}^{\mathrm{imm}}_{a}$ has two accesses to infinity (see \cite[Proposition~6]{HSS}). By \cite[Corollary 5.2]{RS}, there must be a zero of $f_a$ (different from $a$ itself) between these two accesses. However, $\mathcal{B}^{\mathrm{imm}}_{a}$ is contained in the upper half-plane and no other zero of $f_a$ lies in the upper half-plane. Therefore, no zero of $f_a$ can lie between these two accesses to infinity. This contradicts the assumption that $\mathcal{B}^{\mathrm{imm}}_{a}$ contains a free critical point.

The proof for the immediate basin $\mathcal{B}^{\mathrm{imm}}_{\bar{a}}$ is analogous.
\end{proof}
\bigskip

\section{Classification of hyperbolic components}\label{hyperbolic}

Let us start with some basic definitions and notations that we will need in the rest of the paper. The set of all critical points of a rational map $R$ (or a branched cover of the $2$-sphere of degree $d\geq 2$) is denoted by $C(R)$. The postcritical set of $R$ is defined as the iterated forward images of all the critical points of $R$, and is denoted by $P_R$. In other words, $P_R=\displaystyle\bigcup_{n=1}^\infty R^{\circ n}(C(R))$. A rational map (or a branched cover of the sphere) is called postcritically finite if $P_R$ is a finite set.

Recall that a rational map is called \emph{hyperbolic} if the forward orbit of each critical point of the map converges to an attracting cycle. Since the critical points $\pm 1$, $a$, and $\overline{a}$ of every map $N_a$ are fixed, it follows that a map $N_a$ in $\mathcal{N}_4^*$ is hyperbolic if and only if the forward orbits of the two free critical points converge to attracting cycles. If $N_a$ is a hyperbolic map, the corresponding parameter $a$ is called a hyperbolic parameter. A connected component of the set of all hyperbolic parameters (in $\mathcal{N}_4^*$) is called a hyperbolic component. A hyperbolic parameter $a_0$ is called a center of a hyperbolic component if $N_{a_0}$ is postcritically finite.

Using the results of Section \ref{sec_symmetry}, we will now describe all possible types of hyperbolic components that appear in the family $\mathcal{N}_4^*$. For a hyperbolic map $N_a$, the free critical points must converge to attracting cycles. This can happen in a variety of ways. However, the dynamical symmetry of the two free critical points ensures that the two free critical points exhibit symmetric dynamical behavior.

Recall that each $N_a$ has five fixed points; while $\infty$ is a repelling fixed point, the others (namely $\pm 1$, $a$, $\bar{a}$) are super-attracting. We will first discuss the case when the orbit of a free critical point converges to one of these super-attracting fixed points. By Proposition \ref{no_free}, the free critical points cannot belong to the immediate basins of $a$ or $\overline{a}$. On the other hand, if one free critical point lies in the basin of the fixed point $+1$ (respectively $-1$), then by Proposition \ref{basins_prop} the other one must lie in the same basin too. These facts lead to the following three types of hyperbolic components.

\textbf{Principal hyperbolic components}: There are two principal hyperbolic components in the family $\mathcal{N}_4^*$.

$\mathcal{H}_1$: This consists of the set of parameters $a$ for which the immediate basin of attraction of the fixed point $1$ contains both free critical points. In this case, the immediate basin of attraction of the fixed point $1$ has three accesses to infinity and all other immediate basins have a unique access to infinity. Moreover, any Fatou component eventually maps to one of the immediate basins. In Figure \ref{parameter_plane}, the red unbounded component is $\mathcal{H}_1$.

$\mathcal{H}_{-1}$: This consists of the set of parameters $a$ for which the immediate basin of attraction of the fixed point $-1$ contains both free critical points. In this case, the immediate basin of attraction of the fixed point $-1$ has three accesses to infinity and all other immediate basins have a unique access to infinity. Moreover, any Fatou component eventually maps to one of the immediate basins. In Figure \ref{parameter_plane}, the yellow unbounded component is $\mathcal{H}_{-1}$.

\textbf{Capture components}: A capture component $H$ is a connected component of the hyperbolicity locus such that for every parameter $a$ in $H$, the free critical points of $N_a$ lie in pre-periodic Fatou components that eventually map to some of the immediate basins of attraction of the fixed points. If one of the free critical points lies in the basin of $1$ (respectively $-1$), then the other must belong to the same basin. On the other hand, if one of the free critical points lies in the basin of $a$, then the other must lie in the basin of $\overline{a}$. Once again, each Fatou component eventually maps to one of the immediate basins. In Figure \ref{parameter_plane}, the capture components are bounded and shaded in red, yellow, green or orange (depending on which fixed points the free critical points converge to).

We now consider the case when a free critical point converges to an attracting cycle of period greater than one. Since the cycle of immediate basins of any attracting cycle must contain a critical point (necessarily one of the two available free critical points in this case), such a hyperbolic map can have at most two attracting cycles. Therefore, either there is a unique self-conjugate attracting cycle that attracts both free critical points, or there are two distinct attracting cycles each attracting one free critical point.

\textbf{Mandelbrot components:} A Mandelbrot component $H$ is a connected component of the hyperbolicity locus such that for every parameter $a$ in $H$, the map $N_a$ has two distinct attracting cycles (of period greater than one). Due to symmetry, these two attracting cycles are mapped to each other by $\iota$. Hence both these attracting cycles have a common period $n$. The immediate basins of each of these two attracting cycles contain one of the free critical points of $N_a$. Every Fatou component of $N_a$ eventually maps to one of the fixed immediate basins or to one of these two cycles of immediate basins. We will refer to such a component as a Mandelbrot component of period $n$. Figure \ref{mandelbrot} shows a Mandelbrot component.

\begin{figure}[ht!]
\begin{center}
\includegraphics[scale=0.3]{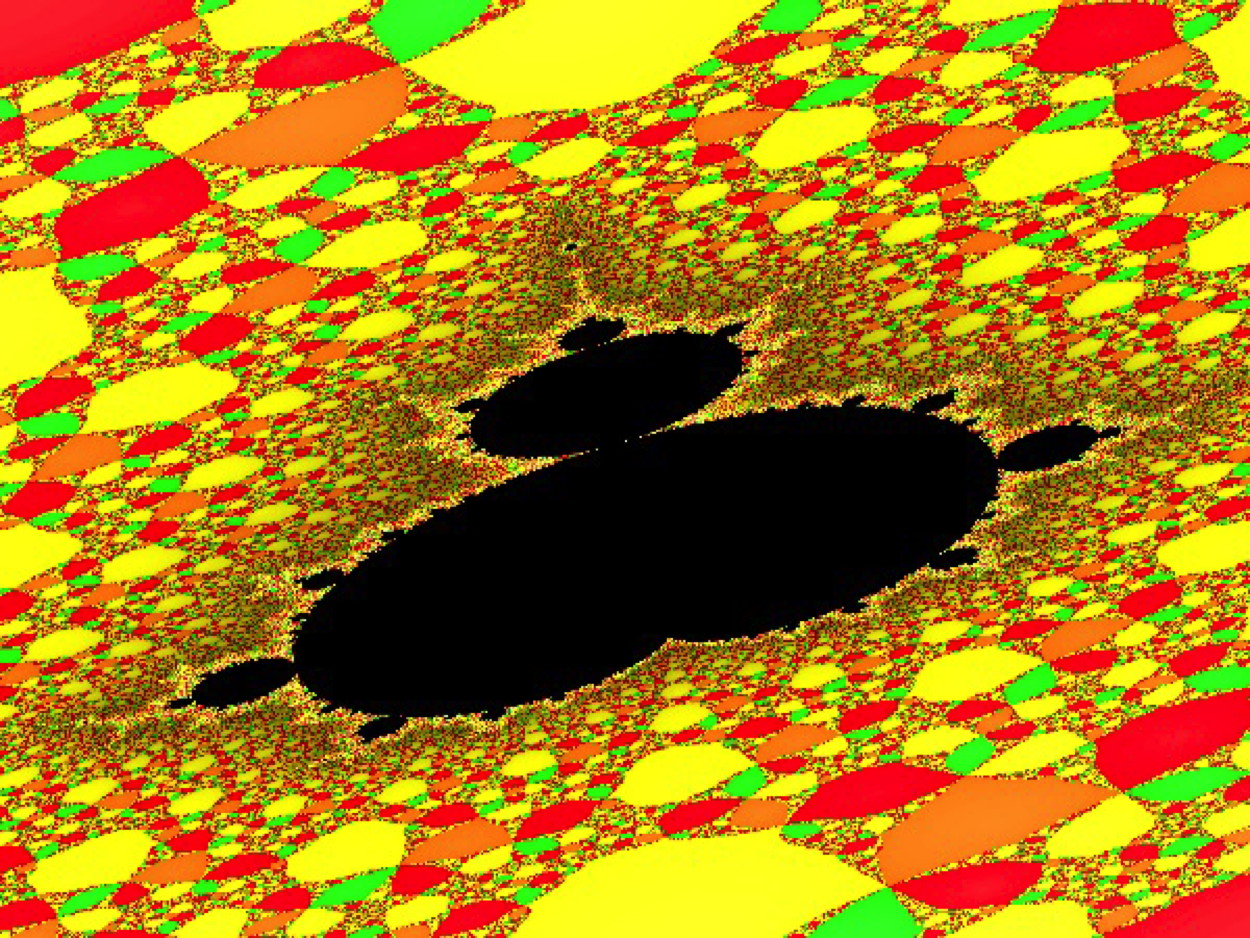}
\caption{The main cardioid of the ``baby Mandelbrot set'' (in black) is a Mandelbrot component.}
\label{mandelbrot}
\end{center}
\end{figure}

\textbf{Tricorn components:} A Tricorn component $H$ is a connected component of the hyperbolicity locus such that for every parameter $a$ in $H$, the map $N_a$ has a unique attracting cycle (of period greater than one). Such an attracting cycle is necessarily self-conjugate and hence must attract both free critical points (as $\iota$ conjugates $N_a$ to itself). By Proposition \ref{basins_prop}, the attracting cycle is disjoint from the real line. Therefore, $\iota$ acts as a fixed-point free involution on this attracting cycle. It follows that the period of the attracting cycle must be an even integer $2n$. Every Fatou component of $N_a$ eventually maps to one of the fixed immediate basins or to this $2n$-periodic cycle of immediate basins. Figure \ref{parameter_plane} shows a Tricorn component (in black).\footnote{Our definition of Tricorn components follows \cite{BBM1,BBM3}. We should warn the readers that this definition is different from the usual definition of Tricorn in the context of quadratic antiholomorphic polynomials. A Tricorn component, as defined in this article, is only a hyperbolic component and does \emph{not} include the decorations attached to it.}

The following proposition follows from the general theory of hyperbolic components of rational maps \cite[Theorem 7.13, Theorem 9.3]{MP}.

\begin{proposition}\label{unique_center}
Every hyperbolic component in the parameter space of the family $\mathcal{N}_4^*$ is simply connected and has a unique center (i.e. a postcritically finite parameter).
\end{proposition}

Let us take a more careful look at the dynamics of the maps in a Tricorn component, which are the main objects of study of this article. Let $a$ belong to a Tricorn component $H$ of period $2n$. Since the real line is disjoint from the immediate basins of this $2n$-periodic cycle, the two free critical points $c_a$ and $\overline{c_a}$ must lie in two distinct immediate basins of this attracting cycle. Let us label the periodic (attracting) Fatou components $\{U_1, U_2, \cdots, U_{2n}\}$ so that $c_a$ is contained in $U_1$. By symmetry, it follows that $\overline{c_a}$ belongs to the conjugate Fatou component $\iota(U_1)$. For any fixed $j$, let $k$ be the smallest positive integer such that $N_{a}^{\circ k}(U_j)=\iota(U_j)$. Then $\iota\circ N_{a}^{\circ k}(U_j)=U_j$; i.e.\ $N_{a}^{\circ k}(\iota(U_j))=U_j$. Therefore, $N_{a}^{\circ 2k}(U_j)=U_j$. This implies that $k=n$; i.e. $N_{a}^{\circ n}(U_j)=\iota(U_j)$ for each $j$ in $\{1, 2, \cdots, 2n\}$. In particular, $N_{a}^{\circ n}(U_1)=\iota(U_1)$.

\bigskip

\section{Tricorn components and their boundaries}\label{Tricorn_comp}

Throughout this section, we assume that $H$ is a Tricorn component of period $2n$. The dynamics of the maps in the Tricorn components are reminiscent of antiholomorphic dynamics. Since $N_a$ commutes with the complex conjugation map $\iota$, we have that $N_a^{\circ 2n}=(\iota\circ N_a^{\circ n})\circ(\iota\circ N_a^{\circ n})$. Hence the first return map $N_a^{\circ 2n}$ of any $2n$-periodic Fatou component is the second iterate of the first antiholomorphic return map $\iota\circ N_a^{\circ n}$.

By Proposition \ref{unique_center}, each Tricorn component is simply connected. An explicit real-analytic uniformization of the Tricorn components can be found in \cite[Theorem 5.9]{NS} or \cite[Lemma 3.2]{BMMS}.

\begin{proposition}
Each Tricorn component $H$ is simply connected. Moreover, there is a dynamically defined real-analytic three-fold cover from $H$ to the unit disk, ramified only over the origin.
\end{proposition}

One of the main features of antiholomorphic dynamics is the existence of abundant parabolics. In particular, any indifferent fixed point of an antiholomorphic map is necessarily parabolic with multiplier $+1$. Hence the boundaries of Tricorn components consist only of parabolic parameters. In fact, $\partial H$ contains three parabolic arcs (real-analytic arcs consisting of simple parabolic parameters) that limit at cusp points (double parabolic parameters).

\begin{lemma}[Indifferent Dynamics on the Boundary of Tricorn Components]\label{LemTricornIndiffDyn}
The boundary of a Tricorn component consists
entirely of parameters having a self-conjugate parabolic cycle of multiplier $+1$. In suitable
local conformal coordinates, the first return map of some (neighborhood of a) parabolic point of such a map has the form $z\mapsto z+z^{r+1}+\ldots$ with $r\in\{1,2\}$.
\end{lemma}
\begin{proof}
For a proof of this fact in the antiholomorphic polynomial case, see \cite[Lemma 2.5]{MNS}. The same argument, which is essentially based on local fixed point theory of holomorphic germs, apply in the current setting as well. The fact that the parabolic cycle is self-conjugate can be seen as follows. For every parameter in a Tricorn component $H$, the two free critical points lie in the same cycle of attracting Fatou components. Since the parabolic cycle is formed by the merger of the unique attracting cycle with one or more repelling periodic cycles, it follows that the resulting cycle(s) of parabolic basins contain(s) both free critical points. Hence, the parabolic cycle is unique. As $\iota$ conjugates $N_a$ to itself, this parabolic cycle must be fixed (as a set) by $\iota$.
\end{proof}

This leads to the following definition.

\begin{definition}[Parabolic Cusps]\label{DefCusp}
A parameter $a$ on $\partial H$ is called a {\em parabolic cusp point} if it has a self-conjugate parabolic
cycle such that $r=2$ in the previous lemma. Otherwise, it is called a \emph{simple} parabolic parameter.
\end{definition}

For the rest of this section, $a$ will stand for a simple parabolic parameter on the boundary of $H$. The holomorphic return map $N_a^{\circ 2n}$ of any attracting petal is conformally conjugate to translation by $+1$ in a right half-plane (see Milnor~\cite[Section~10]{M1new}. The conjugating map $\psi^{\mathrm{att}}$ is called an attracting \emph{Fatou coordinate}. Thus the quotient of the petal by the dynamics is isomorphic to a bi-infinite cylinder, called the \emph{Ecalle cylinder}. Note that Fatou coordinates are uniquely determined up to addition by a complex constant. However, since each petal has an intermediate antiholomorphic return map $\iota\circ N_a^{\circ n}$, we can choose a special attracting Fatou coordinate that provides us with a crucial conformal invariant for parabolic maps.

\begin{lemma}[Fatou Coordinates]\label{normalization of fatou}
Let $a$ be a simple parabolic parameter on the boundary of $H$, $z_i$ be a parabolic periodic point of $N_a$ (so $z_i$ has only one attracting petal), and $U_i$ be a $2n$-periodic Fatou component with $z_i \in \partial U_i$. Then there is an open connected subset $V \subset U_i$ with $z_i \in \partial V$, and $(\iota\circ N_a^{\circ n})(V) \subset V$ so that for every $z \in U_i$, there is a $k \in \mathbb{N}$ with $(\iota\circ N_a^{\circ n})^{\circ k}(z)\in
V$. Moreover, there is a univalent map $\psi^{\mathrm{att}} \colon V \to \mathbb{C}$ with $\psi^{\mathrm{att}}(\iota\circ N_a^{\circ n}(z)) = \overline{\psi^{\mathrm{att}}(z)}+1/2$, and $\psi^{\mathrm{att}}(V)$ contains a right half-plane. This map $\psi^{\mathrm{att}}$ is unique up to horizontal translation.
\end{lemma}
\begin{proof}
The proof is similar to the proof of \cite[Lemma 2.3]{HS}.
\end{proof}

The map $\psi^{\mathrm{att}}$ in the previous lemma will be our normalized Fatou coordinate for the petal $V$. We can extend $\psi^{\mathrm{att}}$ analytically to the entire Fatou component $U_i$ so that it is a semi-conjugacy between $(\iota\circ N_a^{\circ n})\vert_{U_i}$ and $\zeta \mapsto \overline{\zeta}+1/2$ on $\mathbb{C}$. The antiholomorphic map $\iota\circ N_a^{\circ n}$ interchanges the two ends of the Ecalle cylinder, so it must fix one horizontal line around this cylinder (the \emph{equator}). The change of coordinate has been so chosen that the equator maps to the real axis.  We will call the vertical Fatou coordinate the \emph{Ecalle height}. Its origin is the equator. Of course, the same can be done in the repelling petal as well. We will refer to the equator in the attracting (respectively repelling) petal as the attracting (respectively repelling) equator. The existence of this distinguished real line, or equivalently an intrinsic meaning to Ecalle height, is specific to antiholomorphic maps.

We will call the Fatou component $U_1$ containing the critical point $c_a$ the \emph{characteristic Fatou component}.\footnote{This definition is not standard. In unicritical polynomial dynamics, the characteristic Fatou component is usually defined as the unique bounded Fatou component containing the critical value.} Clearly, there is a parabolic point $z_1$ on the boundary of $U_1$ such that the $N_a^{\circ 2n}$-orbit of each point in $U_1$ converges to $z_1$. We will refer to $z_1$ as the \emph{characteristic parabolic point}. We will mainly work with $U_1$, and its normalized attracting Fatou coordinate (as above) $\psi^{\mathrm{att}}$.  The following well-defined quantity, which is called the \emph{critical Ecalle height} of a parameter $a$ with a simple self-conjugate parabolic cycle, is going to be of fundamental importance in the remainder of this paper.

\begin{align*}
h(N_a)&:=\im(\psi^{\mathrm{att}}(c_a))\\
&=\left(\im(\psi^{\mathrm{att}}(c_a))-\im(\psi^{\mathrm{att}}(\iota\circ N_a^{\circ n}(c_a)))\right)/2\\
&=\left(\im(\psi^{\mathrm{att}}(c_a))-\im(\psi^{\mathrm{att}}(N_a^{\circ n}(\overline{c_{a}})))\right)/2.
\end{align*}

We could define critical Ecalle height of $N_a$ as the Ecalle height of the critical point $\overline{c_{a}}$ as well. In fact, the two definitions only differ by a sign. It follows that the set $\lbrace h(N_a), -h(N_a)\rbrace$ is a conformal conjugacy invariant of the simple parabolic map $N_a$.

The next proposition, which is similar to \cite[Theorem 3.2]{MNS}, shows the existence of real-analytic arcs of simple parabolic parameters on the boundaries of Tricorn components. For an orientation reversing map $f$, the pull-back of a Beltrami coefficient $\mu$ under $f$ is defined as (see~\cite[Exercise~1.2.2]{BF14})
$$
f^*(\mu(z))=\overline{\left(\frac{\partial f/\partial z+\mu({f(z)})\partial\bar f/\partial z}{\partial f/\partial\bar z+\mu({f(z)})\partial \bar f/\partial{\bar z}}\right)}.
$$
A Beltrami coefficient $\mu$ is said to be $f$-invariant if $f^*(\mu)=\mu$.

\begin{proposition}[Parabolic Arcs]\label{parabolic arcs}
Let $\widetilde{a}$ be a parameter such that $N_{\widetilde{a}}$ has a self-conjugate simple parabolic cycle. Then $\widetilde{a}$ is on a parabolic arc in the  following sense: there  exists an injective real-analytic arc $\cC$ of simple parabolic parameters $a(h)$ (for $h\in\mathbb{R}$) with quasiconformally equivalent dynamics of which $\widetilde{a}$ is an interior point, and the critical Ecalle height of $N_{a(h)}$ is $h$. In particular, each $N_{a(h)}$ has a self-conjugate simple parabolic cycle.
\end{proposition}
\begin{proof}
We will use quasiconformal deformations to change the critical Ecalle height of $N_{\widetilde{a}}$. Note that the forward orbit of the critical point $c_{\widetilde{a}}$ is
contained in the parabolic basin. We parametrize the horizontal
coordinate within the incoming Ecalle cylinder by $\mathbb{R}/\mathbb{Z}$.

Now, choose the attracting Fatou coordinate $\Psi_{\widetilde{a}}:z\to\zeta$ (at the characteristic parabolic point $z_{\widetilde{a}}$) as in Lemma \ref{normalization of fatou} so that $c_{\widetilde{a}}$ has real part $1/4$ within the Ecalle cylinder (note that the Fatou coordinates constructed in Lemma \ref{normalization of fatou} are unique up to addition of a \emph{real} constant). Let us denote the imaginary part of the Fatou coordinate of $c_{\widetilde{a}}$ by $\widetilde{h}$, so the {\it critical Ecalle height} of $N_{\widetilde{a}}$ is $\widetilde{h}$.

We will change the Ecalle height of $c_a$ in a controlled way so that each perturbation gives a different map in the family $\mathcal{N}_4^*$. Setting $\zeta=x+iy$, we can change the conformal structure within the Ecalle cylinder by the quasi-conformal self-homeomorphism of $\left[ 0 , 1 \right] \times \mathbb{R}$:

$$
L_h : \zeta \mapsto \left\{\begin{array}{ll}
                    \zeta +4i h x & \mbox{if}\ 0\leq x\leq 1/4 \\
                    \zeta + 2i h (1-2x) & \mbox{if}\ 1/4\leq x\leq 1/2 \\
                    \zeta - 2i h (2x-1) & \mbox{if}\ 1/2\leq x\leq 3/4 \\
                    \zeta -4 i h(1-x) & \mbox{if}\ 3/4\leq x\leq 1
                      \end{array}\right.
$$
for $h\in \mathbb{R}$. Translating the map $L_h$ by positive integers, we get a quasiconformal homeomorphism of a right half-plane that commutes with the translation $\zeta \mapsto \overline{\zeta}+1/2$ (alternatively we could define $L_h$ on $[0,1/2]\times\mathbb{R}$ using the above formula and then extend by the glide reflection $(x,y)\mapsto(x+1/2,-y)$). By the coordinate change $z\mapsto \zeta$, we can transport this Beltrami
form (defined by this quasiconformal homeomorphism) into the attracting petal at $z_{\widetilde{a}}$. The construction guarantees that the Beltrami form so defined is forward invariant under $\iota\circ N_{\widetilde{a}}^{\circ n}$ and hence under $(\iota\circ N_{\widetilde{a}}^{\circ n})^{\circ 2} = N_{\widetilde{a}}^{\circ 2n}$. It
is easy to make it backward invariant by pulling it back by the
dynamics. Extending it by the zero Beltrami form outside of  the entire
parabolic basin, we obtain an $N_{\widetilde{a}}$-invariant Beltrami form. Since $N_{\widetilde{a}}$ commutes with $\iota$, it follows that this Beltrami form is also $\iota$-invariant. Using the Measurable
Riemann Mapping Theorem with parameters, we obtain a quasiconformal map $\phi_h$ integrating this
Beltrami form. If we normalize $\phi_h$ such that it fixes $1$, $-1$ and $\infty$, then the coefficients of this newly obtained rational map will depend real-analytically on $h$ (since the Beltrami form depends real-analytically on $h$). Moreover, since the Beltrami form is $\iota$-invariant, $\phi_h \circ \iota\circ \phi_h^{-1}$ is an antiholomorphic map.

We need to check that this new degree $4$ rational map belongs to the family $\mathcal{N}_4^*$. Since $\phi_h$ is a topological conjugacy, it maps the four super-attracting fixed points of $N_{\widetilde{a}}$ to super-attracting fixed points of $\phi_h \circ N_{\widetilde{a}}\circ \phi_h^{-1}$. Moreover, as $\phi_h(\infty)=\infty$, a simple application of the holomorphic fixed point formula proves that $\infty$ is a repelling fixed point of $\phi_h \circ N_{\widetilde{a}}\circ \phi_h^{-1}$ with multiplier $4/(4-1)=4/3$. It follows by \cite[Corollary 2.9]{RS} that $\phi_h \circ N_{\widetilde{a}}\circ \phi_h^{-1}$ is the Newton map of a degree $4$ polynomial $g_h$. The fact that $\phi_h$ fixes $1$ and $-1$ implies that $1$ and $-1$ are roots of $g_h$. Since $N_{\widetilde{a}}$ commutes with the antiholomorphic involution $\iota$ of $\hat{\mathbb{C}}$, it follows that $\phi_h \circ N_{\widetilde{a}}\circ \phi_h^{-1}$ commutes with the antiholomorphic involution $\phi_h \circ \iota\circ \phi_h^{-1}$ of $\hat{\mathbb{C}}$. However any such antiholomorphic involution is an anti-M{\"o}bius map. Since $\phi_h \circ \iota\circ \phi_h^{-1}$ fixes $1$, $-1$ and $\infty$, we have that $\phi_h \circ \iota\circ \phi_h^{-1}=\iota$. Therefore, the Newton map $\phi_h \circ N_{\widetilde{a}} \circ \phi_h^{-1}$ commutes with $\iota$. Let us define $a(h)=\phi_h(\widetilde{a})$. It then follows that $a(h)$ and $\overline{a(h)}$ are the remaining two simple roots of $g_h$. Hence, $\phi_h \circ N_{\widetilde{a}} \circ  \phi_h^{-1}=N_{a(h)}$. Again, $\phi_h$ sends the two distinct free critical points $c_{\widetilde{a}}$ and $\overline{c_{\widetilde{a}}}$ of $N_{\widetilde{a}}$ to a pair of complex conjugate critical points of $\phi_h \circ N_{\widetilde{a}} \circ \phi_h^{-1}$. Hence, the two free critical points of $N_{a(h)}$ are complex conjugate. Therefore, $a(h)\in\cU\cup\iota(\cU)$; i.e. it lies in one of the two connected components of the shaded region in Figure \ref{conjugate}. Finally since $a(h)$ is a continuous function of $h$ and $a(0)=\widetilde{a}\in\cU$, we conclude that each $a(h)$ lies in $\cU$. This completes the proof of the fact that each deformation of $N_{\widetilde{a}}$ belongs to the family $\mathcal{N}_4^*$.

We need to show that for any $h \in \mathbb{R}$, the map $N_{a(h)}$ has a simple parabolic cycle. But this follows from Naishul's theorem \cite{Na}, which states that the multiplier of an indifferent periodic point is preserved under topological conjugacies, and the fact that the number of attracting petals at a parabolic point is a topological invariant of parabolic germs. Moreover, since $\phi_h$ commutes with $\iota$, it sends the self-conjugate parabolic cycle of $N_{\widetilde{a}}$ to a self-conjugate parabolic cycle of $N_{a(h)}$.

By construction, all the $N_{a(h)}$ are quasiconformally conjugate to $N_{\widetilde{a}}$, and hence to each other. Note that a Fatou coordinate of $N_{a(h)}$ (at the characteristic parabolic point) is given by $\Psi_h = L_h \circ \Psi_{\widetilde{a}} \circ \phi_h^{-1}$. Hence, the Ecalle height of the critical point $c_{a(h)}$ is $\im(L_h(\frac{1}{4}+i\widetilde{h})) = \im(\frac{1}{4}+i(\widetilde{h}+h))=(\widetilde{h}+h)$. In particular, for $h_1\neq h_2$ the maps $N_{a(h_1)}$ and $N_{a(h_2)}$ have distinct critical Ecalle heights. This proves that the map $h\mapsto a(h)$ is injective. The image $a((-\infty, \infty))$ is an injective real-analytic arc in the parameter plane, which we call a parabolic arc (denoted by $\cC$).

The above argument produces a parametrization of $\cC$ such that the critical Ecalle height of $N_{a(h)}$ is $\widetilde{h}+h$. We can reparametrize $\cC$ such that the critical Ecalle height of $N_{a(h)}$ is $h$. The completes the proof of the proposition.
\end{proof}

\begin{remark}
The parametrization of $\cC$ such that the critical Ecalle height of $N_{a(h)}$ is $h$ is called the critical Ecalle height parametrization of the parabolic arc $\cC$. Abusing notation, we will denote the critical Ecalle height parametrization of the parabolic arc $\cC$ by $a:\R\to\cC$.
\end{remark}

The next proposition follows essentially by an argument analogous to \cite[Theorem 3.8, Corollary 3.9]{HS}.

\begin{proposition}[Bifurcation along Arcs]\label{bifurcation_arcs}
Every parabolic arc has, at both ends, an interval of positive length across which bifurcation from a Tricorn component of period $2n$ to a Mandelbrot component of period $2n$ occurs.
\end{proposition}

If two distinct parabolic arcs intersect at some parameter $a$, then $N_a$ will either have a double parabolic cycle (i.e. a parabolic cycle of multiplicity two) or two distinct self-conjugate parabolic cycles. But this is impossible as $a$ lies on a parabolic arc and hence has a unique simple self-conjugate parabolic cycle. Therefore, two distinct parabolic arcs cannot intersect. It follows that any parabolic arc that intersects $\partial H$ must in fact be contained in $\partial H$.

For any $a$ in a Tricorn component $H$ of period $2n$ (respectively, for a simple parabolic parameter $a$ on $\partial H$), we label the $2n$-periodic self-conjugate cycle of attracting (respectively parabolic) Fatou components of $N_{a}$ as $U_1^a$, $U_2^a$, $\cdots$, $U_{2n}^a$ such that $U_1^a$ contains $c_a$. It will be important to study the topology of the boundaries of these Fatou components. Let us first prove a sharper version of Proposition \ref{poles} on the location of the poles of $N_a$.

\begin{proposition}[Poles and Boundaries of Fixed Immediate Basins]\label{conjugate_poles}
Let $H$ be a Tricorn component of period $2n$, and $a$ be some parameter in $H$ (or some simple parabolic parameter on $\partial H$). Then $p_a$ is the unique pole on $\partial\mathcal{B}^{\mathrm{imm}}_{1}$ (respectively on $\partial\mathcal{B}^{\mathrm{imm}}_{-1}$). On the other hand, the boundaries of $\mathcal{B}^{\mathrm{imm}}_{a}$ and $\mathcal{B}^{\mathrm{imm}}_{\bar{a}}$ contain exactly one pole (necessarily non-real) each. In particular, $p_a\notin\partial\mathcal{B}^{\mathrm{imm}}_{a}\cup\partial\mathcal{B}^{\mathrm{imm}}_{\bar{a}}$.
\end{proposition}
\begin{proof}
Note that by our choice of $a$, the closure of the immediate basin $\mathcal{B}^{\mathrm{imm}}_{1}$ (respectively $\mathcal{B}^{\mathrm{imm}}_{-1}$) contains exactly one critical point of $N_a$. Clearly, $N_a:\partial\mathcal{B}^{\mathrm{imm}}_{1}\to\partial\mathcal{B}^{\mathrm{imm}}_{1}$ is a degree two covering map. Since $\infty$ is a fixed point of $N_a$, it follows that $\partial\mathcal{B}^{\mathrm{imm}}_{1}$ (respectively $\partial\mathcal{B}^{\mathrm{imm}}_{-1}$) contains exactly one pole of $N_a$. It follows from Proposition \ref{basins_prop} that this pole must be $p_a$.

Now we turn our attention to the other two fixed immediate basins. Since both the free critical points are in Fatou set of $N_a$, it follows that $N_a:\partial\mathcal{B}^{\mathrm{imm}}_{a}\to\partial\mathcal{B}^{\mathrm{imm}}_{a}$ (respectively, $N_a:\partial\mathcal{B}^{\mathrm{imm}}_{\bar{a}}\to\partial\mathcal{B}^{\mathrm{imm}}_{\bar{a}}$) is a degree two covering map. So $\partial\mathcal{B}^{\mathrm{imm}}_{a}$ (respectively, $\partial\mathcal{B}^{\mathrm{imm}}_{\bar{a}}$) contains exactly one pole of $N_a$. Let us assume that this unique pole is $p_a$. But then, all four fixed immediate basins would touch at $p_a$. However, this contradicts the fact that $N_a$ induces a local orientation-preserving diffeomorphism from a neighborhood of $p_a$ to a neighborhood of $\infty$ (note that $p_a$ is not a critical point). Therefore by Proposition \ref{poles}, the unique pole on $\partial\mathcal{B}^{\mathrm{imm}}_{a}$ (respectively on $\partial\mathcal{B}^{\mathrm{imm}}_{\bar{a}}$) must be non-real.
\end{proof}

We observed in Section \ref{hyperbolic} that each $U_i^a$ is fixed by the degree $2$ antiholomorphic map $\iota\circ N_a^{\circ n}$; i.e. $\iota\circ N_{a}^{\circ n}$ is the first antiholomorphic return map of each $U_i^a$. In order to conclude that there are exactly three fixed points of $\iota\circ N_{a}^{\circ n}$ on the boundary of each $U_i^a$, we need to prove the following proposition.

\begin{proposition}\label{jordan}
Let $H$ be a Tricorn component of period $2n$, and $a$ be some parameter in $H$ (or some simple parabolic parameter on $\partial H$). Then the boundary $\partial U_i^a$ of each $2n$-periodic Fatou component of $N_a$ is a Jordan curve.
\end{proposition}
\begin{proof}
The proof follows the arguments of \cite[Theorem 6.13]{BBM1}, so we only give a sketch. However, we need to address a technical difference: our maps $N_a$ commute with $\iota$ (which has a circle of fixed points), as opposed to the maps considered in \cite{BBM1} which commute with the (fixed-point free) antipodal map.

Since Julia sets of Newton maps arising from polynomials are always connected and $N_a$ is geometrically finite, it follows by \cite[Theorem A]{Tan} that the Julia set of $N_a$ is locally connected.

Note that since $\infty$ is a repelling fixed point of $N_a$, it does not lie in the Fatou component $U_i^a$. Moreover, if $\infty$ belonged to $\partial U_i^a$, then $N_a$ would fail to be a local orientation-preserving diffeomorphism on a neighborhood of $\infty$. Therefore, $U_i^a$ is a bounded subset of the plane. Let $V$ be the unique unbounded component of $\C\setminus\overline{U_i^a}$. It follows by the arguments used in \cite[Theorem 6.13]{BBM1} that $U':=\C\setminus\overline{V}$ has a Jordan curve boundary. It now suffices to show that $U_i^a=U'$.

Evidently, $U_i^a\subset U'$. Let us assume that $U_i^a\subsetneq U'$, and $Y$ is a connected component of $U'\setminus U_i^a$. By construction, we have $\partial Y\subset\partial U_i^a$. Since the closure of each of the $2n$-periodic Fatou components of $N_a$ is disjoint from $\R\cup\{\infty\}$ (compare Proposition \ref{basins_prop}), it follows that no iterated forward image of $\partial Y$ intersects $\R\cup\{\infty\}$.

Since $\partial Y$ is contained in the Julia set of $N_a$, it is easy to see that $Y\cup U_i^a$ is an open set containing a Julia point. Therefore, $Y\cup U_i^a$ must contain some iterated pre-images of $\infty$ (which belongs to the Julia set). Since $U_i^a$ is contained in the Fatou set and no iterated forward image of $\partial Y$ intersects $\R\cup\{\infty\}$, we conclude that $\mathrm{int}(Y)$ contains an iterated pre-image of $\infty$.

Note that $\mathrm{int}(Y)$ is a bounded open set. Hence the orbit of $\mathrm{int}(Y)$ must hit one of the poles of $N_a$ before hitting $\infty$. More precisely, there exists some $j\geq 0$ such that $N_a^{\circ j}(\mathrm{int}(Y))$ is a bounded open set containing a pole of $N_a$.

If $p_a\in N_a^{\circ j}(\mathrm{int}(Y))$, then $N_a^{\circ j}(\partial Y)$ must intersect $\R$ (since $p_a$ lies on the real line) yielding a contradiction. If one of the two non-real poles of $N_a$ lies in $N_a^{\circ j}(\mathrm{int}(Y))$, then $N_a^{\circ j}(\partial Y)$ must intersect one of the fixed immediate basins $\mathcal{B}^{\mathrm{imm}}_{a}$ or $\mathcal{B}^{\mathrm{imm}}_{\bar{a}}$ (note that by Proposition \ref{conjugate_poles}, one of the non-real poles of $N_a$ lies on $\partial\mathcal{B}^{\mathrm{imm}}_{a}$ and the other on $\partial\mathcal{B}^{\mathrm{imm}}_{\bar{a}}$). However, this is impossible since $N_a^{\circ j}(\partial Y)$ is contained in the Julia set of $N_a$.

This proves that $U_i^a=U'$, and hence $\partial U_i^a$ is a Jordan curve.
\end{proof}

We denote the three distinct fixed points of $\iota\circ N_a^{\circ n}$ on $\partial U_i^a$ by $p_1(U_i^a)$, $p_2(U_i^a)$, and $p_3(U_i^a)$. Moreover, $p_1(U_i^a)$, $p_2(U_i^a)$, and $p_3(U_i^a)$ are also precisely the fixed points of the first holomorphic return map $N_{a}^{\circ 2n}$ on $\partial U_i^a$ (as the return map has degree $4$).

\begin{definition}[Roots and Co-roots]\label{root_co_root_dyn}
By definition, $p_k(U_i^a)$ is a dynamical \emph{root point} if there exists $i'\in \{1, 2, \cdots, 2n\}$ with $i'\neq i$ and $k'\in\{1,2,3\}$ such that $p_{k'}(U_{i'}^a)=p_k(U_i^a)$; i.e. if two distinct $2n$-periodic Fatou components touch at $p_k(U_i^a)$. In this case, two of the three self-conjugate cycles $\{p_k(U_i^a)\}_i$ ($k=1, 2, 3$) coincide.

Otherwise, $p_k(U_i^a)$ is called a dynamical \emph{co-root}. In this case, all three self-conjugate cycles $\displaystyle\{p_k(U_i^a)\}_{i=1}^{2n}$ (for $k=1, 2, 3$) are disjoint.
\end{definition}

\begin{proposition}\label{disjoint}
Let $H$ be a Tricorn component of period $2n$, and $a$ be a parameter in $H$ or on a parabolic arc on $\partial H$. For each $i\in\{1,2,\cdots,2n\}$ and $k\in\{1,2,3\}$, $p_k(U_i^a)$ is a dynamical co-root.
\end{proposition}

\begin{proof}
Let us fix $a$ in $H$ or on a parabolic arc on $\partial H$. We will drop the superscript $`a\textrm'$, and denote the cycle of attracting/parabolic Fatou components simply by  $U_1$, $U_2$, $\cdots$, $U_{2n}$.

By way of contradiction, suppose that the two components $U_i$ and $U_j$ touch at a dynamical root $p_k(U_i)$. Let $U_{i_1},\ U_{i_2},\ \cdots,\ U_{i_r}$ be all the $2n$-periodic components touching at $p_k(U_i)$. Note that $\iota \circ N_{a}^{\circ n}$ is a local orientation-reversing diffeomorphism from a neighborhood of $p_k(U_i)$ to a neighborhood of $p_k(U_i)$. But if $r\geq 3$, then it would preserve the cyclic order of the Fatou components touching at $p_k(U_i)$. This contradiction proves that $r\leq 2$; i.e. $U_i$ and $U_j$ must be the only Fatou components (among $U_{1},\ U_{2},\ \cdots,\ U_{2n}$) that touch at $p_k(U_i)$. Now let $N_{a}^{\circ l}(U_i)=U_{j}$. We claim that there are at least two dynamical roots on $\partial U_{j}$. Otherwise $N_{a}^{\circ l}(p_k(U_i))=p_k(U_i)$ by the condition that there is exactly one root on $\partial U_{j}$. Clearly, $l\vert 2n$; i.e. $l\leq n$. If $l<n$, then more than two $2n$-periodic Fatou components would touch at $p_k(U_i)$, a contradiction. Hence $l=n$, and $N_{a}^{\circ n}(p_k(U_i))=p_k(U_i)$. But by definition, $p_k(U_i)$ is a fixed point of $\iota\circ N_{a}^{\circ n}$; i.e. $N_{a}^{\circ n}(p_k(U_i))= \iota(p_k(U_i))$. This implies that $p_k(U_i)=\iota(p_k(U_i))$. By Proposition \ref{basins_prop}, the only periodic Julia point that is fixed by $\iota$ is $\infty$.  Thus, $p_k(U_i)=\infty$. But this will contradict the fact that $\iota \circ N_a^{\circ n}$ is a local orientation-reversing diffeomorphism in a neighborhood of $\infty$. This proves the claim that there are at least two dynamical roots (which lie on the same cycle) on $\partial U_{j}$. It follows that there are at least two roots (which lie on the same cycle) on $\partial U_{m}$ for all $m=1, 2, \cdots, 2n$. This implies that (the closures of) the cycle of immediate basins $U_1$, $U_2$, $\cdots$, $U_{2n}$ form a ring such that each component touches its neighboring two components at $2n$-periodic points. However, as $n$ of them lie in the upper half-plane and $n$ of them lie in the lower half-plane, it follows that $\displaystyle\bigcup_{i=1}^{2n} \overline{U_i}$ must intersect the real line in at least two distinct $2n$-periodic points. But this is impossible by Proposition~\ref{basins_prop}. This contradicts our initial assumption that $p_k(U_i)$ is a dynamical root, and proves the proposition.
\end{proof}

\begin{corollary}
Under the assumption of Proposition~\ref{disjoint}, the three self-conjugate $2n$-cycles $\displaystyle\{p_k(U_i^a)\}_{i=1}^{2n}$ (for $k=1, 2, 3$) are disjoint.
\end{corollary}

By Lemma \ref{LemTricornIndiffDyn} and Proposition \ref{parabolic arcs}, the boundary of a Tricorn component consists entirely of parabolic arcs and cusp points. Note that the dynamical co-roots $p_1(U_i^a)$, $p_2(U_i^a)$, and $p_3(U_i^a)$ can be followed continuously as fixed points of $\iota \circ N_{a}^{\circ n}$ (on the boundary of $U_i$) throughout the union of $H$ and the parabolic arcs on $\partial H$. On each parabolic arc on $\partial H$, the unique self-conjugate attracting cycle merges with the self-conjugate repelling cycle $\{p_k(U_i^a)\}_{i=1}^{2n}$, for some fixed $k\in \{1, 2, 3\}$, forming a simple parabolic cycle. Therefore, there are three distinct ways in which a simple parabolic cycle can be formed on $\partial H$. It follows that there are three parabolic arcs $\cC_1$, $\cC_2$ and $\cC_3$ on $\partial H$ satisfying the property that the self-conjugate parabolic cycle of any $a\in \cC_k$  (where $k\in \{1, 2, 3\}$)  is formed by the merger of the self-conjugate attracting cycle  with the self-conjugate repelling cycle $\{p_k(U_i^a)\}_{i=1}^{2n}$. Finally, the cusp points on $\partial H$ are characterized by the merger of two of the three self-conjugate cycles $\{p_k(U_i^a)\}_i$ (for $k=1, 2, 3$) along with the unique self-conjugate attracting cycle. This proves that:

\begin{theorem}[Boundaries of Tricorn Components]\label{Tricorn_boundary}
The boundary of every Tricorn component consists of three parabolic arcs accumulating at cusp points.
\end{theorem}

\bigskip

\section{Constructing centers of tricorn components}\label{sec:ConstructingPCFmaps}
In this section we will apply W. Thurston's characterization of rational maps to prove the existence of infinitely many postcritically finite Newton maps that satisfy certain dynamical properties. A complete combinatorial invariant for postcritically finite Newton maps was given in \cite{LMS1, LMS2}, but for our purposes it is easier to explicitly construct topological branched covers using a subdivision rule, and then apply the ``arcs intersection obstructions'' of \cite{PT} to show that no obstructing multicurves exist in the sense of Thurston. It immediately follows from a slight generalization of Thurston's theorem that the branched cover is equivalent to a rational map that is unique up to M\"obius conjugacy \cite{DH93,BCT}. We then use a result of Pilgrim to show that the Fatou components containing the critical $2n$-cycle are visible from the immediate basins containing $\pm 1$.

A \emph{marked branched cover} is a pair $(f,X)$ where $f:\mathbb{S}^2\to\mathbb{S}^2$ is an orientation-preserving branched cover with topological degree greater than one, and $X\subset \mathbb{S}^2$ is a finite set (called the \emph{marked set}) that contains all critical values of $f$ and satisfies $f(X)\subset X$. Two marked covers $(f,X)$ and $(g,Y)$ are said to be \emph{equivalent} if there are two orientation-preserving homeomorphisms $h_0,h_1:(\mathbb{S}^2,X)\to(\mathbb{S}^2,Y)$ so that $h_0\circ f = g\circ h_1$ where $h_0$ and $h_1$ are homotopic relative to $X$. If furthermore $X=Y$ and $h_0$ and $h_1$ are homotopic to the identity, $f$ and $g$ are said to be \emph{homotopic}.

A simple closed curve $\gamma$ in $\mathbb{S}^2\setminus X$ is said to be \emph{essential} if it does not bound a disk or a punctured disk. A \emph{multicurve in $(\mathbb{S}^2,X)$} is a finite collection of simple closed essential curves in $\mathbb{S}^2\setminus X$ that are pairwise disjoint and non-homotopic. Homotopies of essential curves and multicurves are taken in $\mathbb{S}^2\setminus X$. Thurston's characterization is given in terms of a special kind of multicurve called an irreducible obstructing multicurve. The precise definition is not important for our discussion, so the reader is referred to \cite[\S 3]{PT}.

W. Thurston's theorem is now stated for marked covers, a mild generalization of the original result \cite{DH93}. Our maps will always have more than five postcritical points and will hence have hyperbolic orbifold (the orbifold of a marked cover $(f,X)$ is simply taken to be the orbifold associated to $(f,P_f)$ in the usual sense of \cite{DH93}).

\begin{theorem}[W. Thurston]\label{thm:Thurston}
Suppose that $(f,X)$ is a marked cover with hyperbolic orbifold. Then $(f,X)$ is equivalent to a marked rational map if and only if there is no irreducible obstructing multicurve. The rational map, if it exists, is unique up to M\"obius conjugacy.
\end{theorem}

In practice, it can be difficult to apply Thurston's theorem since there are typically infinitely many multicurves in the complement of the marked set. We use a special case of the ``arcs intersecting obstructions" theorem to show that obstructions do not exist.

An \emph{arc} in $(\mathbb{S}^2,X)$ is an injective map $\lambda:[0,1]\to\mathbb{S}^2$ where $\lambda(0),\lambda(1)\in X$ and $\lambda((0,1))\cap X = \emptyset$.  For a marked cover $(f,X)$,  we say that the arc $\lambda$ is \emph{periodic} if there is some $n$ so that $f^{\circ n}$ maps $\lambda ([0,1])$ homeomorphically to itself. For such a periodic arc $\lambda$, we denote $f^{\circ i}(\lambda)$ by $\lambda_i$ and call $\Lambda:=\{\lambda_0,...,\lambda_{n-1}\}$ an \emph{invariant arc system} (only invariance up to homotopy is required for the original statement of the following theorem, but we specialize for convenience). If $\Gamma$ is a multicurve, we denote by $\Gamma \cdot\Lambda$ the minimum of $|\Gamma'\cap\Lambda|$ over all curves $\Gamma'$ homotopic to $\Gamma$. Let $\tilde{\Gamma}(f^{\circ k})$ denote the union of those components of $f^{-k}(\Gamma)$ that are isotopic in $\mathbb{S}^2\setminus X$ to elements of $\Gamma$.

We have the following specialized case of \cite[Theorem 3.2]{PT}.

\begin{theorem}[Arcs intersecting obstructions]\label{thm:ArcsIntersectingObstructions}
Let $(f,X)$ be a marked cover, $\Gamma$ an irreducible obstruction, and $\Lambda$ an invariant arc system. Suppose further that $|\Gamma\cap\Lambda|=\Gamma\cdot\Lambda$. Then for each $k\geq 1$, either
\begin{enumerate}
\item $\Gamma\cdot\Lambda=0$, and $\Gamma\cdot f^{-k}(\Lambda)=0$, or
\item  $\Gamma\cdot\Lambda\neq 0$ and for each $\lambda\in\Lambda$, there is exactly one connected component $\lambda'$ of $f^{-k}(\lambda)$ so that $\lambda'\cap\tilde{\Gamma}(f^{\circ k})\neq\emptyset$. Moreover, the arc $\lambda'$ is the unique component of $f^{-k}(\lambda)$ that is an element of $\Lambda$.
\end{enumerate}
\end{theorem}

\begin{proposition}[Newton maps with visibility]\label{thm_NewtonExists}
There is a sequence of complex parameters $\{a_n\}_{n>1}$ so that $N_{a_n}\in\mathcal{N}^*_4$ and $N_{a_n}$ has a superattracting cycle of period $2n$ with accesses from the Fatou components containing $\pm 1$. Moreover, each $a_n$ belongs to the symmetry locus $\mathcal{S}$.
\end{proposition}

\begin{figure}
\begin{center}
\includegraphics[scale=0.2]{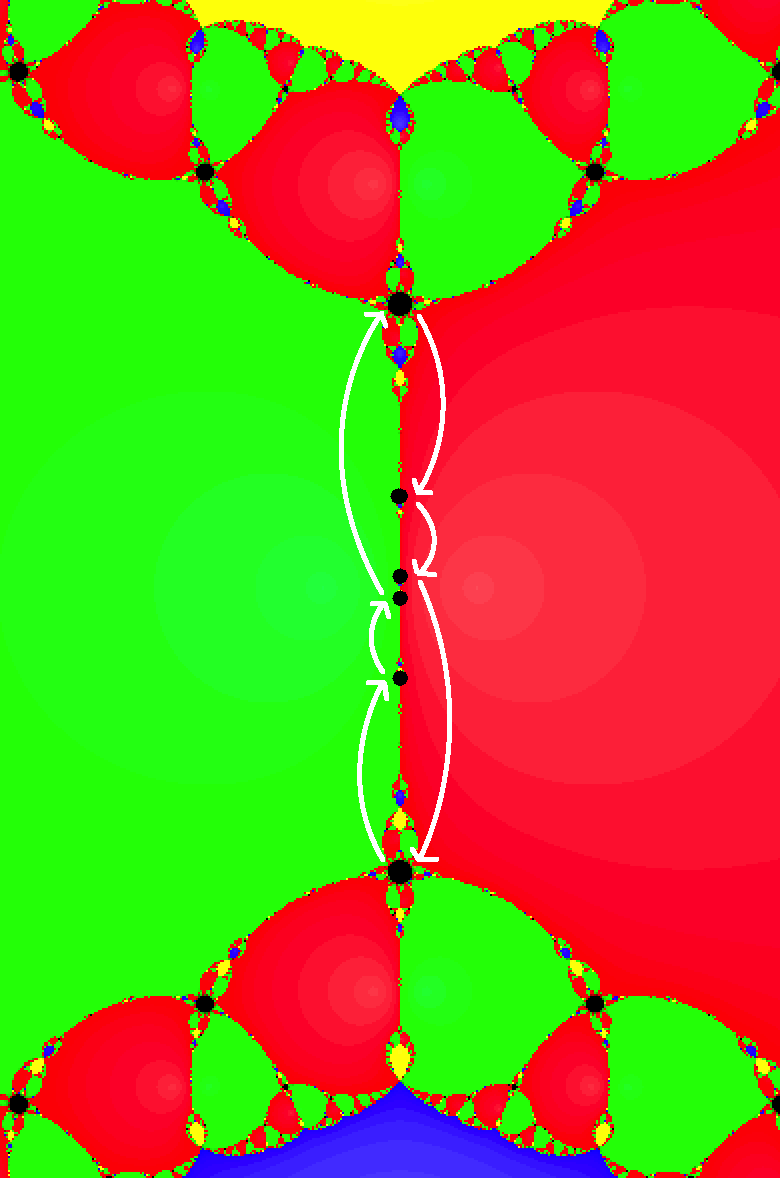}
\caption{The dynamical plane of the map in $\mathcal{N}_4^*$ corresponding to $n=3$ from Proposition \ref{thm_NewtonExists}. The basins of $1$ and $-1$ are colored red and green. The non-fixed critical points are in a period 6 cycle indicated by the white arrows.}
\label{period_6_example}
\end{center}
\end{figure}

\begin{figure}
\begin{tikzpicture}
    \node[anchor=south west,inner sep=0] at (0,0) {\includegraphics[width=\textwidth]{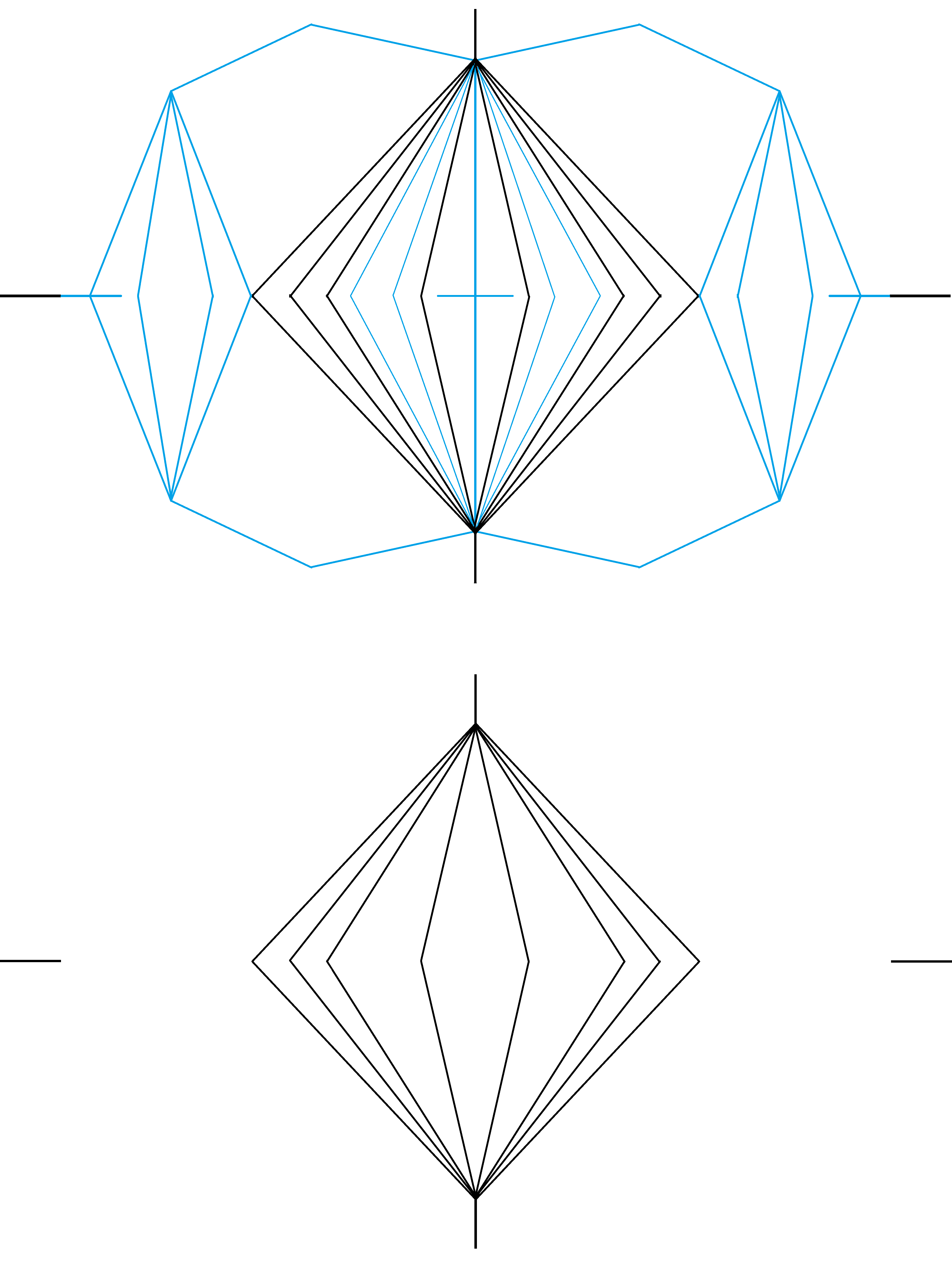}};

\node at (6.95,15.75) {$(0,1)$};
\node at (7,9.75) {$(0,-1)$};

\node at (6.85,7.15) {$(0,1)$};
\node at (6.95,.75) {$(0,-1)$};

\node at (.43,12.6) {$\gamma$};
\node at (12.3,12.6) {$\delta$};
\node at (6.18,16.3) {$\alpha$};
\node at (6.18,9.25) {$\beta$};

\node at (1,12.6) {$\gamma$};
\node at (1.5,12.65) {$\delta$};

\node at (1.5,14) {$\beta$};
\node at (2,13.65) {$\beta_{\text{-}1}$};
\node at (2.7,13.65) {$\beta_3$};
\node at (3,14) {$\beta_2$};

\node at (1.5,11.6) {$\alpha$};
\node at (2,11.9) {$\alpha_{\text{-}1}$};
\node at (2.7,11.9) {$\alpha_3$};
\node at (3,11.6) {$\alpha_2$};

\node at (11.63,12.62) {$\delta$};
\node at (11.2,12.65) {$\gamma$};

\node at (9.6,14) {$\beta_{\text{-}2}$};
\node at (10.18,13.65) {$\beta_{\text{-}3}$};
\node at (10.9,13.65) {$\beta_1$};
\node at (11.2,14) {$\beta$};

\node at (9.5,11.6) {$\alpha_{\text{-}2}$};
\node at (10.2,11.9) {$\alpha_{\text{-}3}$};
\node at (10.9,11.9) {$\alpha_1$};
\node at (11.2,11.6) {$\alpha$};

\node at (5.2,16.35) {$\alpha_1$};
\node at (3.15,16.2) {$\beta_1$};

\node at (7.2,16.3) {$\alpha_{\text{-}1}$};
\node at (9.5,16.2) {$\beta_{\text{-}1}$};

\node at (5.2,9.3) {$\beta_1$};
\node at (3.15,9.45) {$\alpha_1$};

\node at (7.2,9.3) {$\beta_{\text{-}1}$};
\node at (9.5,9.45) {$\alpha_{\text{-}1}$};

\node at (6.27,14.2) {$\alpha$};
\node at (6.27,11.55) {$\beta$};
\node at (6.05,12.8) {$\delta$};
\node at (6.6,12.8) {$\gamma$};

\node at (3.75,12.15) {$\beta_2$};
\node at (4.2,12.35) {$\beta_3$};
\node at (4.6,12.45) {$\beta_{n}$};
\node at (5.05,12.48) {$\beta_{\textit{-}n}$};
\node at (5.5,12.35) {$\beta_{\textit{-}2}$};
\node at (5.9,12.1) {$\beta_{\text{-}1}$};

\node at (8.85,12.1){$\beta_{\text{-}2}$};
\node at (8.5,12.35){$\beta_{\text{-}3}$};
\node at (8.15,12.55) {$\beta_{\text{-}n}$};
\node at (6.9,12.1) {$\beta_{1}$};
\node at (7.75,12.48) {$\beta_{n}$};
\node at (7.3,12.35) {$\beta_{2}$};

\node at (3.75,13.55) {$\alpha_2$};
\node at (4.2,13.35) {$\alpha_3$};
\node at (4.6,13.25) {$\alpha_{n}$};
\node at (5.05,13.15) {$\alpha_{\textit{-}n}$};
\node at (5.5,13.25) {$\alpha_{\textit{-}2}$};
\node at (5.9,13.55) {$\alpha_{\text{-}1}$};

\node at (8.15,13.25){$\alpha_{\text{-}n}$};
\node at (8.5,13.35){$\alpha_{\text{-}3}$};
\node at (8.9,13.55){$\alpha_{\text{-}2}$};
\node at (7.75,13.15) {$\alpha_{n}$};
\node at (7.25,13.25) {$\alpha_{2}$};
\node at (6.9,13.55){$\alpha_{1}$};

\foreach \Point in {(.74,12.82), (3.37,12.82), (4.38,12.82), (6.34,15.95), (6.34,9.67), (8.3,12.82),(9.3,12.82), (11.85,12.82), (3.9,12.82),(8.8,12.82),(5.62,12.82),(7.02,12.82)}{
    \node at \Point {\textbullet};
}

\foreach \Point in {(4.04,12.82),(4.14,12.82),(4.24,12.82)}{
    \node at \Point {$\cdot$};
}

\foreach \Point in {(4.84,12.82),(4.94,12.82),(5.04,12.82)}{
    \node at \Point {$\cdot$};
}

\foreach \Point in {(7.54,12.82), (7.74,12.82), (7.64,12.82)}{
    \node at \Point {$\cdot$};
}

\foreach \Point in {(8.64,12.82), (8.44,12.82), (8.54,12.82)}{
    \node at \Point {$\cdot$};
}

\foreach \Point in {(2.15,12.82), (2.35,12.82), (2.55,12.82)}{
    \node at \Point {$\cdot$};
}

\foreach \Point in {((10.1,12.82), (10.3,12.82), (10.5,12.82)}{
    \node at \Point {$\cdot$};
}

\node at (.43,3.74) {$\gamma$};
\node at (12.3,3.74) {$\delta$};
\node at (6.17,7.4) {$\alpha$};
\node at (6.17,.4) {$\beta$};

\node at (3.75,3.3) {$\beta_1$};
\node at (4.2,3.5) {$\beta_2$};
\node at (4.85,3.6) {$\beta_{n-1}$};
\node at (5.85,3.7) {$\beta_n$};

\node at (8.9,3.2){$\beta_{\text{-}1}$};
\node at (8.4,3.4){$\beta_{\text{-}2}$};
\node at (7.85,3.5) {$\beta_{1\text{-}n}$};
\node at (7.2,3.6) {$\beta_{\text{-}n}$};

\node at (3.75,4.7) {$\alpha_1$};
\node at (4.2,4.5) {$\alpha_2$};
\node at (4.85,4.4) {$\alpha_{n-1}$};
\node at (5.9,4.3) {$\alpha_{n}$};

\node at (7.2,4.3){$\alpha_{\text{-}n}$};
\node at (7.95,4.4){$\alpha_{1\text{-}n}$};
\node at (8.4,4.5){$\alpha_{\text{-}2}$};
\node at (8.9,4.7){$\alpha_{\text{-}1}$};

\foreach \Point in {(.74,3.97), (3.4,3.97), (3.9,3.97), (5.6,3.97), (4.4,3.97), (6.34,7.09), (6.34,.85),(7,3.97), (8.28,3.97),(8.78,3.97),(9.28,3.97), (11.85,3.97)}{
    \node at \Point {\textbullet};
}

\foreach \Point in {(4.05,3.97), (4.15,3.97), (4.25,3.97)}{
    \node at \Point {$\cdot$};
}

\foreach \Point in {(8.6,3.97), (8.5,3.97), (8.4,3.97)}{
    \node at \Point {$\cdot$};
}

\end{tikzpicture}
\caption{Subdivision rule that defines the sequence of critically finite branched covers $f_{n}:\sphere\to\sphere$ where $n>1$. Black edges $\alpha,\beta,\gamma,\delta$ terminate at $\infty$ in both the domain (top) and the range (bottom).  Light blue edges in the domain represent preimage edges of $f_{n}$ that are not edges in the range. Both domain and range graphs are symmetric about the horizontal and vertical coordinate axes.}
\label{fig:subdivision}
\end{figure}

\begin{proof}
The subdivision rule exhibited in Figure \ref{fig:subdivision} defines a sequence of critically finite branched covers $f_{n}:\sphere\to\sphere$ for $n\geq 2$ as follows.  The upper graph represents the domain sphere as the one-point compactification of $\mathbb{R}^2$, and the lower graph represents the range sphere similarly.  Each vertex in the domain maps to a vertex in the range, each edge in the domain maps homeomorphically to the unique edge with the same label in the range, and complementary components of the domain graph are mapped to complementary components of the range graph. All mappings are chosen to respect symmetry over both coordinate axes. This defines $f_{n}$ uniquely up to homotopy relative to the marked set $X:=P_{f_{n}}\cup\{\infty\}$, where the postcritical set $P_{f_n}$ consists of $4+2n$ points. Furthermore, $f_n$ commutes with the homeomorphisms of $\sphere$ induced by $R_y(x,y):=(-x,y)$, $R_x(x,y):=(x,-y)$, and $R_x\circ R_y$.

From here, the label of an edge will be taken to be the label assigned in the range. It is easy to see that the edges labeled $\alpha,\beta,\gamma,$ and $\delta$ are fixed by $f_n$, and edges of the form $\alpha_i$ are permuted as follows:
\begin{align*}
\alpha_1\mapsto\alpha_2\mapsto\alpha_3\mapsto...\mapsto\alpha_{n-2}\mapsto\alpha_{n-1}\mapsto\alpha_{n}\mapsto\alpha_{-1},\\
\alpha_{-1}\mapsto\alpha_{-2}\mapsto\alpha_{-3}\mapsto...\mapsto\alpha_{2-n}\mapsto\alpha_{1-n}\mapsto\alpha_{-n}\mapsto\alpha_{1}.\\
\end{align*}
Furthermore, $f_{n}$ acts similarly on indices of the form $\beta_i$. Thus each of the following six sets forms irreducible arc systems:
\[\{\alpha\},\{\beta\},\{\gamma\},\{\delta\},\]
\[\{\alpha_{-n},...,\alpha_{n}\},\{\beta_{-n},...,\beta_{n}\}.\]
Denote by $\Lambda$ the union of these six arc systems.

We prove by contradiction that $f_{n}$ is unobstructed. Suppose that $\Gamma$ is an irreducible obstruction for $f_n$. Without loss of generality assume that $|\Gamma\cap\Lambda|=\Gamma\cdot\Lambda$. Since $\{\alpha\}$ is an invariant arc system, we apply both cases of Theorem~\ref{thm:ArcsIntersectingObstructions} to argue that $\tilde{\Gamma}(f_n^{\circ n})\cap (f_n^{-n}(\alpha)\setminus\alpha)=\emptyset$. Applying the same argument to the five other invariant arc systems in $\Lambda$, we see that $\tilde{\Gamma}(f_n^{\circ n})\cap(f_n^{-n}(\Lambda)\setminus\Lambda)=\emptyset$. By inspection of Figure \ref{fig:subdivision}, it is seen that the closure of $f_n^{-n}(\Lambda)\setminus\Lambda$ has the property that each complementary component contains at most one marked point. Thus every component of the multicurve $\tilde{\Gamma}(f_n^{\circ n})$ (and consequently $\Gamma$ by irreducibility) bounds a disk or a once-punctured disk in $\sphere\setminus X$, and so $\Gamma$ is not an obstruction contrary to assumption. Theorem~\ref{thm:Thurston} implies that $f_{n}$ is equivalent to a rational map which we call $g_{n}$.

To show that $g_n$ is symmetric, we choose appropriately symmetric maps to realize the equivalence to $f_n$. This requires some tools used in the proof of Thurston's theorem in the setting of marked branched covers \cite{BCT}, which is proven by iteration on the Teichm\"uller space
\[\mathcal{T}_X=\{\phi:\sphere\to\mathbb{\C}:\phi\text{ an orientation-preserving homeomorphism}\}/\sim\]
where $\phi,\psi:\sphere\to\mathbb{\C}$ are equivalent (denoted $\phi\sim\psi$) precisely when there is a M\"obius transformation $M$ so that $M\circ\psi|_X=\phi|_X$ and $M\circ\psi = \phi\circ h$ where $h$ is a homeomorphism isotopic to the identity relative to $X$.   Let $[\tau]\in\mathcal{T}_X$. Pull back the complex structure of $\widehat{\C}$ under the map $\tau\circ f_n:\sphere\to\widehat{\C}$ to give a complex structure on the domain $\sphere$ and choose some $\tilde{\tau}:\sphere\to\widehat{\C}$ that uniformizes this pulled back complex structure. Then the pullback map $\sigma_{f_n}:\mathcal{T}_X\to\mathcal{T}_X$  is defined by $\sigma_{f_n}([\tau])=[\tilde{\tau}]$. It is shown in \cite[\S 1.3]{BCT} that $\sigma_{f_n}$ is well-defined and holomorphic.

Let $\phi_0:\sphere\to\widehat{\C}$ be the homeomorphism induced by $\phi_0(x,y)=y-ix$. Note that $f_n$ commutes with $R_x$. Thus the complex structure pulled back under $\phi_0\circ f_n$ is invariant under $R_x$. Let $\phi_1:\sphere\to\widehat{\C}$ be the unique uniformizing map so that $(0,\pm 1)\mapsto\pm 1$ and $\infty\mapsto\infty$. It follows that $\phi_1\circ R_x = -\iota\circ\phi_1$ and so $\phi_1(\R)\subset i\R$. Evidently $\sigma_{f_n}([\phi_0])=[\phi_1]$. Similarly $\phi_1\circ R_y=\iota \circ \phi_1$.

Normalizing in the same way, obtain a sequence $\phi_k:\sphere\to\widehat{\C}$, $k\in\mathbb{Z}^{\geq 0}$ so that $[\phi_k]=\sigma_{f_n}^{\circ k}([\phi_0])$ and $\phi_k(\R)\subset i\R$. By \cite[Theorem 2.2]{BCT}, $[\phi_k]$ converges in the Teichm\"uller metric to a unique point $[\tau]\in\mathcal{T}_X$ fixed by $\sigma_{f_n}$. Thus there are representatives $h_0, h_1\in[\tau]$ that realize an equivalence between $f_n$ and $g_n$ (i.e. they satisfy $h_0\circ f_n=g_n\circ h_1$) and further satisfy the following for $j\in\{0,1\}$:
 \[h_j\circ R_x = -\iota\circ h_j\] \[h_j\circ R_y=\iota \circ h_j\] \[h_j(\R)= i\R.\]
It immediately follows that $g_n$ commutes with $\iota$ and $-\iota$.

We now prove that $g_{n}\in\mathcal{N}^*_4$. Since $f_{n}$ has degree four, it follows that $g_{n}$ has degree four. Due to our normalization, $g_n$ has fixed simple critical points at $\pm 1$, and two other fixed simple critical points in the imaginary axis that are complex conjugate. Denote the unique fixed simple critical point in $\mathbb{H}$ by $a_n$.  Again due to the normalization, $g_n$ has a fixed point at $\infty$. The holomorphic fixed point theorem implies that $\infty$ is repelling with multiplier $4/3$.  Then $g_n(z)$ is the Newton's map associated to the polynomial $(z-1)(z+1)(z-a_n)(z-\bar{a}_n)$ by \cite[Corollary 2.9]{RS}. Since $g_n$ commutes with $\iota$ and $-\iota$, it follows that $g_{n}\in\mathcal{N}_4^*$. More precisely, $g_n=N_{a_n}$ for some $a_n\in\mathcal{S}$.

Since the arcs $\alpha_i$ are permuted transitively by $f_n$, the arcs $h_0(\alpha_{i})$ are permuted by $N_{a_n}$ up to homotopy. Recall that $f_n$ (and hence $N_{a_n}$) has a unique postcritical cycle of length greater than one. Thus each $h_0(\alpha_i)$ connects distinct points in the $N_{a_n}$-orbit of the free critical point to $1$. By \cite[Theorem 5.13]{KMPthesis}, it follows that $\overline{\mathcal{B}_1^{\mathrm{imm}}}$ intersects the closure of the characteristic Fatou component. Apply iterates of $N_{a_n}$ to this access to produce accesses connecting $1$ to any Fatou component in the forward orbit of the characteristic Fatou component. Since $N_{a_n}$ commutes with $-\iota$, the same statement holds for $-1$.
\end{proof}

Recall that by Proposition \ref{jordan}, there are three fixed points of $\iota\circ N_{a_n}^{\circ n}$ on the boundary of the characteristic component $U_1^{a_n}$. As in Section \ref{Tricorn_comp}, we denote these three boundary fixed points by $p_1(U_1^{a_n})$, $p_2(U_1^{a_n})$, and $p_3(U_1^{a_n})$. We will now show that one of these three boundary fixed  points does not lie on $\partial\mathcal{B}_1^{\mathrm{imm}}\cup\partial\mathcal{B}_{-1}^{\mathrm{imm}}$.

\begin{corollary}[Invisible Co-root]\label{inv_co_root}
For each $n>1$, exactly two of the three fixed points of $\iota\circ N_{a_n}^{\circ n}$ on the boundary of $U_1^{a_n}$ lie on $\partial\mathcal{B}_1^{\mathrm{imm}}\cup\partial\mathcal{B}_{-1}^{\mathrm{imm}}$. In other words, $\lbrace p_1(U_1^{a_n}),$ $p_2(U_1^{a_n}),$ $p_3(U_1^{a_n})\rbrace\setminus$ $\left(\partial\mathcal{B}_1^{\mathrm{imm}}\cup\partial\mathcal{B}_{-1}^{\mathrm{imm}}\right)$ is a singleton.
\end{corollary}
\begin{proof}
We already know from Proposition \ref{thm_NewtonExists} that two of the three fixed points of $\iota\circ N_{a_n}^{\circ n}$ on the boundary of $U_1^{a_n}$ lie on $\partial\mathcal{B}_1^{\mathrm{imm}}\cup\partial\mathcal{B}_{-1}^{\mathrm{imm}}$.

Note that since $a_n\in\mathcal{S}$, the critical point $c_{a_n}$ is strictly imaginary. Since $N_{a_n}$ commutes with $-\iota$, it follows that $-\iota$ fixes the characteristic Fatou component $U_1^{a_n}$. Hence, $-\iota$ must act on the set $\lbrace p_1(U_1^{a_n}),$ $p_2(U_1^{a_n}),$ $p_3(U_1^{a_n})\rbrace$ as a permutation of order two. Therefore, we can assume that $p_2(U_1^{a_n})=-\iota(p_1(U_1^{a_n}))$ and $p_3(U_1^{a_n})=-\iota(p_3(U_1^{a_n}))$.

If $p_3(U_1^{a_n})$ lies on the boundary of $\mathcal{B}_1^{\mathrm{imm}}$, then it must also lie on the boundary of $\mathcal{B}_{-1}^{\mathrm{imm}}$ (note that since $N_{a_n}$ commutes with $-\iota$, we have that $-\iota(\mathcal{B}_1^{\mathrm{imm}})=\mathcal{B}_{-1}^{\mathrm{imm}}$). But this would contradict the fact that $\iota\circ N_{a_n}^{\circ n}$ is a local orientation-reversing diffeomorphism on a neighborhood of $p_3(U_1^{a_n})$. Therefore, $p_3(U_1^{a_n})$ does not lie on $\partial\mathcal{B}_1^{\mathrm{imm}}\cup\partial\mathcal{B}_{-1}^{\mathrm{imm}}$. This completes the proof.
\end{proof}

\begin{remark}
According to the terminology of the next section, the dynamical co-roots $p_1(U_1^{a_n})$ and $p_2(U_1^{a_n})$ are \emph{visible}, and the third dynamical co-root $p_3(U_1^{a_n})$ is \emph{invisible}.
\end{remark}

\bigskip

\section{Invisible parabolic points and invisible hyperbolic components}\label{sec_implosion}

In this section, we will employ a parabolic perturbation argument to show how invisible dynamical co-roots imply the existence of invisible hyperbolic components in the parameter plane. Throughout this section, $H$ will stand for a Tricorn component of period $2n$ with center $a_0$.

\subsection{Background on parabolic implosion}\label{para_imp_back}
The main technical tool used in the proof of the main theorems is perturbation of antiholomorphic parabolic points. For details on the concepts of near-parabolic antiholomorphic Fatou coordinates and the transit map, see \cite[\S 2]{IM1}. The technique of perturbation of antiholomorphic parabolic points will allow us to transfer information from the dynamical planes to the parameter plane.

Let us now fix the notations for our parabolic perturbation step. Let $\cC$ be one of the parabolic arcs on $\partial H$. Its critical Ecalle height parametrization is denoted by $a:\mathbb{R}\to\cC$. We will denote an attracting petal of $N_{a(h)}$ by $V_{a(h)}^{\mathrm{in}}$, and a repelling petal of $N_{a(h)}$ by $V_{a(h)}^{\mathrm{out}}$. There exists an open neighborhood $U$ of $a(h)$ (in the parameter plane) such that for all $a \in U^{-}:= U \setminus \overline{H}$, the characteristic parabolic point splits into two simple periodic points, and the perturbed Fatou coordinates can be followed throughout $U^-$. More precisely, for $a \in U^{-}$, there exist an incoming domain $V^\mathrm{in}_{a}\ni N_a^{\circ 2n}(c_a)$, and an outgoing domain $V^\mathrm{out}_{a}$ (such that they are disjoint) having the two simple periodic points on their boundaries. There exists a curve joining the two simple periodic points, which we call the ``gate'', such that the points in the incoming domain eventually transit through the gate, and escape to the outgoing domain (see \cite[Figure 2]{IM1}). For $a\in U^-$, very close to $a(h)$, the point $N_a^{\circ 2n}(c_a)$ (which is contained in the incoming domain $V^\mathrm{in}_{a}$), takes a large number of iterates to escape to the outgoing domain. There exist injective holomorphic maps $\psi_{a}^{\mathrm{in}}: V^\mathrm{in}_{a} \rightarrow \mathbb{C}$ and $\psi_{a}^{\mathrm{out}}: V^\mathrm{out}_{a} \rightarrow \mathbb{C}$ such that $\psi_{a}^{\mathrm{in/out}}(N_{a}^{\circ 2n}(z))=  \psi_{a}^{\mathrm{in/out}}(z) +1$, whenever $z$ and $N_{a}^{\circ 2n}(z)$ are both in $V^\mathrm{in/out}_{a}$. Moreover, with suitable normalizations, the Fatou coordinates $\psi_{a}^{\mathrm{in}}$ and $\psi_{a}^{\mathrm{out}}$ change continuously. It follows that for every $a \in U^{-}$, the quotients $C^\textrm{in}_{a}:= V^{\mathrm{in}}_a/N_a^{\circ 2n}$ and $C^\textrm{out}_{a}:= V^{\mathrm{out}}_a/N_a^{\circ 2n}$ (the quotients of $V^\textrm{in}_{a}$ and $V^\textrm{out}_{a}$ by the dynamics, identifying points that are on the same finite orbits entirely in $V^\textrm{in}_{a}$ or in $V^\textrm{out}_{a}$) are complex annuli isomorphic to $\mathbb{C}/\mathbb{Z}$. The isomorphisms are given by Fatou coordinates which depend continuously on the parameter throughout $U^{-}$.

Since the map $\iota\circ N_{a}^{\circ n}$ commutes with $N_{a}^{\circ 2n}$, it induces antiholomorphic self-maps from $C^\textrm{in}_{a}$ (respectively $C^\textrm{out}_{a}$) to itself. As $\iota\circ N_{a}^{\circ n}$ interchanges the two periodic points at the ends of the gate, it interchanges the ends of the cylinders, so it must fix a (necessarily unique) closed geodesic in the cylinders $\mathbb{C}/\mathbb{Z}$. This is similar to the situation at the parabolic parameter $a(h)$, so we will call this invariant geodesic the equator. As for the parabolic parameter $a(h)$, we will choose our Fatou coordinates such that they map the equators to the real line. Thus we can again define Ecalle height as the imaginary part in these Fatou coordinates. We will denote the Ecalle height of a point $z \in C^\textrm{in/out}_{a}$ by $E(z)$. For $a\in U^{-}$, the incoming and outgoing cylinders are isomorphic to each other by a natural biholomorphism, namely $N^{\circ 2n}_{a}$. This isomorphism is called the ``transit map'', and is denoted by $T_{a}$. The transit map clearly depends continuously on the parameter $a \in U^{-}$. It maps the fixed geodesic of the incoming cylinder to the fixed geodesic of the outgoing cylinder, and preserves the upper (respectively lower) ends of the cylinders. Thus it must preserve Ecalle heights. The existence of this special isomorphism allows us to relate the Ecalle heights of points in the incoming and outgoing cylinders, and is going to be a crucial tool in our study.

We will also need the concept of the \emph{phase}, which determines the conformal position of the escaping critical point $c_a$ in the outgoing cylinder. To do this, we need to fix a normalization of the persistent Fatou coordinates. Following \cite[\S 2]{IM1},  we choose two continuous functions $\xi_1,\xi_2 :\overline{U^{-}} \rightarrow \mathbb{C}$ such that $\xi_1(a)$ (respectively $\xi_2(a)$) lies on the incoming (respectively outgoing) equator in $V^\mathrm{in}_{a}$ (respectively in $V^\mathrm{out}_{a}$), for all $a \in \overline{U^{-}}$. We can also assume that $\psi_{a}^{\mathrm{out}}(V^\mathrm{out}_{a})$ contains the vertical bi-infinite strip $\left[0,1\right]\times \mathbb{R}$. Let us normalize $\psi_{a}^{\mathrm{in}}$ and $\psi_{a}^{\mathrm{out}}$ by the requirements $\psi_{a}^{\mathrm{in}}(\xi_1(a))=0$ and  $\psi_{a}^{\mathrm{out}}(\xi_2(a))=0$. With these normalization, we have that $(a,z) \mapsto \psi_{a}^{\mathrm{in}}(z)$ and $(a,z)\mapsto \psi_{a}^{\mathrm{out}}(z)$ are continuous functions on the open sets $\mathcal{V}^{\mathrm{in}}:=\{ (a,z): z \in V^\mathrm{in}_{a}\}$ and  $\mathcal{V}^{\mathrm{out}}:=\{ (a,z): z \in V^\mathrm{out}_{a}\}$ (respectively) in $\overline{U^{-}} \times \mathbb{C}$ (for $a\in \partial H \cap \overline{U^{-}}$, $\psi_{a}^{\mathrm{in}}$ and $\psi_{a}^{\mathrm{out}}$ are respectively the attracting and repelling Fatou coordinates for $N_{a}$ as in Lemma \ref{normalization of fatou}). For $a\in U^{-}$, let $k_a$ be the smallest positive integer such that $N_a^{2nk_a}(c_a)$ lies in $V^\textrm{out}_{a}$ and satisfies $\re(\psi_{a}^{\mathrm{out}}(N_a^{2nk_a}(c_a))\geq 0$. The integer $k_a$ will be called the \emph{escaping time} of $c_a$. The phase is a continuous map
\begin{equation*}
\phi:U^{-} \to \R/\Z
\end{equation*}
\begin{equation*}
a\mapsto\re(\psi_{a}^{\mathrm{out}}(N_a^{2nk_a}(c_a))).
\end{equation*}
The \emph{lifted phase} is the following continuous lift of $\phi$:
\begin{equation*}
\widetilde{\phi}:U^{-} \to \R
\end{equation*}
\begin{equation*}
a\mapsto\re(\psi_{a}^{\mathrm{out}}(N_a^{2nk_a}(c_a)))-k_a.
\end{equation*}
By \cite[Lemma 2.5]{IM1}, the lifted phase $\widetilde{\phi}(a)$ tends to $-\infty$ as $a$ approaches $\cC$ from $U^{-}$.

\subsection{Visibility in the dynamical and parameter planes}\label{vis_dyn_plane}

Recall that the immediate basins of attraction of the super-attracting fixed points $1, -1, a$, and $\bar{a}$ of $N_a$ are denoted by $\mathcal{B}^{\mathrm{imm}}_{1}, \mathcal{B}^{\mathrm{imm}}_{-1}, \mathcal{B}^{\mathrm{imm}}_{a}$ and $\mathcal{B}^{\mathrm{imm}}_{\bar{a}}$ respectively. We start with a simple lemma.

\begin{lemma}\label{visible}
Let $H$ be a Tricorn component. Let $a$ be a parameter in $H$ or a simple parabolic parameter on $\partial H$. Let $U_1$ be the characteristic Fatou component of $N_{a}$ and $p_k(U_1)$ (for some $k\in\{1, 2, 3\}$) be a dynamical co-root on the boundary of $U_1$. The following two statements about $p_k(U_1)$ are equivalent.
\begin{enumerate}
\item $p_k(U_1)\in\partial\mathcal{B}^{\mathrm{imm}}_{1}\cup\partial\mathcal{B}^{\mathrm{imm}}_{-1}$.

\item $p_k(U_1)$ lies on the boundary of a Fatou component other than $U_1$.

\end{enumerate}
\end{lemma}

\begin{proof}
$\mathbf{1)\implies 2)}$ This is obvious.
\vspace{2mm}

$\mathbf{2)\implies 1)}$ Suppose that $U$ is a Fatou component different from $U_1$ with $p_k(U_1)\in\partial U$. By the classification of Fatou components of rational maps, $U$ eventually maps to a periodic Fatou component.

If $U$ eventually maps to the $2n$-periodic cycle of Fatou components, then $p_k(U_1)$ must lie on the boundary of some $U_i$ with $i\neq1$. But this contradicts the fact that $p_k(U_1)$ is a co-root.

If $U$ eventually maps to $\mathcal{B}^{\mathrm{imm}}_{a}$, then $p_k(U_1)$ lies on the boundary of the fixed Fatou component $\mathcal{B}^{\mathrm{imm}}_{a}$. But $N_{a}^{\circ n}(p_k(U_1))=\overline{p_k(U_1)}$. Since $\mathcal{B}^{\mathrm{imm}}_{a}$ is a fixed Fatou component, this implies that $\overline{p_k(U_1)}\in \partial \mathcal{B}^{\mathrm{imm}}_{a}$. By Proposition \ref{basins_prop}, $\partial \mathcal{B}^{\mathrm{imm}}_{a}$ is disjoint from the lower half-plane. Therefore, we must have that $p_k(U_1)=\infty$. Since $\infty$ is a fixed point, this contradicts Proposition \ref{disjoint}, which states that two distinct $2n$-periodic Fatou components of $N_{a}$ cannot touch at $p_k(U_1)$. Thus $U$ does not eventually map to $\mathcal{B}^{\mathrm{imm}}_{a}$. One can similarly prove that $U$ does not eventually map to $\mathcal{B}^{\mathrm{imm}}_{\bar{a}}$.

Therefore, $U$ eventually maps to the fixed Fatou components $\mathcal{B}^{\mathrm{imm}}_{1}$ or $\mathcal{B}^{\mathrm{imm}}_{-1}$. It follows that $p_k(U_1)\in\partial\mathcal{B}^{\mathrm{imm}}_{1}\cup\partial\mathcal{B}^{\mathrm{imm}}_{-1}$.
\end{proof}

Lemma \ref{visible} leads to the following definition, which is inspired by an analogous definition in \cite{BBM1}.

\begin{definition}[Visibility of Co-roots]\label{dynamical_visibility}
Let $a$ be a parameter in $H$ or on a parabolic arc on $\partial H$. A dynamical co-root of $N_a$ on the boundary of $U_1$ is called \emph{visible} if it satisfies the equivalent conditions of Lemma \ref{visible}. Otherwise, it is called \emph{invisible}.
\end{definition}

\begin{figure}[ht!]
\begin{center}
\includegraphics[scale=0.345]{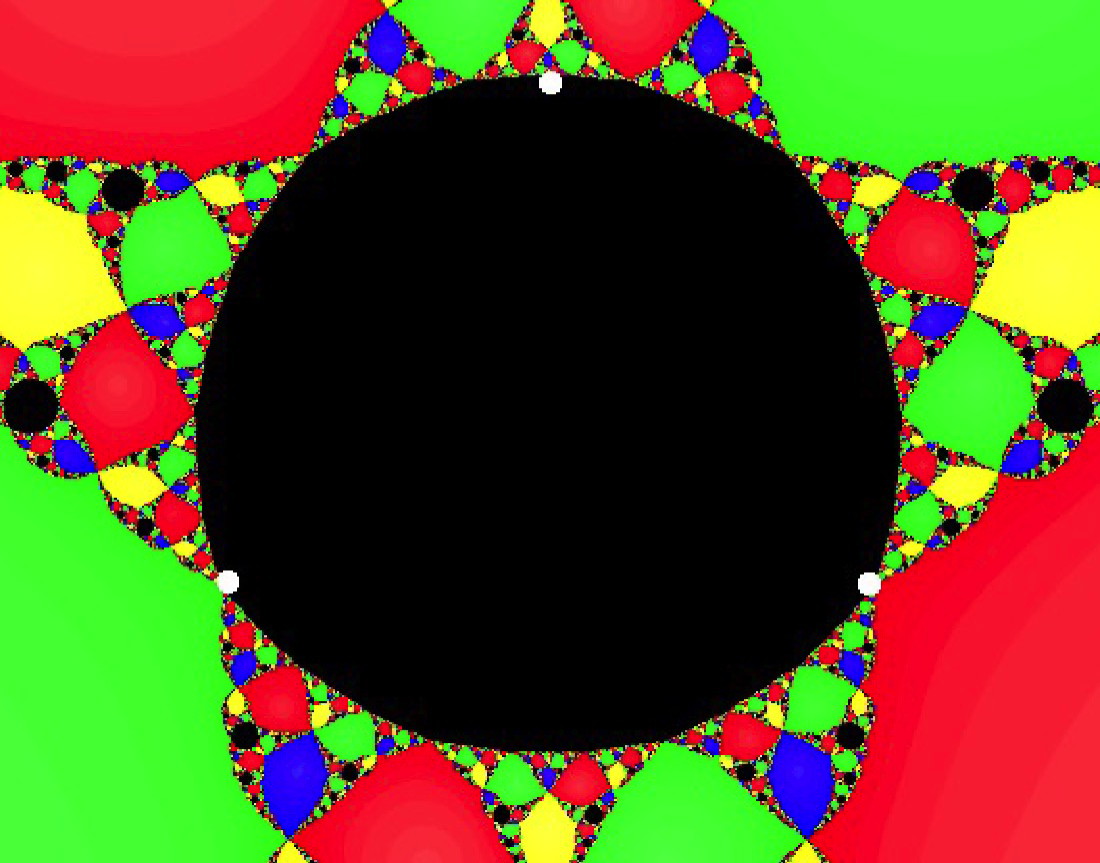}\ \includegraphics[scale=0.28]{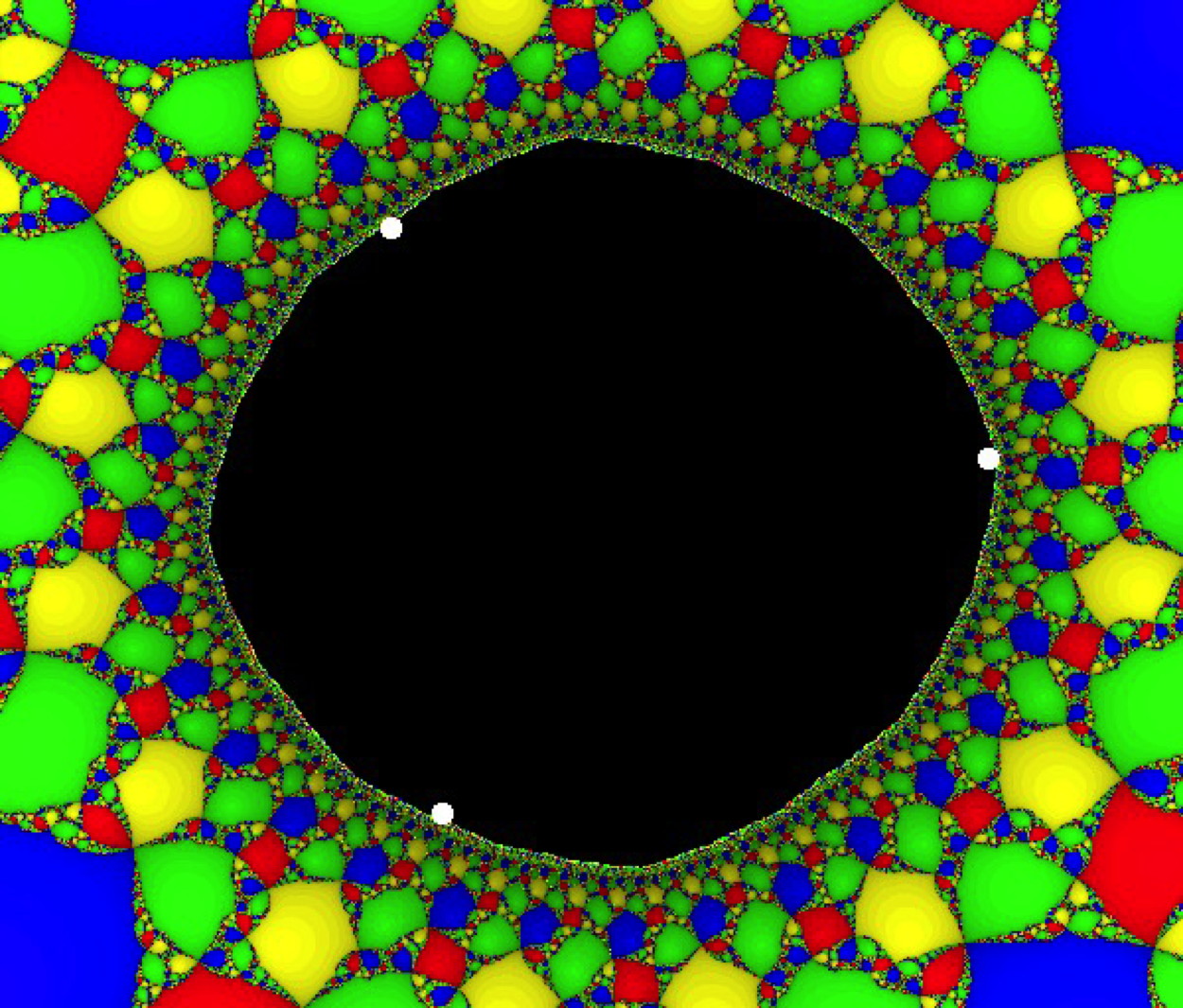}
\caption{Left: The characteristic Fatou component of a parameter belonging to a Tricorn component. The three dynamical co-roots, two of which are visible and one invisible, are marked. Right: The characteristic Fatou component of a parameter belonging to a Tricorn component. The three dynamical co-roots, each of which is invisible, are marked. }
\label{invisible_dynamical}
\end{center}
\end{figure}

Note that the topological structure of the Julia set remains unchanged throughout the union of $H$ and the parabolic arcs of $\partial H$. Therefore, $p_k(U_1)$ (for some $k\in\{1, 2, 3\}$) is visible in the dynamical plane of $N_{a_0}$ (where $a_0$ is the center of the tricorn component $H$) if and only if $p_k(U_1)$ is visible in the dynamical plane of $N_a$ for every $a$ in $H$ or on the parabolic arcs on $\partial H$. So the property of being visible does not depend on the choice of a parameter in the union of a given hyperbolic component and the parabolic arcs on its boundary. By Section \ref{Tricorn_comp}, there are three parabolic arcs $\cC_1, \cC_2, \cC_3$ on $\partial H$ such that for any $a\in\cC_k$ (for $k\in\{1, 2, 3\}$), the self-conjugate simple parabolic cycle of $N_a$ is formed by the merger of the self-conjugate attracting cycle with the self-conjugate repelling cycle $\{p_k(U_i)\}_{i=1}^{2n}$. In particular, for any $a\in\cC_k$, the characteristic parabolic point of $N_a$ is $p_k(U_1)$. We paraphrase this observation in the following lemma.

\begin{lemma}[Visibility of Characteristic Parabolic Points]\label{visibility_parabolic}
The dynamical co-root $p_k(U_1)$ is visible (respectively invisible) in the dynamical plane of $N_{a_0}$ (where $a_0$ is the center of the tricorn component $H$) if and only if the characteristic parabolic point of $N_a$ is visible (respectively invisible) for each $a\in \cC_k$.
\end{lemma}

In order to set up the platform where we can apply the perturbation techniques, we will now discuss some geometric properties of the repelling Ecalle cylinder at an invisible characteristic parabolic point $p_k(U_1)$ for the map $N_{a(h)}$ on $\cC_k$. In this case, $p_k(U_1)$ is on the boundary of each of the fixed basins $\mathcal{B}_1, \mathcal{B}_{-1}, \mathcal{B}_a, \mathcal{B}_{\bar{a}}$ (since the basin of attraction of an attracting fixed point is totally invariant, its boundary is the whole Julia set), but not on the boundary of any single component thereof.

\begin{figure}[ht!]
\begin{center}
\includegraphics[scale=0.298]{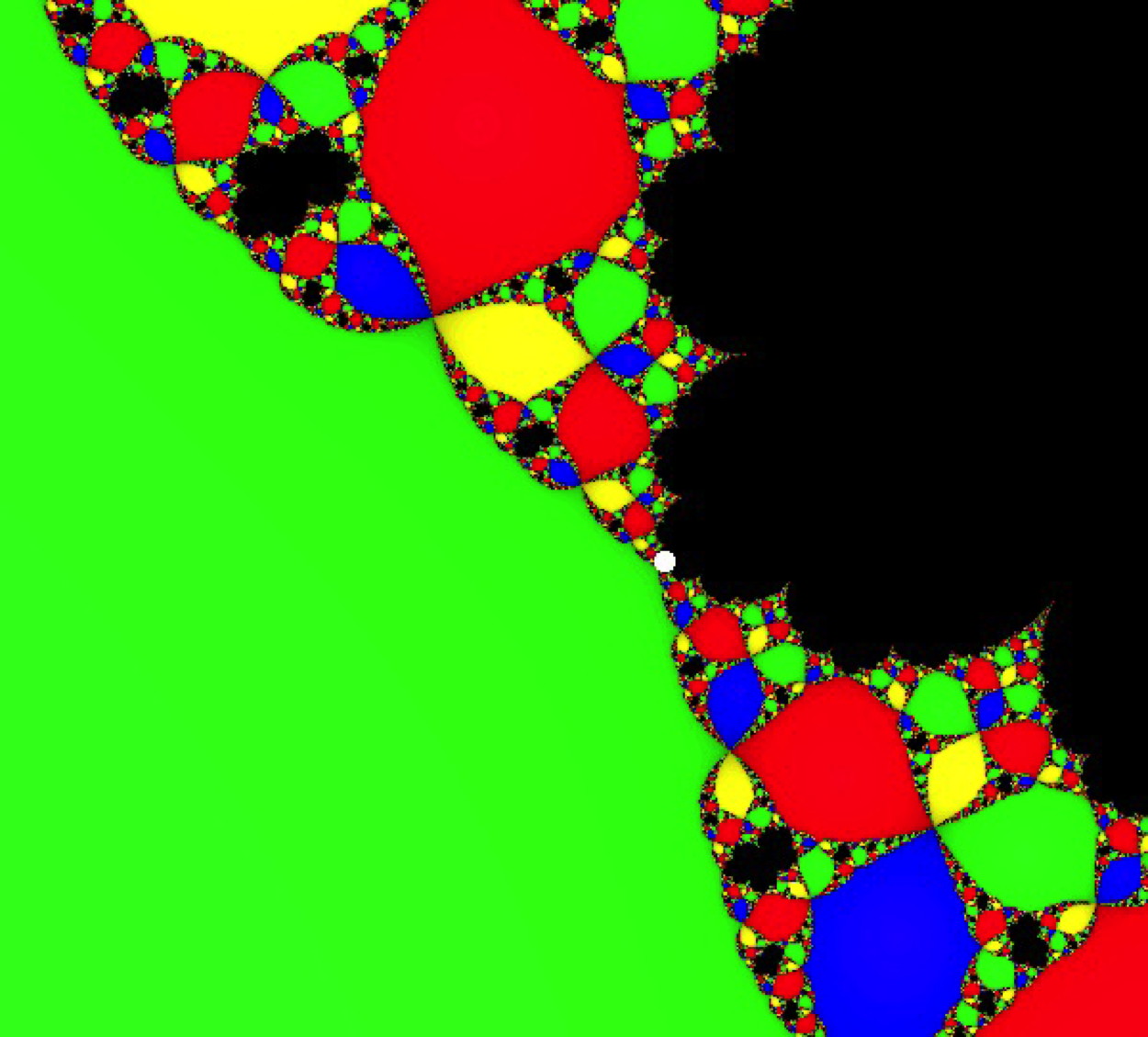}\ \includegraphics[scale=0.27]{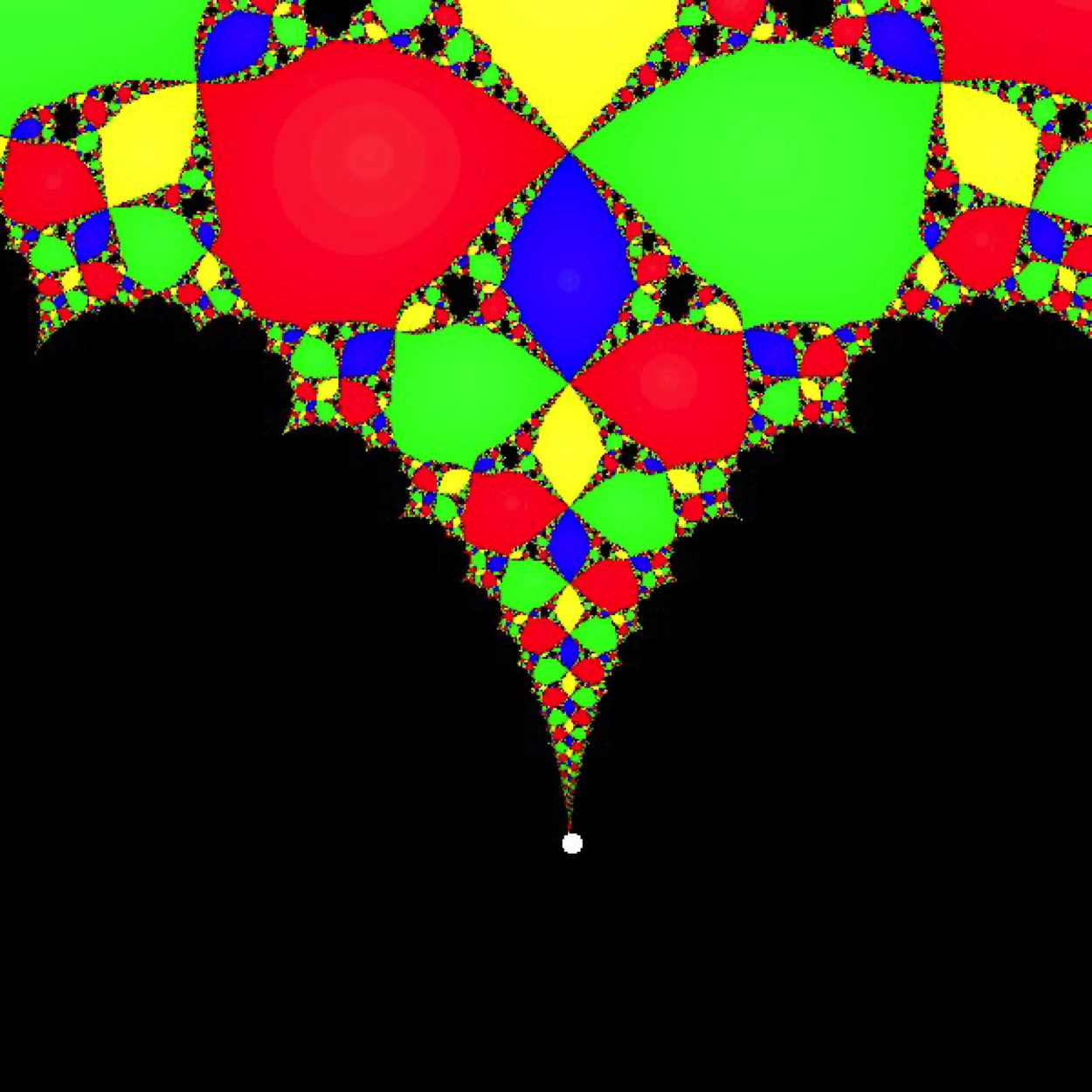}
\caption{Left: A blow-up of the dynamical plane (of a parameter on a parabolic arc of a Tricorn component) around a visible characteristic parabolic point. Here, the characteristic parabolic point is on the boundary of $\mathcal{B}^{\mathrm{imm}}_{-1}$. Right: A blow-up of the dynamical plane (of a parameter on a parabolic arc of a Tricorn component) around an invisible characteristic parabolic point. Here, the characteristic point is in the accumulation set of the pre-periodic Fatou components of $\mathcal{B}_1$ (and of $\mathcal{B}_{-1}$), but not on the boundary of any single component thereof. Consequently, there is a `Julia path' in the repelling cylinder connecting $\partial U_1^+$ and $\partial U_1^-$.}
\label{parabolic_dynamical_zoom}
\end{center}
\end{figure}

Let $V_{a(h)}^{\mathrm{out}}$ be a repelling petal at the invisible characteristic parabolic point $p_k(U_1)$ of $N_{a(h)}$. The projection of $U_1$ into the repelling Ecalle cylinder (of $N_{a(h)}$ at the characteristic parabolic point $p_k(U_1)$) consists of two one-sided infinite cylinders, let us call them $U_1^+$ and $U_1^-$. The projection of $\partial U_1$ to the same cylinder consists of two disjoint Jordan curves $\partial U_1^+$ and $\partial U_1^-$. Note that these two Jordan curves are related by the map $z\mapsto\overline{z}+1/2$. Hence we can assume that the interval of Ecalle heights traversed by $\partial U_1^+$ (respectively by $\partial U_1^-$) is $\left[l_h,u_h\right]$ (respectively $\left[-u_h,-l_h\right]$), where $u_h>l_h$ (since $\partial U_1^+$ is not an analytic curve, its projection to the repelling Ecalle cylinder is not a geodesic, hence not a round circle; therefore, the projection must traverse a positive interval of Ecalle heights). Therefore, $U_1^+$ contains $\mathbb{R}/\mathbb{Z}\times\left(u_h,+\infty\right)$, and $U_1^-$ contains $\mathbb{R}/\mathbb{Z}\times\left(-\infty,-u_h\right)$. Clearly, $u_h>0$ for all $h$ in $\R$.

\begin{lemma}\label{julia_path}
Let the characteristic parabolic point $p_k(U_1)$ of the map $N_{a(h)}$ be invisible for all $a(h)$ on $\cC_k$. Let $\widetilde{J}$ be the projection of the Julia set $J(N_{a(h)})$ into the repelling cylinder at $p_k(U_1)$. Then there exists a path in $\widetilde{J}$ connecting a point of $\partial U_1^+$ at height $u_h$ and a point of $\partial U_1^-$ at height $-u_h$.
\end{lemma}
\begin{proof}
Since $p_k(U_1)$ is not on the boundary of any Fatou component other than $U_1$, it follows that the projection of the Fatou set of $N_{a(h)}$ into the repelling Ecalle cylinder at $p_k(U_1)$ does not contain any conformal annulus of finite modulus in the homotopy class of $\mathbb{R}/\mathbb{Z}$ (compare Figure \ref{parabolic_dynamical_zoom} (right)). Moreover, as every Fatou component of $N_{a(h)}$ is simply connected (recall the the Julia set of any Newton map arising from a polynomial is connected), it follows that no connected component of the projection of its Fatou set into the repelling Ecalle cylinder is an annulus of finite modulus.

Since $N_{a(h)}$ is geometrically finite and $J(N_{a(h)})$ is connected, it is locally connected \cite[Theorem A]{Tan}. Let us denote the projection of $J(N_{a(h)})$ into the repelling cylinder by $\widetilde{J}$. Then $\widetilde{J}$ is compact. Since no connected component of the complement of $\widetilde{J}$ in the repelling cylinder is an annulus of finite modulus, it follows that $\widetilde{J}$ is connected. Furthermore, since $J(N_{a(h)})$ is locally connected, $\widetilde{J}$ is locally connected.\footnote{Quotients of locally connected spaces are locally connected.} As compact, connected, locally connected metrizable spaces are path-connected, we conclude that $\widetilde{J}$ is path connected. Hence in the repelling Ecalle cylinder, there is a path in $\widetilde{J}$ connecting a point of $\partial U_1^+$ at height $u_h$ and a point of $\partial U_1^-$ at height $-u_h$.
\end{proof}

Recall that $\cC_k$ has, at both ends, an interval of positive length across which bifurcation from $H$ (of period $2n$) to a Mandelbrot component (of period $2n$) occurs. According to \cite[Theorem 7.3]{HS}, if $a(h)$ is such a bifurcating parameter, then either $h\geq u_h$, or $h\leq -u_h$. If the critical Ecalle height $0$ parameter $a(0)$ is a bifurcating parameter, then $u_0\leq 0$, which is impossible. Hence there is an interval $(-\epsilon,\epsilon)$ of Ecalle heights such that no bifurcation to Mandelbrot components occurs across the sub-arc $a((-\epsilon,\epsilon))$ of $\cC_k$. It follows that there exist $h_2>0>h_1$ such that $a((h_1,h_2))$ is the maximal sub-arc of $\cC_k$ across which bifurcation to Mandelbrot components does not occur; we call $a((h_1,h_2))$ the \emph{non-bifurcating sub-arc} of $\cC_k$ (compare \cite[\S 7]{HS}).

\begin{definition}[Visibility of Parabolic Arcs and Tricorn Components]
\label{parameter_visibility}
The parabolic arc $\cC_k$ is called \emph{visible} if some point of $a((h_1,h_2))$ lies on the boundary of a hyperbolic component other than $H$. Otherwise, $\cC_k$ is called \emph{invisible}. A Tricorn component $H$ is called invisible if each parabolic arc on $\partial H$ is invisible.
\end{definition}

\subsection{Characterizing invisible tricorn components}\label{char_inv_tri}

In the following, we will analyze some topological properties of the parameter space in a neighborhood of $a((h_1,h_2))$. When $h\in\left(h_1,h_2\right)$, there is no bifurcation to Mandelbrot components across $a(h)$. This implies that the critical Ecalle height $h$ of $N_{a(h)}$ lies in $\left(-u_h,u_h\right)$. Let us start with a couple of lemmas that show that capture and Tricorn components accumulate on parabolic arcs on which the characteristic parabolic points are invisible.

Let $\mathcal{H}^{\mathrm{cap}}$ be the union of all the capture hyperbolic components of the family $\mathcal{N}_4^*$.

\begin{figure}[ht!]
\begin{center}
\includegraphics[scale=0.35]{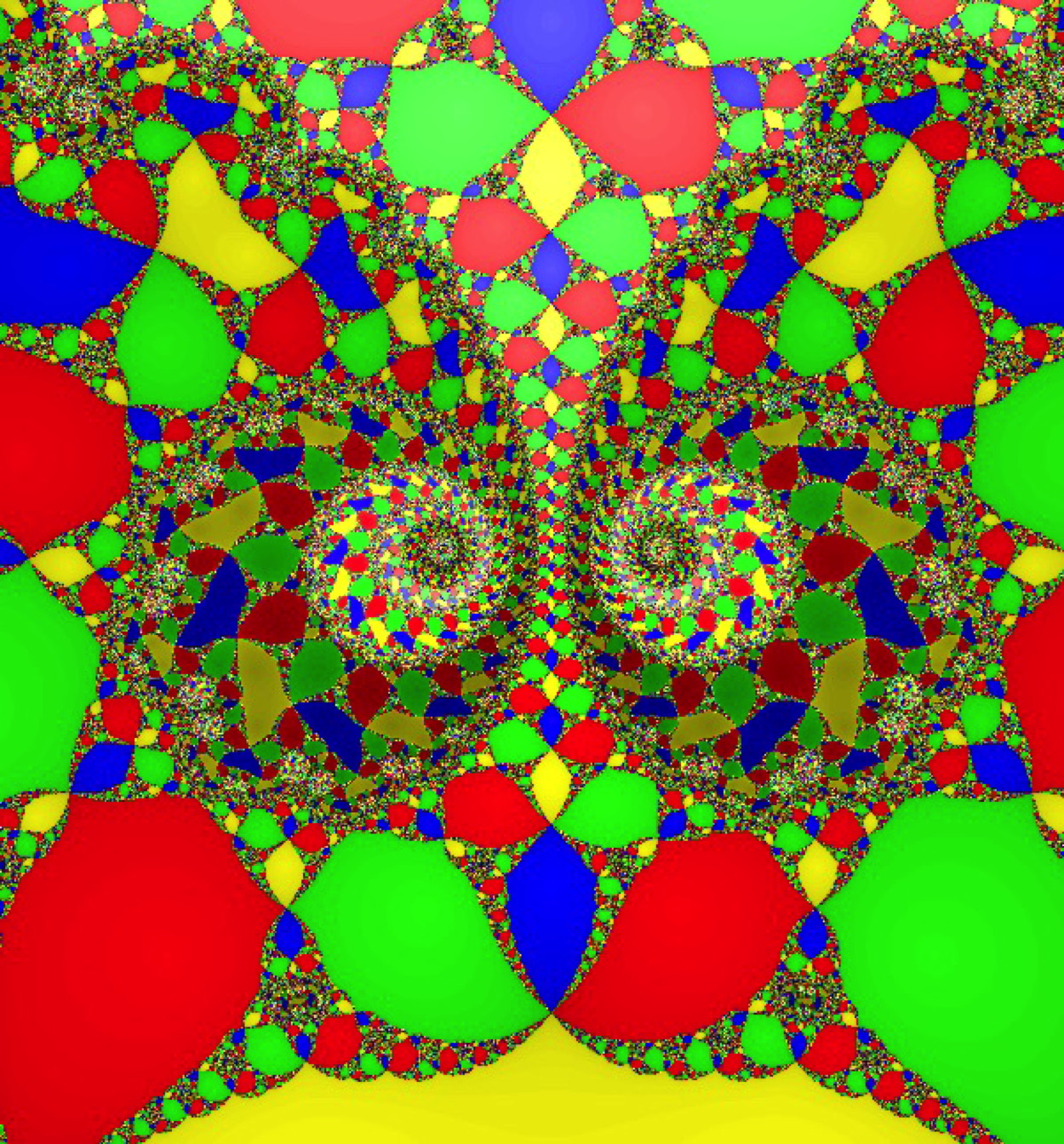}
\caption{A blow-up of the dynamical plane of a near-parabolic parameter showing the eggbeater dynamics ``near'' an invisible parabolic point.}
\label{perturbed_dynamical_invisible_eggbeater}
\end{center}
\end{figure}

\begin{lemma}[Cluster of Capture Components]\label{capture_components_accumulate_1}
Suppose that the characteristic parabolic point $p_k(U_1)$ of $N_a$ is invisible for each $a\in\cC_k$. Then $a[h_1,h_2]$ is contained in the closure of $\mathcal{H}^{\mathrm{cap}}$.
\end{lemma}
\begin{proof}
For $h\in (h_1,h_2)$, we have that $h\in(-u_h,u_h)$. Choose an $\epsilon>0$ such that $(h-\epsilon,h+\epsilon)\subset(-u_h,u_h)$. Then by Lemma \ref{julia_path}, the annulus $\mathbb{R}/\mathbb{Z}\times\left(h-\epsilon,h+\epsilon\right)$ in the repelling cylinder at the characteristic parabolic point of $N_{a(h)}$ intersects a path $\gamma$ contained in the Julia set. Hence, this annulus must also intersect $\mathcal{B}_{a(h)}$. Note that the basin of attraction of a super-attracting fixed point does not get much smaller upon perturbation \cite[Theorem~6.1(a)]{D2}. Therefore, by the parabolic orbit correspondence theorem \cite[Proposition 2.2]{Lei1},\footnote{Due to real-analytic parameter dependence of the family $\mathcal{N}_4^*$, we need to appeal to the alternative topological argument outlined in the proof of \cite[Proposition 2.2]{Lei1}. In fact, the proof is essentially an application of the Brouwer fixed point theorem.} the parameter $a(h)$ can be slightly perturbed outside of $\overline{H}$ so that for the perturbed parameter $a$, the critical point $c_a$ has Ecalle height in $(h-\epsilon,h+\epsilon)$ and the critical orbit $\{N_a^{\circ k}(c_a)\}_{k\in\mathbb{N}}$  lies in $\mathcal{B}_a$ (compare Figure \ref{perturbed_dynamical_invisible_eggbeater}). But by Proposition \ref{no_free}, the critical point $c_a$ cannot belong to $\mathcal{B}^{\mathrm{imm}}_a$. Therefore, the perturbed parameter $a$ lies in a capture component. This shows that there are points of $\mathcal{H}^{\mathrm{cap}}$ arbitrarily close to $a(h)$, for every $h\in(h_1,h_2)$. Hence, $a((h_1,h_2))\subset \overline{\mathcal{H}^{\mathrm{cap}}}$. Since $\overline{\mathcal{H}^{\mathrm{cap}}}$ is a closed set, it follows that $a[h_1,h_2]\subset \overline{\mathcal{H}^{\mathrm{cap}}}$.
\end{proof}

A slightly modified version of the proof of Lemma \ref{capture_components_accumulate_1} exhibits the existence of infinitely many Tricorn components near parabolic arcs with invisible characteristic parabolic points. This will be an important step for the proof of our main theorem.

\begin{lemma}[Accumulation of Tricorn Components]\label{Tricorns_accumulate}
Suppose that the characteristic parabolic point $p_k(U_1)$ of $N_a$ is invisible for each $a\in\cC_k$, and $h\in(h_1,h_2)$. Then $a(h)$ is a limit point of the centers of Tricorn components $\{H_r\}_{r=1}^\infty$. Moreover, if $b_r\in \overline{H_r}$ for each $r$, then all limit points of the sequence $\{b_r\}$ belong to $\cC_k$.
\end{lemma}
\begin{proof}
We will argue as in the proof of Lemma \ref{capture_components_accumulate_1}.

By Lemma \ref{julia_path}, the annulus $\mathbb{R}/\mathbb{Z}\times\left(h-\epsilon,h+\epsilon\right)$ in the repelling cylinder at the characteristic parabolic point of $N_{a(h)}$ intersects a path $\gamma$ contained in the Julia set. Since the set of iterated pre-images of $\overline{c_{a(h)}}$ accumulate on the curve $\gamma$, the annulus under consideration contains some iterated pre-image of $\overline{c_{a(h)}}$. Under small perturbation, this iterated pre-image of $\overline{c_{a(h)}}$ can be continued as an iterated pre-image of the corresponding free critical point, and their heights change continuously. Therefore, by \cite[Proposition 2.2]{Lei1}, there exist parameters $\{a_r\}$ near $a(h)$ (and lying outside of $\overline{H}$) so that for the perturbed parameter $a_r$, the critical point $c_{a_r}$ exits through the gate in $k_r$ steps (with $k_r\to+\infty$) and its image $T_{a_r}(c_{a_r})=N_{a_r}^{\circ 2nk_r}(c_{a_r})$ under the transit map is an iterated pre-image of $\overline{c_{a_r}}$.

Evidently, for each such parameter $a_r$, the free critical point $c_{a_r}$ eventually maps to the other free critical point $\overline{c_{a_r}}$, so they are centers of Tricorn components. This proves that $a(h)$ is a limit point of centers of infinitely many Tricorn components, which we label as $H_r$.

For the second part of the lemma, assume that $b_r\in \overline{H_r}$ for each $r$. Since all maps in $H_r$ (except its center) are topologically conjugate, it is easy to see that the escaping time of the critical point $c_a$ of each parameter $a$ in $\overline{H_r}$ is $k_r-1$ or $k_r$ or $k_r+1$. Since $k_r\to+\infty$, it follows that the escaping times of the critical points $c_{b_r}$ of $N_{b_r}$ also tend to $+\infty$. Hence their lifted phase $\widetilde{\phi}(b_r)$ tends to $-\infty$. Therefore, any accumulation point of the sequence $\{b_r\}$ belongs to $\cC_k$.
\end{proof}

The next lemma is a sharpened version of Lemma \ref{capture_components_accumulate_1} that relates the notions of invisibility in the dynamical plane (Definition \ref{dynamical_visibility}) and in the parameter plane (Definition \ref{parameter_visibility}).

\begin{lemma}[Invisible Parabolic Points Yield Invisible Parabolic Arcs]\label{capture_components_accumulate_2}
If the dynamical co-root $p_k(U_1)$ is invisible in the dynamical plane of the center $a_0$ (of $H$), then the parabolic arc $\cC_k$ is invisible.
\end{lemma}
\begin{proof}
It suffices to prove that every open set $V$ satisfying
\begin{enumerate}
\item $V\cap \overline{H}=\emptyset$, and

\item $\partial V\cap a\left(h_1, h_2\right)\neq\emptyset$,
\end{enumerate}
intersects infinitely many distinct hyperbolic components.

Let $V$ be an open set as above and $a(h)\in \partial V\cap a\left(h_1, h_2\right)$. Since $h\in(h_1,h_2)$, we have that $h\in(-u_{h},u_{h})$. Note that by Lemma \ref{visibility_parabolic}, if the dynamical co-root $p_k(U_1)$ is invisible in the dynamical plane of the center $a_0$ (of $H$), then the characteristic parabolic point $p_k(U_1)$ of $N_a$ is invisible for each $a\in\cC_k$. Hence, there is a curve $\gamma$ in the Julia set of $N_{a(h)}$ connecting a point of $\partial U_1^+$ at height $u_h$ and a point of $\partial U_1^-$ at height $-u_h$. Since $a(h)$ is not an exceptional point of $N_{a(h)}$, the closure of the iterated pre-images of $a(h)$ under $N_{a(h)}$ contains $J(N_{a(h)})$. In particular, any neighborhood of a point on $\gamma$ contains infinitely many iterated pre-images of $a(h)$.

Choose a neighborhood $U$ of $a((h_1,h_2))$ such that the Fatou coordinates of $N_{a(h)}$ persist throughout $U^-:=U\setminus\overline{H}$. Since $a(h)\in \overline{U^-\cap V}$, the projection of the set $\widetilde{U}:=\displaystyle \{(T_{a}(c_a), a):a\in U^-\cap V\}\subset \bigsqcup_{a\in \overline{U^-}} C_a^{\mathrm{out}}\cong \C/\Z\times\overline{U^-}$ to the first coordinate ``winds around" $\C/\Z$ infinitely often as the lifted phase goes to $-\infty$ \cite[Lemma 2.5]{IM1}. The accumulation set of the image contains $\{h\}\times\R/\Z$. Hence, in particular, the part of the image inside the cylinder $(-u_h,u_h)\times\R/\Z$ winds around $\C/\Z$ infinitely often (since $h\in(h_1,h_2)$). It follows that the projection of $\gamma$ in the cylinder $\C/\Z$ intersects the projection of $\widetilde{U}$ in $\C/\Z$ infinitely many times.

Under sufficiently small perturbations, the iterated pre-images of $a(h)$ that accumulate on $\gamma$ can be continuously followed as iterated pre-images of $a$ under $N_a$, and their heights change continuously (since the normalized Fatou coordinates depend continuously on the parameter). So the set $\widetilde{U}$ defined above will contain infinitely many elements of the form $\left(z(a),a\right)$, where $a\in U^-\cap V$, and $z(a)$ is an iterated pre-image of $a$ under $N_a$. This means that there are infinitely many parameters $a$ in $U^-\cap V$ for which the critical point $c_a$ eventually maps to the super-attracting fixed point $a$. Evidently, any such parameter $a$ belongs to a capture component and the corresponding map $N_a$ is postcritically finite. Since every hyperbolic component in the parameter space of the family $\mathcal{N}_4^*$ has a unique center (see Proposition \ref{unique_center}), we have proved that $V$ intersects infinitely many distinct hyperbolic components. This completes the proof of the fact that no parameter on $a((h_1,h_2))$ lies on the boundary of a hyperbolic component other than $H$.
\end{proof}

\begin{corollary}\label{no_path}
If the dynamical co-root $p_k(U_1)$ is invisible in the dynamical plane of the center $a_0$ (of $H$), then there does not exist any path in the hyperbolic locus that accumulates on $a((h_1,h_2))$ from the exterior of $H$.
\end{corollary}

We now prove the converse of Lemma \ref{capture_components_accumulate_2} which completes the connection between visibility in the dynamical plane and in the parameter plane.

\begin{lemma}[Visible Parabolic Points Yield Visible Parabolic Arcs]\label{visible_arc}
If the dynamical co-root $p_k(U_1)$ is visible in the dynamical plane of the center $a_0$ (of $H$), then the parabolic arc $\cC_k$ is visible.
\end{lemma}
\begin{proof}
By assumption, $p_k(U_1)\in\partial\mathcal{B}^{\mathrm{imm}}_{1}\cup\partial\mathcal{B}^{\mathrm{imm}}_{-1}$. To be specific, let us assume that $p_k(U_1)\in\partial\mathcal{B}^{\mathrm{imm}}_{1}$. We will show that $a(0)\in\partial \mathcal{H}_1$.

Since the free critical points of $N_{a(0)}$ do not lie in $\mathcal{B}^{\mathrm{imm}}_{1}$, it follows that $\partial\mathcal{B}^{\mathrm{imm}}_{1}$ is a Jordan curve. In particular, some $2n$-periodic dynamical ray $R_\theta$ (of $\mathcal{B}^{\mathrm{imm}}_{1}$) lands at $p_k(U_1)$. Note that $R_\theta$ is fixed by $\iota \circ N_{a(0)}^{\circ n}$. Let $V$ be a fundamental domain of $N_{a(0)}^{\circ 2n}$ in the repelling petal at $p_k(U_1)$. The repelling Fatou coordinates induce an isomorphism $\Psi^{\mathrm{rep}}:V\to\C/\Z$. We choose a thin tubular neighborhood $\widetilde{W}$ of $\Psi^{\mathrm{rep}}(R_\theta)$ such that
\begin{enumerate}
\item $\widetilde{W}\subset\Psi^{\mathrm{rep}}(\mathcal{B}^{\mathrm{imm}}_1)$, and

\item $\widetilde{W}$ is invariant under the automorphism $\zeta\mapsto\overline{\zeta}+\frac{1}{2}$ $(\textrm{mod}\ \Z)$ of $\C/\Z$.
\end{enumerate}
In particular, it follows from the construction that $\widetilde{W}$ is an annulus in the homotopy class of $\R/\Z$, and it intersects $\{0\}\times\R/\Z$. Setting $W:=(\Psi^{\mathrm{rep}})^{-1}(\widetilde{W})$, we have that $W\subset \mathcal{B}^{\mathrm{imm}}_1$, and $W$ is ``$\iota \circ N_{a(0)}^{\circ n}$-invariant''.

\begin{figure}
\begin{tikzpicture}
\node[anchor=south west,inner sep=0] at (0,0) {\includegraphics[scale=0.5]{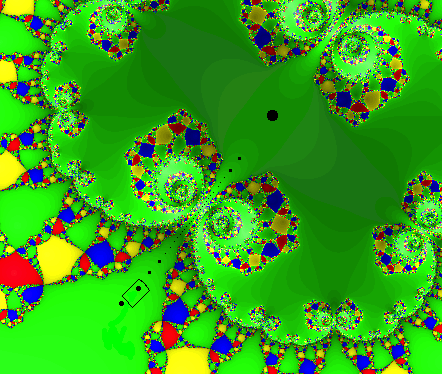}};
\node at (2,.5) {$W_a$};
\draw [->] (2,.7) -- (2.35,1.3);
\node at (5.1,4.5) {$c_a$};
\end{tikzpicture}
\caption{The $N_a^{\circ 2n}$-orbit of $c_a$ exits through the gate and lands in $W_a$, which is contained in the basin $\mathcal{B}_1$.}
\label{perturbed_dynamical_visible_eggbeater}
\end{figure}

It is known that the basin of attraction $\mathcal{B}_1$ can not get too small when $a(0)$ is perturbed a little bit (compare \cite[Theorem 6.1(a)]{D2}). Thus for sufficiently small perturbations $a$ of $a(0)$, there exists a fundamental domain $V_a$ of $N_{a}^{\circ 2n}$ in the outgoing domain $V_a^{\mathrm{out}}$ and a connected set $W_a\subset V_a$ such that
\begin{enumerate}
\item $W_a\subset\mathcal{B}_1$, and

\item the projection of $W_a$ into the outgoing cylinder is invariant under the automorphism $\zeta\mapsto\overline{\zeta}+\frac{1}{2}$ $(\textrm{mod}\ \Z)$ of $\C/\Z$.
\end{enumerate}

Clearly, the projection of $W_a$ into the outgoing cylinder intersects $\{0\}\times\R/\Z$. Since the critical point $c_a$ and the Fatou coordinates depend continuously on the parameter, there exist parameters $a$ arbitrarily close to $a(0)$ such that the forward orbit of $c_a$ hits $W_a$ (compare Figure \ref{perturbed_dynamical_visible_eggbeater}). Clearly, such a parameter $a$ either lies in a capture component or in $\mathcal{H}_1$.

To finish the proof, we only need to show that for such parameters $a$, the critical point $c_a$ lies in $\mathcal{B}_1^{\mathrm{imm}}$. It easily follows from our construction that the connected component $\widehat{W}_a$ of $\displaystyle\bigcup_{k=0}^\infty N_a^{\circ (-2nk)}(W_a)$ containing $c_a$ also contains $W_a$. It follows that $N_a^{\circ 2n}(c_a)\in N_a^{\circ 2n}(\widehat{W}_a)\cap \widehat{W}_a$; i.e. $N_a^{\circ 2n}(\widehat{W}_a)\cap \widehat{W}_a\neq\emptyset$. Moreover, as $W_a\subset\mathcal{B}_1$, it follows that $\widehat{W}_a$ is contained in a single periodic component of $\mathcal{B}_1$. So, $c_a\in \widehat{W}_a\subset\mathcal{B}_1^{\mathrm{imm}}$.

To summarize, we have showed that $a(0)\in\cC_k$ lies on the boundary of $\mathcal{H}_1$. Therefore, $\cC_k$ is visible.
\end{proof}

Combining Lemma \ref{capture_components_accumulate_2} and (the proof of) Lemma \ref{visible_arc}, we have the following.

\begin{corollary}\label{visibility_principal}
A parabolic arc $\cC_k$ on the boundary of a Tricorn component $H$ is visible if and only if $\cC_k$ intersects $\partial\mathcal{H}_1\cup\partial\mathcal{H}_{-1}$.
\end{corollary}
\begin{proof}
Evidently, if $\cC_k$ intersects $\partial\mathcal{H}_1\cup\partial\mathcal{H}_{-1}$, then $\cC_k$ is visible.

The opposite implication, although not automatic, follows from our previous work. First note that by Lemma \ref{capture_components_accumulate_2}, visibility of the parabolic arc $\cC_k$ implies visibility of the dynamical co-root $p_k(U_1)$ in the dynamical plane of the center $a_0$ (of $H$). It now follows from the proof of Lemma \ref{visible_arc} that $\cC_k$  intersects $\partial\mathcal{H}_1\cup\partial\mathcal{H}_{-1}$.
\end{proof}

\begin{corollary}[Characterization of Invisible Tricorn Components]\label{visibility_parameter_dynamical}
The parabolic arc $\cC_k\subset\partial H$ is visible (equivalently, it intersects $\partial\mathcal{H}_1\cup\partial\mathcal{H}_{-1}$) if and only if the dynamical co-root $p_k(U_1)$ is visible in the dynamical plane of the center of the Tricorn component $H$. In particular, a Tricorn component $H$ is invisible if and only if all dynamical co-roots are invisible in the dynamical plane of the center of $H$.
\end{corollary}

\begin{figure}[ht!]
\begin{center}
\includegraphics[scale=0.263]{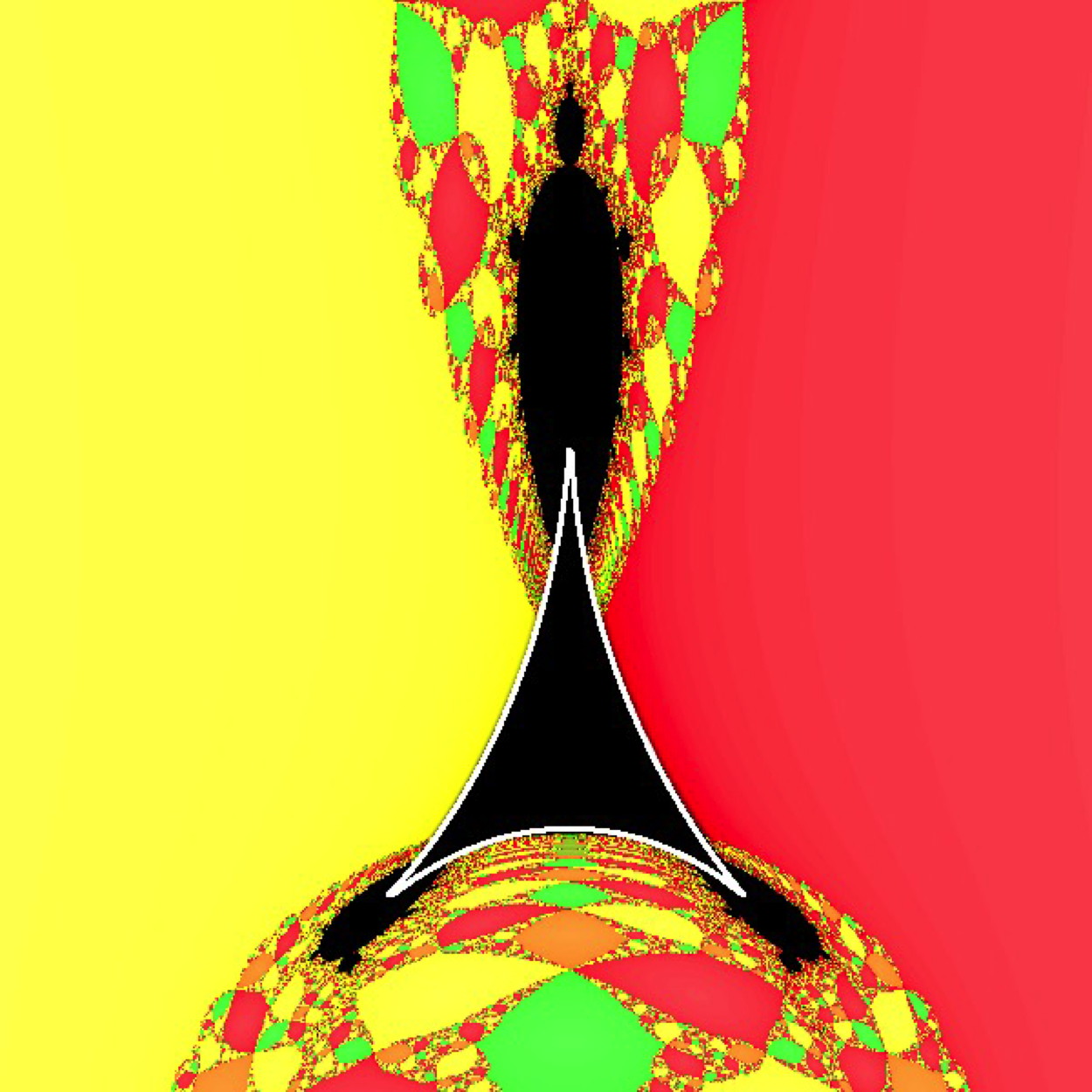}\ \includegraphics[scale=0.342]{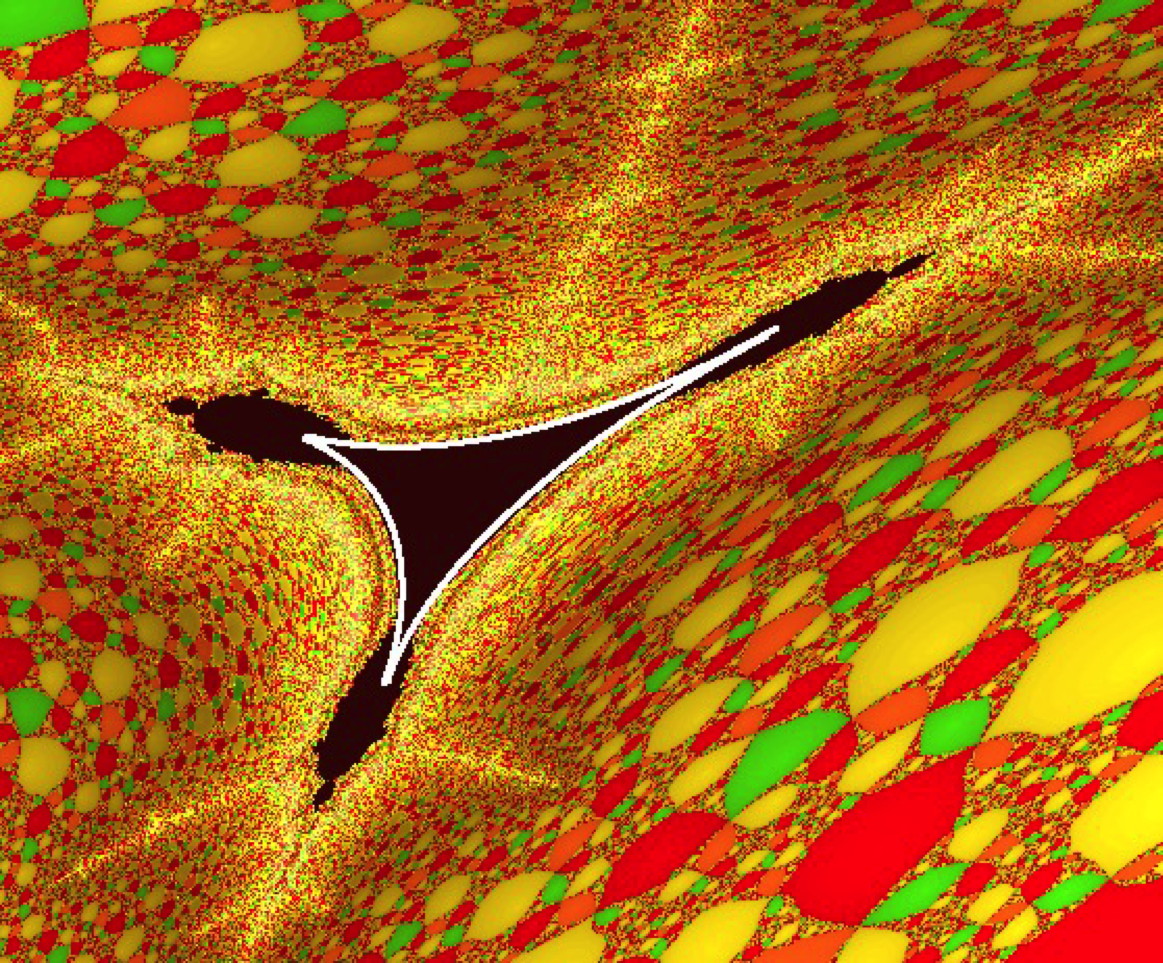}
\caption{Left: A Tricorn component (enclosed by white parabolic arcs) with an invisible parabolic arc. Right: A blow-up of the parameter plane near the invisible parabolic arc (on the left) shows an invisible Tricorn component (enclosed by white parabolic arcs).}
\label{inaccessible}
\end{center}
\end{figure}

\bigskip

\section{Proof of Theorem \ref{inaccessible_arcs}}\label{sec_mainthm_1}

We are now ready to prove one of the main theorems of this paper. The proof essentially consists of two steps. The first step is to show the existence of infinitely many Tricorn components each of which has an invisible parabolic arc. This follows from our construction of postcritically finite maps in Section \ref{sec:ConstructingPCFmaps}. Note that these components, whose centers were constructed in Proposition \ref{thm_NewtonExists}, are not invisible because each of them has two visible parabolic arcs on its boundary (coming from the two visible dynamical co-roots in the dynamical plane of their centers). In order to find invisible Tricorn components, we need to study the local topology of the parameter space near invisible parabolic arcs. This is where the results of Section \ref{sec_implosion} come into play. Indeed, thanks to Lemma \ref{Tricorns_accumulate}, there exist infinitely many Tricorn components ``near" each invisible parabolic arc. A straightforward topological argument using the criterion of visibility of parabolic arcs (obtained in Corollary \ref{visibility_principal}) now shows that most of these nearby Tricorn components are invisible (compare Figure \ref{inaccessible}).

\begin{proof}[Proof of Theorem \ref{inaccessible_arcs}]
By Proposition \ref{thm_NewtonExists} and Corollary \ref{inv_co_root}, each $N_{a_n}$ is a postcritically finite map (with a self-conjugate $2n$-periodic superattracting cycle) in the family $\mathcal{N}_4^*$ with two visible dynamical co-roots $p_1(U_1^{a_n})$ and $p_2(U_1^{a_n})$, and an invisible dynamical co-root $p_3(U_1^{a_n})$. Let $a_n$ be the center of the Tricorn component $H^{(n)}$. By Corollary \ref{visibility_parameter_dynamical}, the boundary of the Tricorn component $H^{(n)}$ has two visible parabolic arcs and an invisible parabolic arc. Let us fix any such hyperbolic component $H^{(n)}$, and call its invisible parabolic arc $\cC^{(n)}$. By Lemma~\ref{capture_components_accumulate_1}, capture components accumulate on the invisible parabolic arc $\cC^{(n)}$.

By Lemma \ref{Tricorns_accumulate}, there is a sequence $\{H_r^{(n)}\}_r$ of Tricorn components whose closures accumulate on $\cC^{(n)}$. Suppose that for a fixed $n>1$, infinitely many of these hyperbolic components $H_r^{(n)}$ are visible. According to Corollary \ref{visibility_principal}, this implies that there exist infinitely many components $H_{r_s}^{(n)}$ such that $\partial H_{r_s}^{(n)}\cap(\partial\mathcal{H}_1\cup\partial\mathcal{H}_{-1})\neq\emptyset$, for $s\in\N$. For each $s\geq 1$, let us choose $b_{r_s}^{(n)}\in\partial H_{r_s}^{(n)}\cap(\partial\mathcal{H}_1\cup\partial\mathcal{H}_{-1})$. By Lemma \ref{Tricorns_accumulate}, the sequence $\{b_{r_s}^{(n)}\}_s$ accumulates on $\cC^{(n)}$. It follows that $\cC^{(n)}\cap(\partial\mathcal{H}_1\cup\partial\mathcal{H}_{-1})\neq\emptyset$, which contradicts the fact that $\cC^{(n)}$ is an invisible parabolic arc. This contradiction proves that in the neighborhood of each $\cC^{(n)}$, all but (possibly) finitely many hyperbolic components $H_r^{(n)}$ are invisible. The proof is now complete.
\end{proof}

\begin{huge}\part{The family $\mathcal{A}_3$}\label{part_two}\end{huge}
\bigskip

In this part, we will look at the parameter space of antipode preserving cubic rational maps, and prove Theorem \ref{inaccessible_arcs_1} concerning local topology of the parameter space near tongues and Tricorn components. While discussing the parameter space of $\mathcal{A}_3$, we will use the same notations used in \cite{BBM1}. Since many important combinatorial and topological properties of the dynamics and parameter space of the family $\mathcal{A}_3$ have been studied in \cite{BBM1,BBM3}, we will only review the relevant aspects and refer the readers to the above-mentioned papers for the proofs.

\bigskip

\section{Hyperbolic components}\label{hyperbolic_1}

Note that each $f_q$ has two fixed critical points at $0$ and $\infty$, and two mutually antipodal critical points (different from $0$ and $\infty$ for $q\neq 0$). These two ``free'' critical points of $f_q$ are denoted by $c_0(q)$ and $c_{\infty}(q)$ such that $c_0(q)$ is a positive multiple of $q$, and $c_\infty(q)$ is a negative multiple of $q$ (see \cite[p. 10]{BBM1}).

The classification of hyperbolic components of the family $\mathcal{A}_3$ is similar to that for the family $\mathcal{N}_4^*$ (compare \cite[Lemma 2.5]{BBM1}).

\textbf{Principal Hyperbolic component}: In the principal hyperbolic component $\mathcal{H}_0$, the free critical point $c_0(q)$ (respectively $c_\infty(q)$) lies in the immediate basin of attraction of $0$ (respectively $\infty$).

\textbf{Capture components}: The capture components are characterized by the fact that one of the free critical points lies in a pre-periodic Fatou component of the basin of $0$, and the other belongs to a pre-periodic Fatou component of the basin of $\infty$.

\textbf{Mandelbrot components}: In the Mandelbrot components, the maps have two distinct antipodal attracting cycles such that one of them attracts $c_0$, and the other attracts $c_{\infty}$. Due to the antipodal symmetry, both these attracting cycles have common period $n$. We will refer to such a component as a Mandelbrot component of period $n$.

\textbf{Tricorn components}: In a Tricorn component, each map has exactly one self-antipodal attracting cycle (of even period $2n$), and this cycle attracts both the critical points $c_0$ and $c_{\infty}$. Let $q_0$ be the center of a Tricorn-type hyperbolic component $H$ of period $2n$. Let us label the periodic (super-attracting) Fatou components $\{U_1^q, U_2^q, \cdots, U_{2n}^q\}$ so that $c_0(q_0)$ is contained in $U_1^q$. By symmetry, it follows that $c_{\infty}(q_0)$ belongs to the antipodal Fatou component $\eta(U_1^q)$. By an argument similar to the one used for the family $\mathcal{N}_4^*$, we have that $f_{q_0}^{\circ n}(U_i^q)=\eta(U_i^q)$ for each $i$ in $\{1, 2, \cdots, 2n\}$. In particular, $f_{q_0}^{\circ n}(c_0(q_0))=c_{\infty}(q_0)$.

The Tricorn components in the parameter space of $\mathcal{A}_3$ come in two different flavors. In addition to the bounded Tricorn components, there are unbounded Tricorn components that are called \emph{tongues} (see \cite[Theorem~6.1]{BBM3} for a proof of existence and dynamical classification of tongues).

Finally, the following result is a consequence of the general theory of hyperbolic components of rational maps \cite[Theorem 7.13, Theorem 9.3]{MP}.

\begin{proposition}\label{unique_center_1}
Every hyperbolic component in the parameter space of the family $\mathcal{A}_3$ is simply connected and has a unique center (i.e. a postcritically finite parameter).
\end{proposition}

\bigskip

\section{Tongues and bounded tricorn components}\label{tri_comp_antipode}

Throughout this section, we assume that $H$ is a tongue or a bounded Tricorn component of period $2n$.

As in the family $\mathcal{N}_4^*$, antiholomorphic dynamics plays a key role in studying the dynamics of maps in the Tricorn components of $\mathcal{A}_3$. Since $f_q$ commutes with the complex conjugation map $\eta$, we have that $f_q^{\circ 2n}=(\eta\circ f_q^{\circ n})\circ(\eta\circ f_q^{\circ n})$. Hence the first return map $f_q^{\circ 2n}$ of any $2n$-periodic Fatou component is the second iterate of the first antiholomorphic return map $\eta\circ f_q^{\circ n}$.

\begin{lemma}[Indifferent Dynamics on the Boundary of Tricorn Components]\label{LemTricornIndiffDyn_1}
The boundary of a Tricorn component consists
entirely of parameters having a self-antipodal parabolic cycle of multiplier $+1$. In suitable
local conformal coordinates, the first return map of some (neighborhood of a) parabolic point of such a map has the form $z\mapsto z+z^{r+1}+\ldots$ with $r\in\{1,2\}$.
\end{lemma}
\begin{proof}
The proof of Lemma \ref{LemTricornIndiffDyn} applies mutatis mutandis to this setting.
\end{proof}

This leads to the following definition.

\begin{definition}[Parabolic Cusps]\label{DefCusp_1}
A parameter $q$ on $\partial H$ is called a {\em parabolic cusp point} if it has a self-antipodal parabolic
cycle such that $r=2$ in the previous lemma. Otherwise, it is called a \emph{simple} parabolic parameter.
\end{definition}

As in Lemma \ref{normalization of fatou}, for a simple parabolic parameter $q\in\partial H$, the first antiholomorphic return map $\eta\circ f_q^{\circ n}$ of an attracting petal is conformally conjugate to the map $\zeta\mapsto\overline{\zeta}+1/2$ via the Fatou coordinate $\psi^{\mathrm{att}}$.

The pre-image of the real line (which is fixed by $\zeta\mapsto\overline{\zeta}+1/2$) under $\psi^{\mathrm{att}}$ is called the attracting equator. By definition, the attracting equator is invariant under $\eta\circ f_q^{\circ n}$. The imaginary part of $\psi^{\mathrm{att}}$ is called the \emph{Ecalle height}.

For any simple parabolic parameter $q\in\partial H$, we will call the Fatou component $U_1^q$ containing the critical point $c_0(q)$ the \emph{characteristic Fatou component}. Clearly, there is a parabolic point $z_1$ on the boundary of $U_1^q$ such that the $f_q^{\circ 2n}$-orbit of each point in $U_1^q$ converges to $z_1$. We will refer to $z_1$ as the \emph{characteristic parabolic point}. We will mainly work with $U_1^q$, and its normalized attracting Fatou coordinate (as above) $\psi^{\mathrm{att}}$. The \emph{critical Ecalle height} of $f_q$ is defined as $$h(f_q):=\im(\psi^{\mathrm{att}}(c_0(q))).$$

The next proposition, which is similar to \cite[Theorem 3.2]{MNS}, shows the existence of real-analytic arcs of simple parabolic parameters on the boundaries of Tricorn components.

\begin{proposition}[Parabolic Arcs]\label{parabolic arcs_1}
Let $\widetilde{q}$ be a parameter such that $f_{\widetilde{q}}$ has a self-antipodal simple parabolic cycle. Then $\widetilde{q}$ is on a parabolic arc in the  following sense: there  exists an injective real-analytic arc $\cC$ of simple parabolic parameters $q(h)$ (for $h\in\mathbb{R}$) with quasiconformally equivalent dynamics of which $\widetilde{q}$ is an interior point, and the critical Ecalle height of $f_{q(h)}$ is $h$. In particular, each $f_{q(h)}$ has a self-antipodal simple parabolic cycle.
\end{proposition}

The following result was proved in \cite[Theorem 6.13]{BBM1}.

\begin{proposition}\label{jordan_1}
Let $H$ be a Tricorn component of period $2n$, and $q$ be some parameter in $H$ (or some simple parabolic parameter on $\partial H$). Then the boundary $\partial U_i^q$ of each $2n$-periodic Fatou component of $f_q$ is a Jordan curve.
\end{proposition}

For any $q\in H$ (respectively, for a simple parabolic parameter $q$ on $\partial H$), we label the $2n$-periodic self-antipodal cycle of attracting (respectively parabolic) Fatou components of $f_{q}$ as $U_1^q$, $U_2^q$, $\cdots$, $U_{2n}^q$ such that $U_1^q$ contains $c_0(q)$. There are exactly three fixed points of the degree two first antiholomorphic return map $\iota\circ f_{q}^{\circ n}$ on the boundary of each $U_i^q$, and we call them $p_1(U_i^q)$, $p_2(U_i^q)$, $p_3(U_i^q)$. These are also precisely the fixed points of the first holomorphic return map $f_{q}^{\circ 2n}$ on $\partial U_i^q$.

\begin{definition}[Roots and Co-roots]\label{root_co_root_dyn_1}
By definition, $p_k(U_i^q)$ is a dynamical \emph{root point} if there exists $i'\in \{1, 2, \cdots, 2n\}$ with $i'\neq i$ and $k'\in\{1,2,3\}$ such that $p_{k'}(U_{i'}^q)=p_k(U_i^q)$; i.e. if two distinct $2n$-periodic Fatou components touch at $p_k(U_i^q)$. In this case, two of the three self-antipodal cycles $\{p_k(U_i^q)\}_i$ ($k=1, 2, 3$) coincide.

Otherwise, $p_k(U_i^q)$ is called a dynamical \emph{co-root}. In this case, all three self-antipodal cycles $\displaystyle\{p_k(U_i^q)\}_{i=1}^{2n}$ (for $k=1, 2, 3$) are disjoint.
\end{definition}

The points $p_1(U_i^q)$, $p_2(U_i^q)$, $p_3(U_i^q)$ can be followed continuously as fixed points of $\eta\circ f_q^{\circ n}$ throughout the union of $H$ and the parabolic arcs on $\partial H$, and the topological structure of the Julia set remains unchanged along the way. On each parabolic arc on $\partial H$, the unique self-antipodal attracting cycle merges with the self-antipodal repelling cycle $\{p_k(U_i^q)\}_{i=1}^{2n}$, for some fixed $k\in \{1, 2, 3\}$, forming a simple parabolic cycle. Therefore, there are three distinct ways in which a simple parabolic cycle can be formed on $\partial H$. It follows that there are three parabolic arcs $\cC_1$, $\cC_2$ and $\cC_3$ on $\partial H$ satisfying the property that the self-antipodal parabolic cycle of any $q\in \cC_k$  (where $k\in \{1, 2, 3\}$)  is formed by the merger of the self-antipodal attracting cycle with the self-antipodal repelling cycle $\{p_k(U_i^q)\}_{i=1}^{2n}$. Finally, the cusp points (when they exist) on $\partial H$ are characterized by the merger of two of the three self-antipodal cycles $\{p_k(U_i^q)\}_i$ (for $k=1, 2, 3$) with the unique self-antipodal attracting cycle.

\begin{figure}[ht!]
\begin{tikzpicture}
\node[anchor=south west,inner sep=0] at (0,0) {\includegraphics[width=0.48\linewidth]{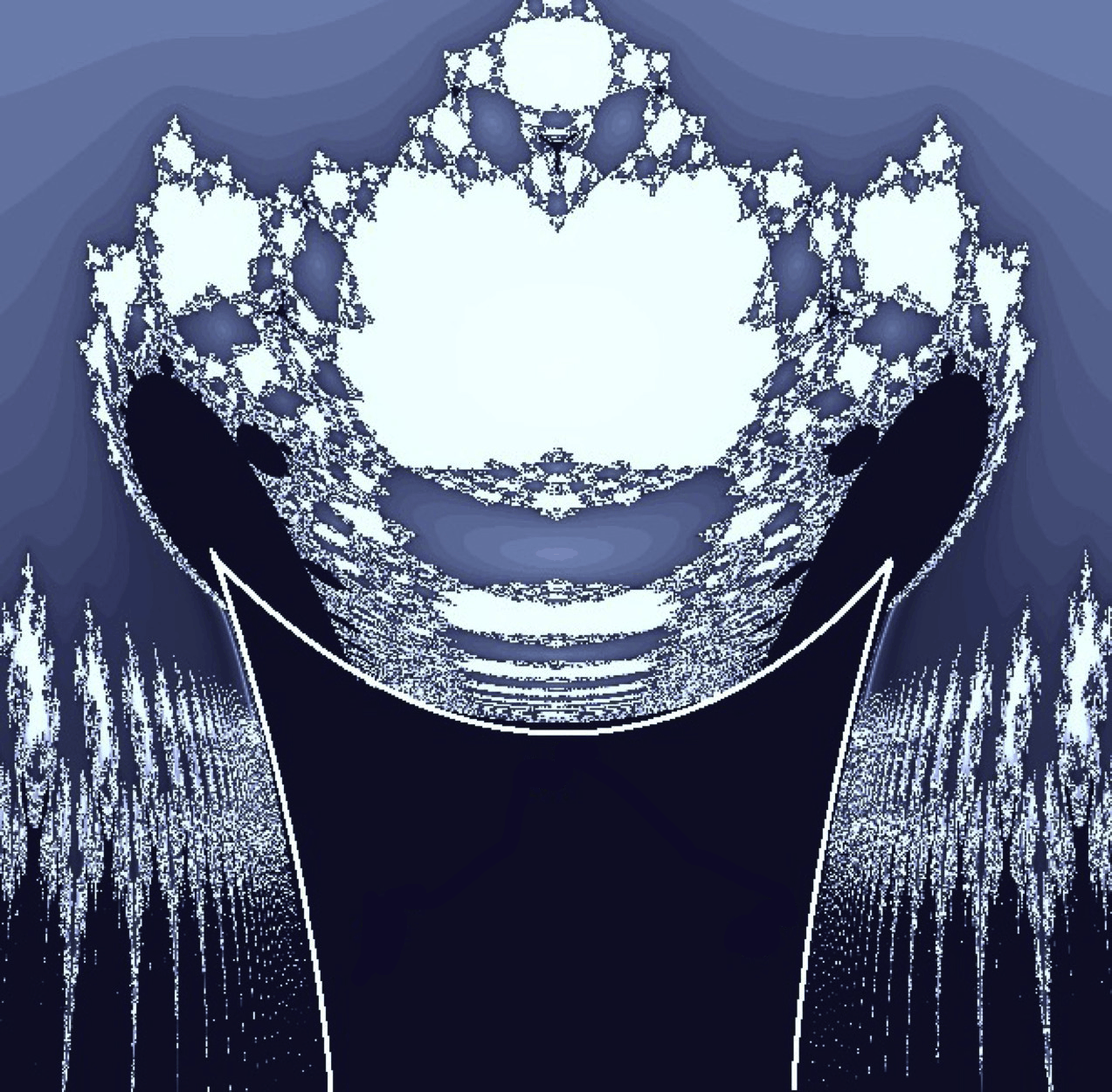}\ \includegraphics[width=0.47\linewidth]{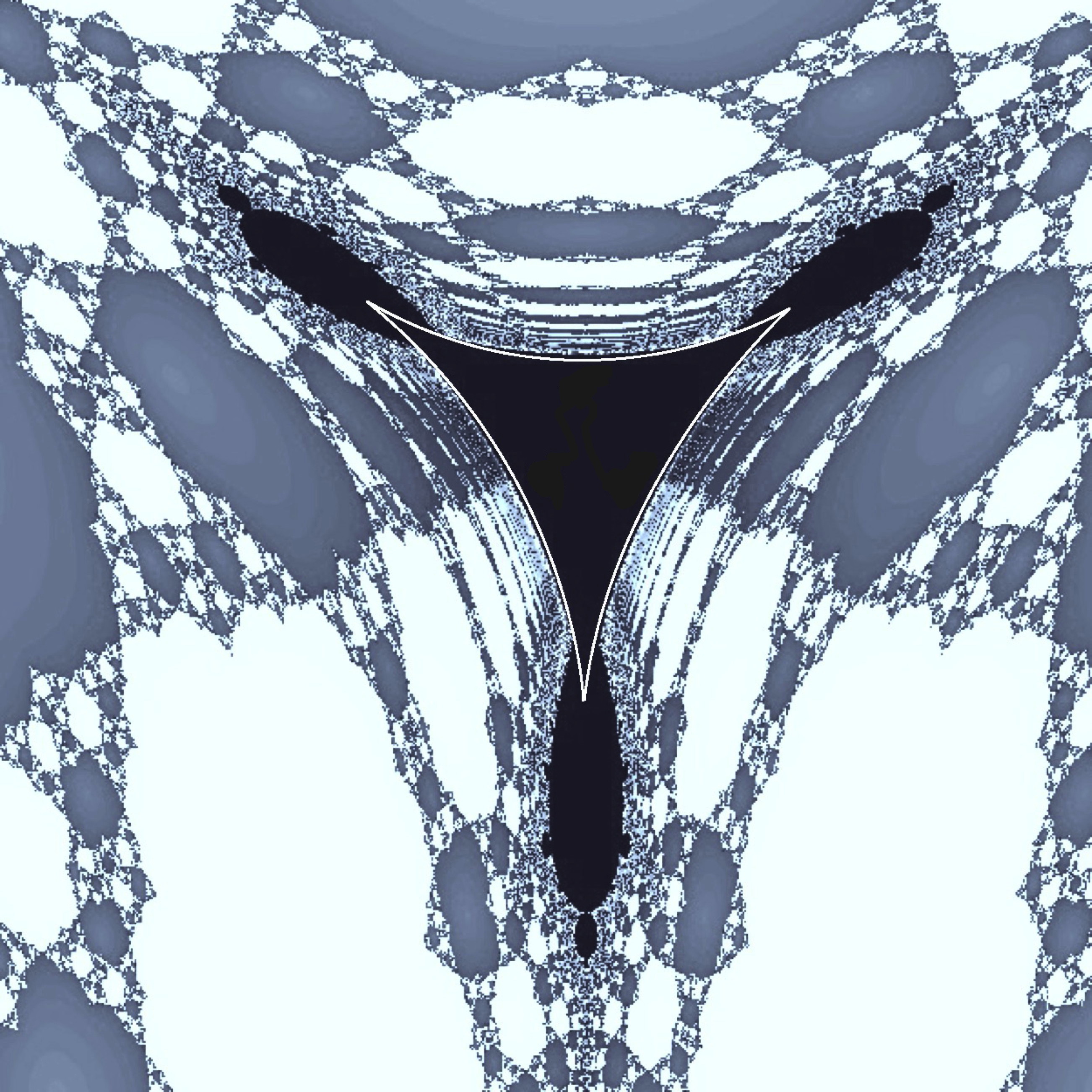}};
\node at (2,0.9) {\textcolor{white}{$\cC_1$}};
\node at (3.2,1.7) {\textcolor{white}{$\cC_2$}};
\node at (4.2,0.9) {\textcolor{white}{$\cC_3$}};
\node at (9.15,3.6) {\textcolor{white}{\begin{scriptsize}$\cC_1$\end{scriptsize}}};
\node at (9.47,3.9) {\textcolor{white}{\begin{scriptsize}$\cC_2$\end{scriptsize}}};
\node at (9.8,3.6) {\textcolor{white}{\begin{scriptsize}$\cC_3$\end{scriptsize}}};
\end{tikzpicture}
\caption{Left: A tongue component of period $2$ with an invisible co-root parabolic arc $\cC_2$. The parabolic arcs $\cC_1$ and $\cC_3$ contain bare regions. Right: A bounded Tricorn component with all three co-root parabolic arcs invisible.}
\label{inaccessible_2}
\end{figure}

\begin{figure}
\begin{tikzpicture}
\node[anchor=south west,inner sep=0] at (0,0) {\includegraphics[scale=0.285]{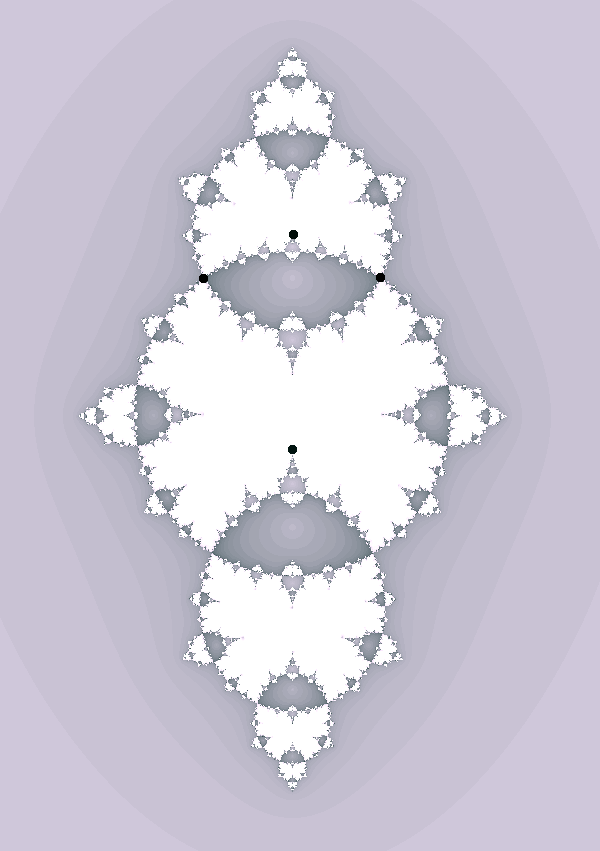}\ \includegraphics[scale=0.539]{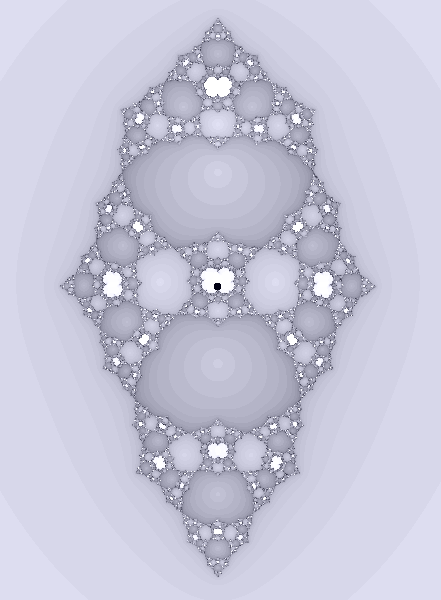}};

\node at (.6,6.5) {$p_3(U_2)$};
\node at (.73,6) {$=p_1(U_1)$};
\draw [->] (1.5,6) -- (2.02,5.8);

\node at (5.3,6.5) {$p_1(U_2)$};
\node at (5.3,6) {$=p_3(U_1)$};
\draw [->] (4.7,6.3) -- (3.9,5.8);

\node at (5,7.3) {$p_2(U_2)$};
\draw [->] (4.5,7) -- (3.0,6.2);

\node at (4.8,3.1) {$p_2(U_1)$};
\draw [->] (4.5,3.3) -- (3.0,4.06);

\node at (2.6,6.6) {$U_2$};
\node at (2.2,4) {$U_1$};

\node at (9,5.5) {$U_1$};
\draw [->] (9,5.33) -- (9.15,4.6);

\node at (9.5,3.5) {$p_2(U_1)$};
\draw [->] (9.5,3.7) -- (9.25,4.4);
\end{tikzpicture}
\caption{Left: The dynamical plane of a parameter on the co-root arc $\cC_2$ of a tongue of period $2$. The only dynamical co-root on the boundary of $U_1$ is $p_2(U_1)$ (which is also the characteristic parabolic point for parameters on $\cC_2$), and it is invisible. The dynamical roots $p_1(U_1)$ and $p_3(U_1)$ on $\partial U_1$ are visible. Right: The dynamical plane of a parameter on a co-root arc $\cC_2$ of a bounded Tricorn component. For this parameter, each dynamical co-root $p_k(U_1)$ is invisible. In particular, the characteristic parabolic point $p_2(U_1)$ is invisible (marked in black).}
\label{invisible_dynamical_antipode}
\end{figure}

\begin{definition}[Root and Co-root Arcs]\label{root_co_root}
A parabolic arc $\cC_k$ on $\partial H$ is called a root arc if for any parameter $q$ on $\cC_k$, the characteristic parabolic point $p_k(U_1^q)$ of $f_q$ is a root point of the characteristic Fatou component. Otherwise, $\cC_k$ is called a co-root arc.
\end{definition}

\begin{remark}
The above definition makes sense because all parameters on a given parabolic arc have quasi-conformally conjugate dynamics.
\end{remark}

Following \cite[\S 6]{BBM3}, we now describe the principal differences between tongues and bounded Tricorn components. To ease notations, we will now drop the superscript $`q\textrm'$ from $U_j^q$.

\textbf{Tongues}: For a tongue component of period $2n$, the fixed points of $\eta\circ f_q^{\circ n}$ on the boundaries of $U_i$ can be numbered such that $p_3(U_j) = p_1(U_{j+1})$, for $j\in\Z_{2n}$. Topologically speaking, the component $U_j$ only touches its neighboring components $U_{j-1}$ and $U_{j+1}$, and hence, $U_1,\cdots, U_{2n}$ form a ring separating $0$ from $\infty$. In particular, $p_1(U_j)$ and $p_3(U_j)$ are dynamical root points, and $p_2(U_j)$ is a dynamical co-root (see Figure \ref{invisible_dynamical_antipode} (left)). It follows that $\cC_1$ and $\cC_3$ are root arcs and $\cC_2$ is a co-root arc on $\partial H$. The parabolic arc $\cC_2$ is bounded, and it meets $\cC_1$ (respectively, $\cC_3$) at a cusp parameter such that in the corresponding dynamical plane, the attracting periodic point in $U_j$ merges with the repelling periodic points $p_1(U_j)$ and $p_2(U_j)$ (respectively, with the repelling periodic points $p_3(U_j)$ and $p_2(U_j)$). On the other hand, if the parabolic arcs $\cC_1$ and $\cC_3$ were to meet in the finite parameter plane, then the attracting periodic point in $U_j$ would merge with both $p_1(U_j)$ and $p_3(U_j)$. However, $p_1(U_j)$ and $p_3(U_j)$ are distinct points of a common self-antipodal period $2n$ orbit; thus, if $p_1(U_j)$ and $p_3(U_j)$ merged, then this entire orbit would collapse to a self-antipodal fixed point, which is impossible. It follows from this discussion and the description of the boundary of a Tricorn component given in Lemma~\ref{LemTricornIndiffDyn_1} that $\cC_1$ and $\cC_3$ are unbounded arcs. In particular, every tongue component is unbounded.

\textbf{Bounded Tricorn Components}: For a bounded Tricorn component of period $2n$, all the repelling cycles $\{p_k(U_i)\}$ are disjoint. In particular, each $p_k(U_1)$ is a dynamical co-root (see Figure \ref{invisible_dynamical_antipode} (right)). It follows that each parabolic arc $\cC_k$ on $\partial H$ is a co-root arc. Moreover, each $\cC_k$ is bounded.

We will denote the immediate basins of attraction of $0$ and $\infty$ by $\mathcal{B}^{\mathrm{imm}}_{0}$ and $\mathcal{B}^{\mathrm{imm}}_{\infty}$ respectively. The following Lemma will allow us to define a notion of visibility for dynamical co-roots (analogous to the corresponding notion for the family $\mathcal{N}_4^*$). In fact, we will use the antipodal symmetry of $\mathcal{B}^{\mathrm{imm}}_{0}$ and $\mathcal{B}^{\mathrm{imm}}_{\infty}$ to obtain a slightly stronger characterization of visibility for the family $\mathcal{A}_3$ than what was proved for the family $\mathcal{N}_4^*$ in Lemma \ref{visible}.

\begin{lemma}\label{visible_1}
Let $H$ be a tongue or a bounded Tricorn component of period $2n$, and $q$ be a parameter in $H$ or a simple parabolic parameter on $\partial H$. Let $U_1$ be the characteristic Fatou component of $f_{q}$ and $p_k(U_1)$ (for some $k\in\{1, 2, 3\}$) be a dynamical co-root on the boundary of $U_1$. The following statements about $p_k(U_1)$ are equivalent.
\begin{enumerate}
\item $p_k(U_1)\in\partial\mathcal{B}^{\mathrm{imm}}_{0}\cap\partial\mathcal{B}^{\mathrm{imm}}_{\infty}$.

\item $p_k(U_1)\in\partial\mathcal{B}^{\mathrm{imm}}_{0}\cup\partial\mathcal{B}^{\mathrm{imm}}_{\infty}$.

\item $p_k(U_1)$ lies on the boundary of a Fatou component other than $U_1$.

\end{enumerate}
\end{lemma}

\begin{proof}
The proof of equivalence of conditions (2) and (3) is similar to that of Lemma \ref{visible}.

Condition (1) trivially implies condition (2). Hence, it suffices to show that condition (2) implies condition (1).

To this end, let us first suppose that $p_k(U_1)$ lies on $\partial \mathcal{B}^\mathrm{imm}_0$. Since $p_k(U_1)$ is a fixed point of $\eta\circ f_q^{\circ n}$, it follows that $p_k(U_1)$ also lies on $(\eta\circ f_q^{\circ n})(\partial \mathcal{B}^\mathrm{imm}_0)$. Since $f_q$ commutes with $\eta$, we have that $(\eta\circ f_q^{\circ n})(\partial \mathcal{B}^\mathrm{imm}_0)=\partial\mathcal{B}^\mathrm{imm}_\infty$. But then, $p_k(U_1)\in\partial\mathcal{B}^{\mathrm{imm}}_{0}\cap\partial\mathcal{B}^{\mathrm{imm}}_{\infty}$.

A completely analogous argument proves that if $p_k(U_1)$ lies on $\partial \mathcal{B}^\mathrm{imm}_\infty$, then $p_k(U_1)\in\partial\mathcal{B}^{\mathrm{imm}}_{0}\cap\partial\mathcal{B}^{\mathrm{imm}}_{\infty}$.

This completes the proof.
\end{proof}

\begin{definition}[Visibility of Co-roots]\label{dynamical_visibility_1}
Let $q$ be a parameter in $H$ or a simple parabolic parameter on $\partial H$. A dynamical co-root of $f_q$ on the boundary of $U_1$ is called \emph{visible} if it satisfies the equivalent conditions of Lemma \ref{visible_1}. Otherwise, it is called \emph{invisible}.
\end{definition}

We will now prove invisibility of parabolic cycles for parameters on the bounded arc of a tongue.

\begin{lemma}[Invisibility of Parabolic Points on Bounded Arcs of Tongues]\label{bi_invisible}
Let $H$ be a tongue, and $\cC_2$ be the co-root (i.e. bounded) arc on $\partial H$. Then, for any $q\in \cC_2$, the characteristic parabolic point $p_2(U_1)$ of $f_q$ is invisible.
\end{lemma}
\begin{proof}
If $q$ lies on the unique co-root arc on the boundary of a tongue, then $\partial\mathcal{B}^{\mathrm{imm}}_{0}\cap\partial\mathcal{B}^{\mathrm{imm}}_{\infty}$ consists of the root cycle $\{p_1(U_i)\}_i=\{p_3(U_i)\}_i$ (see \cite[Remark~6.13]{BBM3} and \cite[Theorem~6.8]{BBM1}). In particular, the characteristic parabolic point $p_2(U_1)$ (which is a dynamical co-root) does not lie on $\partial\mathcal{B}^{\mathrm{imm}}_{0}\cap\partial\mathcal{B}^{\mathrm{imm}}_{\infty}$. Therefore, $p_2(U_1)$ is invisible (compare Figure \ref{invisible_dynamical_antipode} (left)).
\end{proof}

\bigskip

\section{Invisible tricorns and bare regions}\label{inv_tri_bare_arc}

Throughout this section, $H$ will stand for a tongue of period $2n$, and $\cC_2$ will stand for the unique co-root (i.e. bounded) arc on $\partial H$ (unless mentioned otherwise). We will denote the critical Ecalle height parametrization of $\cC_2$ by $q:\R\to\cC_2$.

Let $V_{q(h)}^{\mathrm{out}}$ be a repelling petal at the characteristic parabolic point $p_2(U_1)$ of $f_{q(h)}$. The projection of $U_1$ into the repelling Ecalle cylinder (of $f_{q(h)}$ at the characteristic parabolic point $p_2(U_1)$) consists of two one-sided infinite cylinders, let us call them $U_1^+$ and $U_1^-$. The projection of $\partial U_1$ to the same cylinder consists of two disjoint Jordan curves $\partial U_1^+$ and $\partial U_1^-$. Note that these two Jordan curves are related by the map $z\mapsto\overline{z}+1/2$. Hence we can assume that the interval of Ecalle heights traversed by $\partial U_1^+$ (respectively by $\partial U_1^-$) is $\left[l_h,u_h\right]$ (respectively $\left[-u_h,-l_h\right]$), where $u_h>l_h$ (since $\partial U_1^+$ is not an analytic curve, its projection to the repelling Ecalle cylinder is not a geodesic, hence not a round circle; therefore, the projection must traverse a positive interval of Ecalle heights). Therefore, $U_1^+$ contains $\mathbb{R}/\mathbb{Z}\times\left(u_h,+\infty\right)$, and $U_1^-$ contains $\mathbb{R}/\mathbb{Z}\times\left(-\infty,-u_h\right)$. Clearly, $u_h>0$ for all $h$ in $\R$.

By Lemma \ref{bi_invisible}, $p_2(U_1)$ is invisible. Thus, we are in the situation of Lemma \ref{julia_path}. Hence, there is a path in $\widetilde{J}$ (which is the projection of $J(f_{q(h)})$ into the repelling cylinder at $p_2(U_1)$) connecting a point of $\partial U_1^+$ at height $u_h$ and a point of $\partial U_1^-$ at height $-u_h$.

In order to define visibility of parabolic arcs, we need the following result.

\begin{proposition}[Bifurcation along Arcs]\label{bifurcation_arcs_1}
Every co-root parabolic arc (on the boundary of a tongue or a bounded Tricorn) has, at both ends, an interval of positive length across which bifurcation from a Tricorn component of period $2n$ to a Mandelbrot component of period $2n$ occurs.
\end{proposition}
\begin{proof}
The proof is similar to the proof of \cite[Theorem 3.8, Corollary 3.9]{HS}.
\end{proof}

\begin{remark}\label{restrict_co_root_bif}
The restriction to co-root arcs is essential here. Indeed, the proof of Proposition \ref{bifurcation_arcs_1} uses the existence of parabolic cusps (double parabolic parameters) at the ends of a parabolic arc. However, a root arc on the boundary of a tongue component is unbounded, and has a finite parabolic cusp only on one end.
\end{remark}

According to \cite[Theorem 7.3]{HS}, if $q(h)$ is such a bifurcating parameter, then either $h\geq u_h$, or $h\leq -u_h$. If the critical Ecalle height $0$ parameter $q(0)$ is a bifurcating parameter, then $u_0\leq 0$, which is impossible. Hence there is an interval $(-\epsilon,\epsilon)$ of Ecalle heights such that no bifurcation to a Mandelbrot component occurs across the sub-arc $q((-\epsilon,\epsilon))$ of $\cC_2$. It follows that there exist $h_2>0>h_1$ such that $q(h_1,h_2)$ is the maximal sub-arc of $\cC_2$ across which bifurcation to Mandelbrot components does not occur (compare \cite[\S 7]{HS}). We can now define visibility of parabolic arcs and Tricorn components for the family $\mathcal{A}_3$ following the corresponding definition for $\mathcal{N}_4^*$ (see Definition \ref{parameter_visibility}).

We are now ready to prove the key lemmas leading to Theorem \ref{inaccessible_arcs_1}.

For $h\in\left(h_1,h_2\right)$, there is no bifurcation to a period $2n$ Mandelbrot component across $f_{q(h)}$. This implies that the critical Ecalle height $h$ of $f_{q(h)}$ lies in $\left(-u_h,u_h\right)$.

\begin{lemma}[Bounded Arcs of Tongues are Invisible]\label{capture_components_accumulate_3}
Let $H$ be a tongue, and $\cC_2$ be the bounded arc on $\partial H$. Then $\cC_2$ is invisible.
\end{lemma}
\begin{proof}
Let $h\in(h_1,h_2)$. Recall that there is a curve $\gamma$ in the Julia set of $f_{q(h)}$ connecting a point of $\partial U_1^+$ at height $u_h$ and a point of $\partial U_1^-$ at height $-u_h$. Moreover, the iterated pre-images of $0$ are dense on $\gamma$. One can now mimic the proof of Lemma \ref{capture_components_accumulate_2} to show that for every open set $V$ satisfying
\begin{enumerate}
\item $V\cap \overline{H}=\emptyset$, and

\item $\partial V\cap q\left(h_1, h_2\right)\neq\emptyset$,
\end{enumerate}
there exist infinitely many parameters $q$ in $V$ such that in the dynamical plane of $f_q$, the forward orbit of $c_0(q)$ lands on an iterated pre-image of $0$. But this means that $V$ contains centers of infinitely many distinct capture components. Therefore, $q(h)$ cannot lie on the boundary of any hyperbolic component other than $H$.

This proves that $\cC_2$ is invisible.
\end{proof}

Recall that the ``non-bifurcating" sub-arc of a parabolic arc $\cC$ stands for the part of $\cC$ across which bifurcation to Mandelbrot components does not occur.

\begin{corollary}\label{capture_components_accumulate_4}
Let $H$ be a tongue, and $\cC_2$ be the bounded arc on $\partial H$. Then any neighborhood of the non-bifurcating sub-arc of $\cC_2$ intersects infinitely many distinct capture components.
\end{corollary}

\begin{remark}\label{root_visible}
By \cite[Remark~6.13]{BBM3}, for a parameter in a tongue component, the dynamical root points of $U_1$ lie on $\partial\mathcal{B}^{\mathrm{imm}}_{0}\cap\partial\mathcal{B}^{\mathrm{imm}}_{\infty}$. The parabolic implosion techniques of this paper can be easily applied to this situation to show that the root (i.e. unbounded) arcs on the boundary of a tongue always intersect the boundary of the principal hyperbolic component $\mathcal{H}_0$.
\end{remark}

\begin{lemma}\label{Tricorn_everywhere}
Let $H$ be a tongue, and $\cC_2$ be the bounded arc on $\partial H$. Let $q(h)$ be a parameter on the non-bifurcating sub-arc of $\cC_2$. Then $q(h)$ is a limit point of the centers of bounded Tricorn components $\{H_k\}_{k=1}^\infty$.
\end{lemma}
\begin{proof}
Pick a parameter $q(h)$ on the non-bifurcating sub-arc of $\cC_2$; i.e. $h\in(h_1,h_2)$. Let $\gamma$ be a curve in the Julia set of $f_{q(h)}$ connecting a point of $\partial U_1^+$ at height $u_h$ and a point of $\partial U_1^-$ at height $-u_h$. Note that the set of iterated pre-images of $c_{\infty}(q(h))$ are dense on the curve $\gamma$. Now the proof of Lemma \ref{Tricorns_accumulate} shows mutatis mutandis that there are infinitely many parameters $q$, arbitrarily close to $q(h)$, such that some forward image of $c_0(q)$ is an iterated pre-image of $c_{\infty}(q)$. Hence for each such parameter $q$, there is some $r\in\N$ (depending on $q$) with $f_{q}^{\circ r}(c_0(q))=c_{\infty}(q)$. Since $f_{q}$ commutes with $\eta$, it follows that $f_{q}^{\circ r}(c_{\infty}(q))=c_0(q)$. Therefore, such a parameter $q$ has a self-antipodal super-attracting cycle; i.e. $q$ is the center of a Tricorn component. This proves that $q(h)$ is a limit point of the centers of Tricorn components $\{H_k\}_{k=1}^\infty$.

We can now argue as in the proof of Lemma~\ref{Tricorns_accumulate} to conclude that if $q_k\in \overline{H_k}$ for each $k$, then all limit points of the sequence $\{q_k\}$ belong to $\cC_2$. Since tongue components are unbounded and $\cC_2$ is a bounded arc, it follows that at most finitely many $H_k$ can be tongue components. Hence, $q(h)$ is a limit point of the centers of bounded Tricorn components.
\end{proof}

\begin{remark}\label{period_increases}
As the parameter $q$ approaches $H$ from the exterior, the ``escape time'' of the critical point $c_0(q)$ (i.e. the number of iterates taken by $c_0(q)$ to pass through the gate and escape to the outgoing domain $V_q^{\mathrm{out}}$) tends to $+\infty$. It follows that the periods of the bounded Tricorn components $H_k$ constructed in Lemma \ref{Tricorn_everywhere} tend to $+\infty$ as $k$ increases.
\end{remark}

Before proceeding to the proof of Theorem \ref{inaccessible_arcs_1}, we need to discuss some properties of the set of ``bi-visible'' points for maps outside $\mathcal{H}_0$.

We continue to work with a tongue of period $2n$ with bounded parabolic arc $\cC_2$. By Lemma \ref{Tricorn_everywhere}, there exists a sequence of bounded Tricorn components $\{H_k\}$ accumulating on $\cC_2$. Let $q$ be a parameter in $H_k$ or on a parabolic arc on $\partial H_k$. Then $f_q$ has no Herman ring. Hence by \cite[Theorem 2.1]{BBM1}, the Julia set $J(f_q)$ is connected. Since such a map $f_q$ is geometrically finite, it follows that $J(f_q)$ is also locally connected \cite[Theorem A]{Tan}. By \cite[Corollary 6.5]{BBM1}, every point of $\partial \mathcal{B}^{\mathrm{imm}}_0\cap\partial \mathcal{B}^{\mathrm{imm}}_\infty$ in the dynamical plane of $f_q$ is the common landing point of a dynamical ray in $\mathcal{B}^{\mathrm{imm}}_0$ and a dynamical ray in $\mathcal{B}^{\mathrm{imm}}_\infty$. In the language of \cite{BBM1}, $\partial \mathcal{B}^{\mathrm{imm}}_0\cap\partial \mathcal{B}^{\mathrm{imm}}_\infty$ is precisely the set of bi-visible points for the maps $f_q$ under consideration. Now, according to \cite[Theorem 6.8]{BBM1}, if the set of bi-visible points $\partial \mathcal{B}^{\mathrm{imm}}_0\cap\partial \mathcal{B}^{\mathrm{imm}}_\infty$ (of $f_q$) contains a periodic point, then it consists of a single cycle of even period.

Recall that for parameters on $\cC_2$, the intersection $\partial \mathcal{B}^{\mathrm{imm}}_0\cap\partial \mathcal{B}^{\mathrm{imm}}_\infty$ of the boundaries of the immediate basins of $0$ and $\infty$ consists precisely of a repelling cycle of period $2n$. Under small perturbation, this repelling cycle can be real-analytically followed and it continues to lie on the intersection of the boundaries of the immediate basins. Hence, if $H_k$ is sufficiently close to $\cC_2$ (i.e. if $k$ is sufficiently large), and $q$ lies in $H_k$ or on a parabolic arc on $\partial H_k$, then $\partial \mathcal{B}^{\mathrm{imm}}_0\cap\partial \mathcal{B}^{\mathrm{imm}}_\infty$ consists of a single $2n$-periodic orbit of $f_q$.

\begin{proof}[Proof of Theorem \ref{inaccessible_arcs_1}]
The first statement follows from Lemma \ref{capture_components_accumulate_3}.

The statement about accumulating capture components follows from Corollary~\ref{capture_components_accumulate_4}.

For the final statement about invisible bounded Tricorn components, consider a tongue $H$ of period $2n$ and a parameter $q(h)$ on the non-bifurcating sub-arc of the bounded arc $\cC_2$ on $\partial H$. By Lemma \ref{Tricorn_everywhere}, every neighborhood of $q(h)$ intersects infinitely many bounded Tricorn components $\{H_k\}_{k=1}^\infty$. Moreover, if the period of $H_k$ is $2r_k$, then $2r_k\to+\infty$ as $k\to+\infty$. Now note that the parabolic cycle of any simple parabolic parameter on $\partial H_k$ has length $2r_k$. Since $2r_k>2n$ for all large $k$, it follows from Lemma \ref{visible_1} and the discussion preceding the proof of this theorem that the characteristic parabolic point of every simple parabolic parameter on $\partial H_k$ is invisible (for $k$ large enough). One can now repeat the arguments of the proof of Lemma \ref{capture_components_accumulate_3} to conclude that all parabolic arcs on $\partial H_k$ are invisible; i.e. $H_k$ is an invisible Tricorn component (for $k$ large enough). This proves that every neighborhood of a non-bifurcating parameter on $\cC_2$ intersects infinitely many invisible bounded Tricorn components.
\end{proof}

We end with a complementary result about the unbounded arcs on the boundaries of tongue components. We show that unlike the bounded arcs, the unbounded arcs on the boundaries of tongues of small period contain sub-arcs that are accessible from the principal hyperbolic component $\mathcal{H}_0$. This was numerically observed in \cite[Remark~6.12]{BBM3}.  A similar phenomenon can be observed in the actual Tricorn, which was proved in \cite[Theorem 1.2]{IM1}. We employ the same techniques here, but the existence of capture components in $\mathcal{A}_3$ add some subtleties to the situation.

\begin{proposition}[Bare Regions for Low Period Tricorns]\label{open_beach}
Let $H$ be a tongue component of period at most $8$. Then any unbounded (i.e. root) parabolic arc $\cC$ on $\partial H$ contains a sub-arc of points that are accessible from the principal hyperbolic component $\mathcal{H}_0$; i.e. there exists $h$ in $\R$, and an open neighborhood $U$ of $q(h)\in\cC$ such that $U\setminus\overline{H}$ is contained in $\mathcal{H}_0$.
\end{proposition}
\begin{proof}
In the dynamical plane of a parameter $q(h)$ on $\cC$ (of even parabolic period $2n$), the projection of the immediate basin of zero $\mathcal{B}_0^{\mathrm{imm}}$ (respectively the immediate basin of infinity $\mathcal{B}_{\infty}^{\mathrm{imm}}$) into the repelling Ecalle cylinder is an annulus of modulus $\frac{\pi}{2n\ln2}$ (see the proof of \cite[Theorem B]{BE} for a computation of the modulus). For $n\leq 4$, this modulus is greater than $1/2$; i.e. the corresponding annulus is not too thin. It is well-known (see \cite[Theorem I]{BDH}, for instance) that such a conformal annulus contains a round annulus centered at the origin. In other words, there is a non-degenerate interval $(a_h,b_h)$ of outgoing Ecalle heights such that in the repelling Ecalle cylinder, the round cylinder $(a_h,b_h)\times\R/\Z$ (respectively $(-b_h,-a_h)\times\R/\Z$) is contained in the projection of $\mathcal{B}_0^{\mathrm{imm}}$ (respectively in the projection of $\mathcal{B}_{\infty}^{\mathrm{imm}}$).

We can assume that as $h\to+\infty$, $q(h)$ limits at a finite cusp point, and as $h\to-\infty$, $q(h)$ goes to an ``ideal'' cusp point at infinity. Hence, perturbing a parameter $q(h)$ with sufficiently large positive $h$ outside $\overline{H}$ makes its critical point $c_0$ escape through the gate to a period $2n$ attracting Fatou component at positive heights. It follows that for sufficiently large positive $h$, we have that $h>b_h>a_h>0$. On the other hand, for negative $h$, we have $h<0<a_h$. Since critical Ecalle heights and Fatou coordinates depend continuously on the parameter, we deduce that there is an $h_0$ such that $h_0\in(a_{h_0},b_{h_0})$.

Let us fix an $\epsilon>0$ such that $h_0\in(a_{h_0}+2\epsilon,b_{h_0}-2\epsilon)$. It is known that the basin of zero $\mathcal{B}_0$ can not get too small when $q(h_0)$ is perturbed a little bit (compare \cite[Theorem~6.1(a)]{D2}). Since the critical point $c_0(q)$ and the Fatou coordinates depend continuously on the parameter, we can choose a small neighborhood $U$ of $q(h_0)$ such that for all $q\in U\setminus\overline{H}$, the round cylinder $(a_{h_0}+\epsilon/2,b_{h_0}-\epsilon/2)\times\R/\Z$ is contained in projection of $\mathcal{B}_0$ into the repelling Ecalle cylinder (note that in the outgoing cylinder of $q(h_0)$, the round cylinder $(a_{h_0},b_{h_0})\times\R/\Z$ is contained in the projection of $\mathcal{B}_0$), and such that the critical point $c_0(q)$ has incoming Ecalle height in $(a_{h_0}+\epsilon,b_{h_0}-\epsilon)$.

We claim that $U\setminus\overline{H}$ is contained in the principal hyperbolic component $\mathcal{H}_0$. Let $q\in U\setminus\overline{H}$. To finish the proof, we only need to show that $c_0(q)$ lies in $\mathcal{B}_0^{\mathrm{imm}}$. Let us denote the set of all points in $V_q^{\mathrm{in}}$ with incoming Ecalle heights in the interval $(a_{h_0}+\epsilon,b_{h_0}-\epsilon)$ by $S_\epsilon$. In particular, $f_q^{\circ 2n}(c_0(q))\in S_\epsilon$. All points in $S_\epsilon$ eventually escape through the gate, and their Ecalle heights are preserved by the transit map. Since all points in $V^{\mathrm{out}}_q$ with outgoing Ecalle height in $(a_{h_0}+\epsilon/2,b_{h_0}-\epsilon/2)$ lie in $\mathcal{B}_0$, it follows that every point of $S_\epsilon$ eventually lands in $\mathcal{B}_0$.\footnote{In the polynomial case, this already completes the proof since the basin of infinity of a polynomial is connected. However, in the present setting, $\mathcal{B}_0$ is not necessarily connected. Hence, we need to rule out the possibility that $c_0(q)$ lies in a strictly pre-periodic Fatou component of $\mathcal{B}_0$.}  As $S_\epsilon$ is a connected open set, it must be contained in a single component of $\mathcal{B}_0$. However, as $q$ is very close to $q(h_0)$, the critical point $c_0(q)$ takes a large number of iterates to escape from $V_q^{\mathrm{in}}$. It follows that $f_q^{\circ 2n}\left(S_\epsilon\right)\cap S_\epsilon\neq\emptyset$. Therefore, $S_\epsilon$ cannot be contained in a strictly pre-periodic component of $\mathcal{B}_0$; i.e. $c_0(q)\in S_\epsilon\subset\mathcal{B}_0^{\mathrm{imm}}$.

To summarize, we have shown that for all $q\in U\setminus\overline{H}$, $c_0(q)$ is contained in the corresponding immediate basin of zero. Therefore, $U\setminus\overline{H}\subset\mathcal{H}_0$ (see Figure \ref{inaccessible_2}).
\end{proof}

\begin{remark}
Numerical experiments suggest that there are tongue components (of large period) whose root parabolic arcs do not contain such bare regions.
\end{remark}

\noindent\textbf{Acknowledgments.} The authors would like to acknowledge the support of the Institute for Mathematical Sciences at Stony Brook University during part of the work on this project. The second author was supported by an endowment from Infosys Foundation.


\begin{thebibliography}{BMMS17}

\bibitem[BBM]{BBM3}
A.~Bonifant, X.~Buff, and J.~Milnor.
\newblock On antipode preserving cubic maps (draft of {F}eb. 15, 2015).
\newblock \url{http://www.math.stonybrook.edu/~jack/bbm.pdf}.

\bibitem[BBM18]{BBM1}
A.~Bonifant, X.~Buff, and J.~Milnor.
\newblock Antipode preserving cubic maps: the {F}jord theorem.
\newblock {\em Proc. London Math. Soc.}, 116:670--728, 2018.

\bibitem[BCT14]{BCT}
X.~Buff, G.~Cui, and L.~Tan.
\newblock Teichm{\"u}ller spaces and holomorphic dynamics.
\newblock In Athanase Papadopoulos, editor, {\em Handbook of Teichm{\"u}ller
  theory}, volume~IV, pages 717--756. Soci{\'e}t{\'e} math{\'e}matique
  europ{\'e}enne, 2014.

\bibitem[BDH04]{BDH}
G.~Ble, A.~Douady, and C.~Henriksen.
\newblock Round annuli.
\newblock {\em Contemporary Mathematics: In the Tradition of Ahlfors and Bers,
  {III}}, 355:71--76, 2004.

\bibitem[BE02]{BE}
X.~Buff and A.~L. Epstein.
\newblock A parabolic {P}ommerenke-{L}evin-{Y}occoz inequality.
\newblock {\em Fund. Math.}, 172:249--289, 2002.

\bibitem[BF14]{BF14}
B.~Branner and N.~Fagella.
\newblock {\em Quasiconformal surgery in holomorphic dynamics (with
  contributions by Xavier Buff, Shaun Bullett, Adam L. Epstein, Peter
  Ha\"issinsky, Christian Henriksen, Carsten L. Petersen, Kevin M. Pilgrim, Tan
  Lei and Michael Yampolsky)}, volume 141 of {\em Cambridge Studies in Advanced
  Mathematics}.
\newblock Cambridge University Press, Cambridge, 2014.

\bibitem[BMMS17]{BMMS}
K.~Bogdanov, K.~Mamayusupov, S.~Mukherjee, and D.~Schleicher.
\newblock Antiholomorphic perturbations of {W}eierstrass {Z}eta functions and
  {G}reen's function on tori.
\newblock {\em Nonlinearity}, 30(8):3241--3254, 2017.

\bibitem[CFG15]{CFG}
J.~Canela, N.~Fagella, and A.~Garijo.
\newblock On a family of rational perturbations of the doubling map.
\newblock {\em Journal of Difference Equations and Applications}, 21:715--741,
  2015.

\bibitem[DH93]{DH93}
A.~Douady and J.~Hubbard.
\newblock A proof of {T}hurston's topological characterization of rational
  functions.
\newblock {\em Acta Math.}, 171:263--297, 1993.

\bibitem[Dou94]{D2}
A.~Douady.
\newblock Does a {J}ulia set depend continuously on the polynomial?
\newblock {\em Complex dynamical systems, Proc. Sympos. Appl. Math.},
  49:91--138, 1994.

\bibitem[HS14]{HS}
J.~H. Hubbard and D.~Schleicher.
\newblock {M}ulticorns are not path connected.
\newblock In {\em Frontiers in Complex Dynamics: In Celebration of John
  Milnor's 80th Birthday}, pages 73--102. Princeton University Press, 2014.

\bibitem[HSS01]{HSS}
J.~Hubbard, D.~Schleicher, and S.~Sutherland.
\newblock How to find all roots of complex polynomials by {N}ewton's method.
\newblock {\em Inventiones Mathematicae}, 146:1--33, 2001.

\bibitem[IM16a]{IM2}
H.~Inou and S.~Mukherjee.
\newblock Discontinuity of straightening in antiholomorphic dynamics.
\newblock \url{https://arxiv.org/abs/1605.08061}, 2016.

\bibitem[IM16b]{IM1}
H.~Inou and S.~Mukherjee.
\newblock Non-landing parameter rays of the multicorns.
\newblock {\em Inventiones Mathematicae}, 204:869--893, 2016.

\bibitem[Lei00]{Lei1}
Tan Lei.
\newblock Local properties of the {M}andelbrot set at parabolic points.
\newblock In Tan Lei, editor, {\em The Mandelbrot Set, Theme and Variations},
  number 274 in Lecture Note Series, pages 130--160. London Mathematical
  Society, Cambridge, 2000.

\bibitem[LMSa]{LMS2}
R.~Lodge, Y.~Mikulich, and D.~Schleicher.
\newblock A classification of postcritically finite {N}ewton maps.
\newblock \url{https://arxiv.org/abs/1510.02771}.

\bibitem[LMSb]{LMS1}
R.~Lodge, Y.~Mikulich, and D.~Schleicher.
\newblock Combinatorial properties of {N}ewton maps.
\newblock \url{https://arxiv.org/abs/1510.02761}.

\bibitem[Mil00]{M4}
J.~Milnor.
\newblock On rational maps with two critical points.
\newblock {\em Experiment. Math.}, 9:333--411, 2000.

\bibitem[Mil06]{M1new}
J.~Milnor.
\newblock {\em Dynamics in one complex variable}.
\newblock Princeton University Press, New Jersey, 3rd edition, 2006.

\bibitem[Mil12]{MP}
J.~Milnor.
\newblock Hyperbolic components in spaces of polynomial maps, with an appendix
  by {A}. {P}oirier.
\newblock In {\em Conformal Dynamics and Hyperbolic Geometry}, volume 573 of
  {\em Contemporary Mathematics}, pages 183--232. American Mathematical
  Society, Providence, RI, 2012.

\bibitem[MNS15]{MNS}
S.~Mukherjee, S.~Nakane, and D.~Schleicher.
\newblock On {M}ulticorns and {U}nicorns {II}: bifurcations in spaces of
  antiholomorphic polynomials.
\newblock {\em Ergodic Theory and Dynamical systems}, 37:859--899, 2015.

\bibitem[Nai83]{Na}
V.~A. Naishul.
\newblock Topological invariants of analytic and area preserving mappings and
  their applications to analytic differential equations in $\mathbb{C}^2$ and
  $\mathbb{CP}^2$.
\newblock {\em Transactions of the Moscow Mathematical Society}, 42:239--250,
  1983.

\bibitem[NS03]{NS}
S.~Nakane and D.~Schleicher.
\newblock On {M}ulticorns and {U}nicorns {I} : Antiholomorphic dynamics,
  hyperbolic components and real cubic polynomials.
\newblock {\em International Journal of Bifurcation and Chaos}, 13:2825--2844,
  2003.

\bibitem[Pil94]{KMPthesis}
K.~Pilgrim.
\newblock {\em Cylinders for iterated rational maps}.
\newblock PhD thesis, University of California, Berkeley, 1994.

\bibitem[Prz89]{Prz}
Feliks Przytycki.
\newblock Remarks on the simple connectedness of basins of sinks for iterations
  of rational maps.
\newblock In K.~Krzyzewski, editor, {\em Dynamical Systems and Ergodic Theory},
  pages 229--235. Polish Scientific Publishers, Warszawa, 1989.

\bibitem[PT98]{PT}
K.~Pilgrim and L.~Tan.
\newblock Combining rational maps and controlling obstructions.
\newblock {\em Ergodic Theory and Dynamical Systems}, 18:221--245, 1998.

\bibitem[RS07]{RS}
J.~R{\"u}ckert and D.~Schleicher.
\newblock On {N}ewton's method for entire functions.
\newblock {\em J. London Math. Soc.}, 75:659--676, 2007.

\bibitem[Shi09]{Shi1}
M.~Shishikura.
\newblock The connectivity of the {J}ulia set and fixed points.
\newblock In {\em Complex Dynamics: Families and Friends}, pages 257--276. A K
  Peters, Ltd., Massachusetts, 2009.

\bibitem[Sut89]{Su}
S.~Sutherland.
\newblock {\em Finding Roots of Complex Polynomials with {N}ewton's Method}.
\newblock PhD thesis, Boston University, 1989.
\newblock Doctoral Dissertation.

\bibitem[TY96]{Tan}
Lei Tan and Yongcheng Yin.
\newblock Local connectivity of the {J}ulia set for geometrically finite
  rational maps.
\newblock {\em Science China Mathematics}, 39(1):39--47, 1996.

\end{thebibliography}
\end{document}